\documentclass{amsart}

\usepackage[utf8]{inputenc}
\usepackage[T1]{fontenc}
\usepackage{float}
\usepackage{mathtools}
\usepackage{color}
\usepackage[english]{babel}
\usepackage{amsmath} 
\usepackage{amsfonts}
\usepackage{thmtools}
\usepackage{amssymb}
\usepackage{amsthm}
\usepackage{mathrsfs}
\usepackage{setspace}
\usepackage{graphicx}
\usepackage{bm}
\usepackage[all]{xy}
\usepackage{lscape}
\usepackage{latexsym}
\usepackage{epigraph}
\usepackage[Sonny]{fncychap}
\usepackage{tikz} 
\usepackage{tikz-cd}
\usetikzlibrary{matrix}
\usepackage{epsfig}
\usepackage{fancyhdr}
\usepackage{hyperref}
\usepackage{paralist}
\usepackage[mathcal]{eucal}
\usepackage{xcolor}
\usepackage{caption}
\usepackage{subcaption}
\usepackage{float}
\numberwithin{equation}{section}

\theoremstyle{cupthm}
\newtheorem{theorem}{Theorem}[section]
\newtheorem{proposition}[theorem]{Proposition}
\newtheorem{lemma}[theorem]{Lemma}
\newtheorem{corollary}[theorem]{Corollary}
\newtheorem{convention}[theorem]{Convention}
\newtheorem{claim}[theorem]{Claim}
\newtheorem*{theorem*}{Theorem}
\theoremstyle{definition}
\newtheorem{definition}[theorem]{Definition}
\newtheorem{example}[theorem]{Example}

\theoremstyle{remark}
\newtheorem{remark}[theorem]{Remark}

\theoremstyle{definition}
\newtheorem*{definition*}{Definition}

\newtheorem*{corollary*}{Corollary}

\usepackage{geometry}
\definecolor{darkred}{rgb}{0,0,0} 
\definecolor{darkgreen}{rgb}{0,0,0}
\definecolor{darkblue}{rgb}{0,0,0}

\newcommand{\A}{\mathbb{A}} 
\newcommand{\AO}{\mathbb{A}^{\mathcal{O}}}
\newcommand{\B}{\mathbb{B}}

\newcommand{\D}{((\mathbb{B}, u_1),\phi,  \{p_j\},  \{\gamma_j\})}
 \newcommand{\G}{\Gamma_{((\mathbb{B}, u_1),\phi,  \{p_j\},  \{\gamma_j\})}}
\newcommand{\st}{\text{St}}    
\newcommand{\rk}{\text{rk}} 
   
\newcommand{\M}{((\mathbb{B}, u_0), \varphi, \mathcal{T})}

\begin{document}

\title{Nielsen equivalence in closed $2$-orbifold groups}

\author{Ederson R. F. Dutra} \thanks{Author supported by FAPESP, São Paulo Research Foundation, grant 2018/08187-6.}
\address{Universidade Federal de São Carlos,  São Carlos, Brazil}
\email{edersondutra@dm.ufscar.br}

\begin{abstract} We prove that any generating tuple of the fundamental group of a sufficiently large $2$-dimensional orbifold is represented by an almost orbifold covering. As a corollary we obtain a generalization of Louder's Theorem \cite{Louder} which asserts that any two generating tuples of the fundamental group of a closed surface are Nielsen equivalent.
\end{abstract}


\maketitle
\section{Introduction}

If $G$ is a group, $n\geqslant 1$ and $\mathcal{T}=(g_1,\ldots, g_n)$ is a $n$-tuple of elements of $G$,  an \textit{elementary Nielsen transformation} on $\mathcal{T}$ is one of the following three types of moves:
\begin{enumerate}
\item[(T1)] For some $i\in \{1,\ldots, n\}$ replace $g_i$ by $g_i^{-1}$ in $\mathcal{T}$.
\item[(T2)] For some $i\neq j$, $i,j\in \{1, \ldots, n\}$ interchange $g_i$ and $g_j$ in $\mathcal{T}$. 
\item[(T3)] For some $i\neq j$, $i,j\in \{1, \ldots, n\}$ replace $g_i$ by $g_ig_j$ in $\mathcal{T}$.
\end{enumerate}
Two $n$-tuples $\mathcal{T}$ and $\mathcal{T}'$ of elements of $G$ are called	 \textit{Nielsen equivalent}, denoted by  $\mathcal{T}\sim_{NE} \mathcal{T}'$, if there exists a finite chain of elementary Nielsen transformations taking $\mathcal{T}$ to $\mathcal{T}'$.  Nielsen equivalence clearly defines an equivalence relation on the set  of $n$-tuples of  elements of  $G$ for every $n\geqslant 1$.

Recall that, for a finitely generated group $G$, the \emph{rank} of $G$, denoted by $\rk(G)$, is the smallest number of elements in a generating set of $G$.  A generating $n$-tuple is called \emph{minimal} if $n=\rk(G)$.

An  $n$-tuple $\mathcal{T}$  is called  \emph{reducible} if $\mathcal{T}$ is Nielsen equivalent to a  tuple of  type $(g_1',\ldots, g_{n-1}',1)$. Otherwise,  we say that $\mathcal{T}$ is  \emph{irreducible}.  Note that any minimal generating tuple  is  irreducible.

J. Nielsen~\cite{N1,N2,N3} showed that for the free group $F_n=F(x_1,\ldots, x_n)$ any generating  tuple  of $F_n$ is Nielsen equivalent to $(x_1,\ldots, x_n, 1, \ldots, 1)$. This result implies in particular that in a  free group all generating tuples of the same size are Nielsen equivalent.

Nielsen's result was later generalized to free products by I. A. Grushko~\cite{Grushko}. Note that Nielsen's result can be easily derived from Grushko's theorem and the euclidean algorithm, which is nothing but Nielsen's method in the case of a free group of rank $1$.

\begin{theorem*}[Grushko]{\label{thm:Grushko}}
Let $G=H\ast K$ be a free product. Then any generating tuple $\mathcal{T}=(g_1,\ldots,g_n)$ of   $G$ is Nielsen equivalent to a tuple   $(h_1,\ldots,h_s,k_{s+1},\ldots,k_n)$
with $h_i\in H$ and  $k_i\in K$ for all $i$.\end{theorem*}

The purpose of  this paper is to study Nielsen equivalence of generating tuples  of    $2$-orbifold groups. A $2$-\emph{orbifold} is a space locally modeled on the quotients of the Euclidean plane by finite group actions. Further, these local models are glued together by maps compatible with the finite group action. We  shall restrict our  attention to $2$-orbifolds without reflector curves. This class of orbifolds can be defined in an entirely combinatorial way as follows.

\begin{definition}
A  $2$-\emph{orbifold}  $\mathcal{O}$ is a pair $(F, p)$ where $F$ is a  connected  surface, called the \emph{underlying surface} of $\mathcal{O}$, and $p:F\rightarrow \mathbb{N}$ is a function such that $\Sigma(\mathcal{O}):=\{x\in F \ | \ p(x)\geqslant 2\}$ is countable, discrete and is contained in the interior $ \text{Int}(F)=F-\partial F$ of $F$. The number $p(x)$ is called the \emph{order} of   $x$ and the points of $\Sigma(\mathcal{O})$   are called \emph{cone points}.  An  orbifold is said to be \emph{compact/closed} if its underlying surface is  compact/closed.   
\end{definition}

We will often denote a compact   $2$-orbifold $\mathcal{O}=(F,p)$ with cone points $x_1, \ldots, x_r$   by $\mathcal{O}=F(p_1, \ldots, p_r)$ where $p_i$ is the order of $x_i$.   To describe a  $2$-orbifold pictorially, we show the underlying surface $F$, mark the cone points, and label each cone point by its order. Figure~\ref{fig:orbifold1} illustrates a $2$-orbifod with underlying surface an annulus and with a single cone point of order $p_1$.

\smallskip 

The \emph{fundamental group} $\pi_1^{o}(\mathcal{O})$ of a   $2$-orbifold $\mathcal{O}=(F, p)$  is defined as follows.  Let $x_1, x_2, \ldots$ be the cone points of $\mathcal{O}$ and for each $i$ let $D_i\subseteq \text{Int} F$ be a disk centered at $x_i$ such that $ {D}_i\cap {D}_j= \emptyset$ for $i\neq j$. Let $S_{\mathcal{O}}$ be the surface $F - \cup\ \text{Int}(D_i)$.  Then $\pi_1^{o}(\mathcal{O})$ is the  group obtained from $\pi_1 (S_{\mathcal{O}})$ by adding the relations $s_i^{p_i}=1$, where $s_i$ is the homotopy class represented by $\partial D_i$ and $p_i$ is the order of  $x_i$.  Thus a presentation of the fundamental group of a compact  orbifold  $\mathcal{O}=F(p_1,\ldots, p_r)$ is  
$$\pi_1^o(\mathcal{O})=\langle a_1,\ldots , a_{p}, t_1, \ldots, t_q,  s_1, \ldots , s_{r} \ | \  s_1^{p_1},\ldots, s_r^{p_r},  R \cdot s_1\cdot \ldots\cdot s_r =t_1\cdot \ldots \cdot t_q\rangle,$$
where the element  $t_j$ ($1\leqslant j\leqslant q$) is the homotopy class of the $j$th boundary component of $F$. The word $R$ and the number $p\geqslant 0$  are  given by:
\begin{enumerate}
\item[(P1)] $p$ is even and $R =[a_1,a_2]\cdot \ldots \cdot [a_{p-1},a_p]$ if $F$ is orientable;
\item[(P2)]   $p\geqslant 1$ and $R =a_1^2\cdot\ldots \cdot a_p^2$ if $F$ is non-orientable.
\end{enumerate}
For a more detailed discussion  about orbifolds the reader is referred to the article of P. Scott~\cite{Scott}. 
\begin{figure}[h]
\begin{minipage}[b]{0.4\linewidth}
\centering
\includegraphics[width=0.35\linewidth]{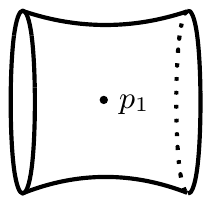}
\caption{ $\mathcal{O}=A(p_1)$. }\label{fig:orbifold1}
\end{minipage}
\begin{minipage}[b]{0.4\linewidth}
\centering
\includegraphics[width=0.35\linewidth]{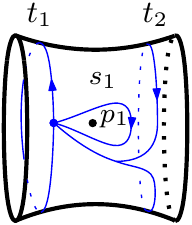}
\caption{Generators of $\pi_1^o(\mathcal{O})$. }\label{fig:genfundgp}
\end{minipage}
\end{figure}
\begin{convention}
Let $\mathcal{O}=F(p_1, \ldots, p_r)$ be a compact  connected  orbifold with cone points $x_1, \ldots, x_r$ of order $p_1, \ldots,  p_r$ respectively. Given a point $x$ of $\mathcal{O}$ of order $p\geqslant 1$, we will denote by  $s_x$     the element of  $\pi_1^o(\mathcal{O})$ represented by  the boundary of a small disk  containing  $x$ in its interior.  Thus $s_x=1$ if $p=1$ and $s_{x}=s_i$ if $x=x_i$ for some $1\leqslant i\leqslant r$. 
\end{convention}

\begin{definition*}[Standard generating tuple]
Let $\mathcal{O}$ be a   compact  orbifold. A \emph{standard generating tuple} of $\pi_1^{o}(\mathcal{O})$ is a tuple of the form  $(a_1,a_2,\ldots,a_p,t_{j_1}, \ldots,t_{j_{q'}},  s_{i_1}^{\nu_1},\ldots,s_{i_{r'}}^{\nu_{r'}})$
where  the following hold:
\begin{enumerate}
\item[(S1)] $q'+r'=q+r-1$.
\item[(S2)] $1\leqslant j_1<\ldots <j_{q'}\leqslant q$ and $1\leqslant i_1<\ldots<i_{r'}\leqslant r$.
\item[(S3)] $\nu_1,\ldots,\nu_{r'}$ are positive integers   such that $(\nu_j,p_j)=1$   for all $1\leqslant j\leqslant r'$.
\end{enumerate}
\end{definition*}

\begin{remark}{\label{remark:1}}Let $\mathcal{O}$ be a compact   $2$-orbifold.  If $\mathcal{O}$ has non-empty    boundary, then 
$$\pi_1^o(\mathcal{O})=F(a_1,\ldots, a_p, t_1, \ldots, t_{q-1})\ast \langle s_1\rangle\ast \ldots \ast \langle s_r\rangle.$$ 
Grushko's theorem and the euclidean algorithm imply  that any generating tuple of $\pi_1^o(\mathcal{O})$ is  Nielsen equivalent to a tuple 
$(a_1,\ldots, a_p, t_1,\ldots, t_{q-1}, s_{1}^{\nu_1}, \ldots, s_{r}^{\nu_r}, 1,\ldots, 1)$
with  $\nu_1,\ldots, \nu_r$ positive integers  such that  $(\nu_i, p_i)=1$ for all $1\leqslant i\leqslant r$. In particular, any minimal generating tuple of $\pi_1^o(\mathcal{O})$ is Nielsen equivalent to a standard generating tuple.
\end{remark}

In~\cite[Satz 6]{Zieschang}  H. Zieschang proves that any    minimal generating tuple of the fundamental group of a closed  surface of genus $\neq3$ is Nielsen equivalent to a standard generating tuple.  Rosenberger~\cite{Rosenberger, Rosenberger1} proves  a similar result for a large class of Fuchsian groups. In~\cite{Lustig, Lustig1} the authors distinguish Nielsen equivalence classes of standard generating tuples of Fuchsian groups and use this to distinguish  vertical  Heegaard splittings of Seifert fibered spaces. Recently L. Louder~\cite{Louder}  proves that any two generating tuples of the fundamental group of a closed surface are Nielsen equivalent.

In order to formulate the main result of this paper we need the notion of an almost orbifold covering. We first recall the definition of an orbifold covering. Let $\mathcal{O}=(F,p)$ and $\mathcal{O}'=(F', p')$ be two  $2$-orbifolds. An \emph{orbifold covering} $\eta:\mathcal{O}'\rightarrow \mathcal{O}$ is a continuous surjective map $\eta:F'\rightarrow F$ 
with the following properties:
\begin{enumerate}
\item For each point $y\in F'$, the order of $y$ divides the order of $\eta(y)$.

\item For each point $x\in \text{Int}(F)$ the set $\eta^{-1}(x)\subseteq F'$ is discrete and, over a disk in $F$ centered at $x$,   $\eta$ is equivalent to the map 
$$(z, y)\in \mathbb{D}^2 \times \eta^{-1}(x) \mapsto e^{( \frac{2\pi p'(y)}{p(x)}i ) } z\in \mathbb{D}^2$$ 
and $x$ corresponds to $0$ in $\mathbb{D}^2$.

\end{enumerate} 
Note that the map $\eta|_{F'- \eta^{-1}(\Sigma(\mathcal{O}))}:F'- \eta^{-1}(\Sigma(\mathcal{O}))\rightarrow F-\Sigma(\mathcal{O})$ is a genuine covering.  The \emph{degree}  of  $\eta:\mathcal{O}'\rightarrow\mathcal{O}$, denoted by $\text{deg}(\eta)$,    is defined as  the degree of $\eta|_{F'- \eta^{-1}(\Sigma(\mathcal{O}))}$. It is not hard to see that an orbifold covering $\eta:\mathcal{O}'\rightarrow \mathcal{O}$    induces a monomorphism   $\eta_{\ast}:\pi_1^o(\mathcal{O}')\rightarrow \pi_1^o(\mathcal{O}).$ 
Conversely, for any subgroup $H\leq \pi_1^o(\mathcal{O})$ there is an orbifold covering $\eta:\mathcal{O}_{H}\rightarrow  \mathcal{O}$ such that $\eta_{\ast}(\pi_1^o(\mathcal{O}_{H}))=H$.

\begin{definition*}[Almost orbifold covering]{\label{def:almost}}
Let $\mathcal{O}'=(F', p')$ and $\mathcal{O}=(F, p)$ be two  compact  $2$-orbifolds.  An \emph{almost orbifold covering} $\eta:\mathcal{O}'\rightarrow \mathcal{O}$ is a continuous map $\eta:F'\rightarrow F$ having the following properties:
\begin{enumerate}
\item[(C1)] For each $y\in F'$ the order  $y$ divides the order  of $\eta(y)$. 

\item[(C2)] There is a point $x\in\text{Int}(F)$ of order $p\geqslant 1$,  called the \emph{exceptional point}, and a disk ${D}\subseteq\text{Int}(F)$ centered at $x\in F$, called the \emph{exceptional disk}, with $(D-\{x\})\cap \Sigma(\mathcal{O})=\emptyset$  
such that $\eta$ restricted to $F'-\eta^{-1}(\text{Int}(D))$ defines  an orbifold covering of finite degree between the compact $2$-orbifolds  
$$\mathcal{Q}':=(F'-\eta^{-1}(\text{Int}(D)), p'|_{F'-\eta^{-1}(\text{Int}(D))})\subseteq \mathcal{O}' \ \ \text{ and } \ \ \mathcal{Q}:=(F-\text{Int}(D), p|_{F-\text{Int}(D)})\subseteq\mathcal{O}.$$

\item[(C3)] $\eta^{-1}(D)=D_1\sqcup D_2\sqcup\ldots\sqcup D_t\sqcup C$, where $t\geqslant 0$ and  
\begin{enumerate}
\item[(C3.a)]  $C\subseteq \partial F'$ is a boundary component of $\mathcal{O}'$, called the \emph{exceptional boundary component} of $\mathcal{O}'$.

\item[(C3.b)] Each $D_j\subseteq \text{Int}(F')$ ($1\leqslant j\leqslant t$) is a disk and $\eta|_{D_j}:D_j\rightarrow D$ is equivalent to the map
$$z\in \mathbb{D}^2\longmapsto e^{\frac{2\pi p}{q}i}  z\in \mathbb{D}^2$$
and  $x\in D$ corresponds to $0$ in $\mathbb{D}^2$, where  $q$ is the order of the  point $\eta^{-1}(x)\cap D_j$.  
\end{enumerate}

\end{enumerate}
The degree of $\eta:\mathcal{O}'\rightarrow \mathcal{O}$, which we also denote by $\text{deg}(\eta)$,  is defined as the degree of the orbifold covering  $\eta|_{\mathcal{Q}'}:\mathcal{Q}'\rightarrow \mathcal{Q}$.   We further call an almost orbifold covering   $\eta :\mathcal{O}'\rightarrow \mathcal{O}$   \emph{special} if the degree of the map  $\eta|_{C}:C\rightarrow \partial D$ is at most  the order $p$ of the exceptional point $x$.
\end{definition*}

An almost orbifold covering $\eta:\mathcal{O}'\rightarrow \mathcal{O}$ also induces a homomorphism $\eta_{\ast}:\pi_1^o(\mathcal{O}')\rightarrow \pi_1^o(\mathcal{O})$.  But contrary to orbifold coverings, the induced homomorphism $\eta_{\ast}$ is not injective since $\eta_{\ast}$ maps the element $g$,   represented by the exceptional boundary component  $C\subseteq \partial \mathcal{O}'$,   onto  a conjugate of  $s_x^l$,  where $x$ is the exceptional point and   $l$ is the degree of the map $\eta|_{C}:C\rightarrow \partial D$.    Another important difference from orbifold coverings is that the degree of an almost orbifold covering may be strictly larger than the index of $\eta_{\ast}(\pi_1^o(\mathcal{O}'))$ in $\pi_1^o(\mathcal{O})$; we will give examples of $\pi_1$-surjective almost orbifold coverings of degree $\geqslant 2$.

\begin{remark}{\label{remark:2}}
Note that if $\eta:\mathcal{O}'\rightarrow \mathcal{O}$ is an almost orbifold covering,  then  $\mathcal{O}'$ has  non-empty boundary, and therefore $\pi_1^o(\mathcal{O}')$  splits as a free product of cyclic groups. It follows from  Remark~\ref{remark:1}  that any generating tuple of $\pi_1^o(\mathcal{O}')$ is either reducible or  is Nielsen equivalent to a standard generating tuple.    
\end{remark}

\begin{example}{\label{ex:almost1}}
Let $p_1=2n+1$ and $p_2=2n'$ with $n, n'\geqslant 1$. Fig.~\ref{fig:almost1} shows a special  almost orbifold covering $\eta:\mathcal{O}'\rightarrow \mathcal{O}$ where $\mathcal{O}'=F'(p_1,p_2,\frac{p_2}{2})$ with $F'$ a punctured double torus,   and $\mathcal{O}=T^2(p_1, p_2)$.  The exceptional point is the cone point of $\mathcal{O}$  of order $p_1$ and the exceptional disk is  $D$. The inverse image of $D$ under $\eta$ is equal to $D_1\sqcup C$.    
\begin{figure}[h!]
\begin{center}
\includegraphics[scale=.95]{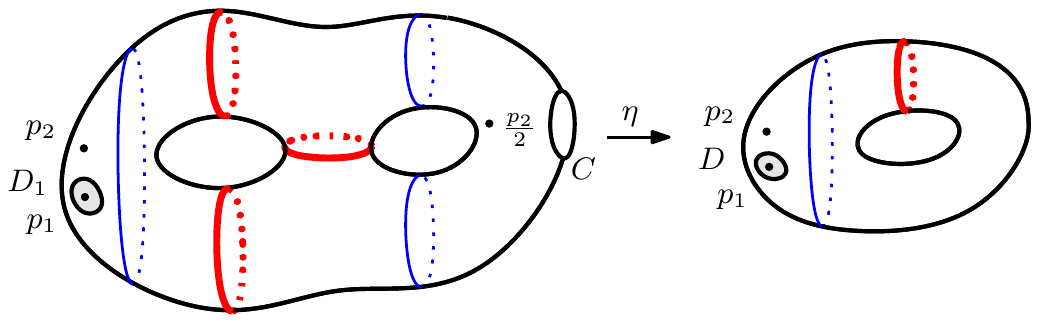}
\end{center}
\caption{$\eta:\mathcal{O}'\rightarrow \mathcal{O}$ is a special almost orbifold cover of   degree $3$.}\label{fig:almost1}
\end{figure}  
\end{example}

\begin{example}{\label{ex:almost2}}
Suppose that $\mathcal{O}=(F,p)$ and $\mathcal{O}''=(F'', p'')$ are compact $2$-orbifolds and that $\eta'':\mathcal{O}''\rightarrow \mathcal{O}$ is an orbifold covering of finite degree.   Let $x$ be a point of $F$ of order $p\geqslant1$,  $D\subseteq\text{Int}(F)$ a disk centered at $x$ with $(D-\{x\})\cap \Sigma(\mathcal{O})=\emptyset$  and let $D''\subseteq F''$ a component of $(\eta'')^{-1}(D)$. Then $\eta''$ restricted to  $F''-\text{Int}(D'')$  defines a special  almost orbifold covering $\eta:\mathcal{O}'\rightarrow \mathcal{O}$ where   
$$\mathcal{O}'=(F''-\text{Int}(D'') , p''|_{F''-\text{Int}(D'')})\subseteq \mathcal{O}''.$$ 
The exceptional boundary component of $\mathcal{O}''$   is  $\partial D''\subseteq \partial\mathcal{O}'$, the exceptional point is  $x$ and the exceptional disk is $D$. Note that the  degree of the map $\eta|_{\partial D''}:\partial D''\rightarrow \partial D$ divides $p$.
\end{example}

\begin{example}{\label{ex:almost3}}
Conversely, suppose that $\eta: \mathcal{O}'\rightarrow \mathcal{O}$ is a special  almost orbifold covering   with exceptional boundary component  $C\subseteq \partial \mathcal{O}'$,  exceptional disk $D\subseteq \text{Int}(F)$  and exceptional point $x\in \text{Int}(D)$. If the degree of the map  $\eta|_{C}:C\rightarrow \partial D$   divides the order  $p$  of $x$,  then $\eta$ is    the restriction of some orbifold covering of finite degree $\eta'':\mathcal{O}''\rightarrow \mathcal{O}$ as in the previous example.
\end{example}

In this article we study Nielsen classes of generating tuples of the fundamental group of sufficiently large $2$-orbifolds where we say that a $2$-orbifold $\mathcal{O}$ is \emph{sufficiently large} if $\mathcal{O}$ is closed and the following hold:
\begin{enumerate}
\item If the underlying surface of $\mathcal{O}$ is $S^2$,  then $\mathcal{O}$ has at least $4$ cone pints. 

\item If the underlying surface of $\mathcal{O}$ is $\mathbb{R}P^2$,  then $\mathcal{O}$ has at least $2$ cone points.  
\end{enumerate}
The main result of this paper is the following theorem.
\begin{theorem}{\label{MainThm}}
Let $\mathcal{O}$ be a sufficiently large $2$-orbifold and  $\mathcal{T}$  a  generating tuple of $\pi_1^o(\mathcal{O})$.  Then there exists a special    almost orbifold covering $\eta:\mathcal{O}'\rightarrow \mathcal{O}$ and a generating tuple $\mathcal{T}'$ of $\pi_1^o(\mathcal{O}')$ such that $\eta_{\ast}(\mathcal{T}')$ and $\mathcal{T}$ are   Nielsen equivalent. 
\end{theorem}

\begin{example}[Standard generating tuples]{\label{ex:standard}}
Let $\mathcal{O}=F(p_1,\ldots, p_r)$ be a sufficiently large $2$-orbifold and let $\mathcal{T}=(a_1,\ldots, a_p, s_{1}^{\nu_{1}},\ldots, s_{i-1}^{\nu_{i-1}}, s_{i+1}^{\nu_{i+1}}, \ldots,  s_{r}^{\nu_r})$  be a standard generating tuple of $\pi_1^o(\mathcal{O})$. A special  almost orbifold covering $\eta:\mathcal{O}'\rightarrow \mathcal{O}$ and a generating tuple $\mathcal{T}'$ of $\pi_1^o(\mathcal{O}')$ such that $\eta_{\ast}(\mathcal{T}')\sim_{NE}\mathcal{T}$ can be constructed as follows. If $r=0$, i.e. $\Sigma(\mathcal{O})=\emptyset$, then  let $x$ be an arbitrary point of $F$.  If $r\geqslant 1$ then let $x\in F$ be the cone point of $\mathcal{O}$ that corresponds to the generator $s_i$ of $\pi_1^o(\mathcal{O})$. Put   
$$\mathcal{O}'=(F-\text{Int}(D), p|_{F-\text{Int}(D)})$$
where $D\subseteq F$  is a disk centered at  $x \in F$ with $(D-\{x\})\cap \Sigma(\mathcal{O})=\emptyset$. Then the inclusion map $\eta:F-\text{Int}(D)  \hookrightarrow F$ defines an almost orbifold covering $\mathcal{O}'\rightarrow \mathcal{O}$ of degree one (Indeed any almost orbifold  covering of degree one is obtained in this way). The fundamental group of $\mathcal{O}'$ has the following presentation:
\begin{equation*} \label{eq1}
\begin{split}
\langle a_1,\ldots , a_p, t,    s_{1},\ldots, s_{i-1}, s_{i+1}, \ldots ,s_{r}  \ &  |  \ s_{1}^{p_{1}},\ldots, s_{i-1}^{p_{i-1}}, s_{i+1}^{p_{i+1}}, \ldots, s_{r}^{p_r},\\
 &   R\cdot s_{1}\cdot\ldots \cdot  s_{i-1}\cdot  t \cdot s_{i+1}\cdot  \ldots \cdot s_{r}  \rangle
\end{split}
\end{equation*}
where $t\in \pi_1^o(\mathcal{O'})$ is the element represented by $\partial D$.  Thus  $\mathcal{T}'=(a_1,\ldots, a_p, s_{1}^{\nu_1}, \ldots, s_{i-1}^{\nu_{i-1}}, s_{i+1}^{\nu_{i+1}}, \ldots,  s_r^{\nu_{r}})$   generates $\pi_1^o(\mathcal{O}')$. As $\eta$ is simply the inclusion map we further see that    $\eta_{\ast}(\mathcal{T}')=\mathcal{T}$.  
\begin{figure}[h!]
\begin{center}
\includegraphics[scale=1]{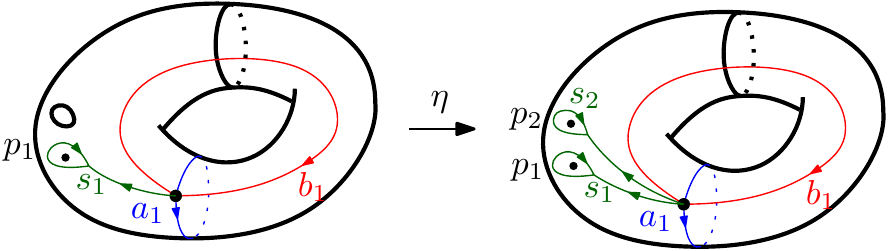}  
\end{center} 
\caption{A standard generating tuple of $\pi_1^o(\mathcal{O})$.}\label{fig:standard1}
\end{figure}  
\end{example}

\begin{remark}{\label{remark:deg1}}
Grushko's theorem implies that the converse of the previous example holds, i.e. if  there is an almost orbifold covering $\eta:\mathcal{O}'\rightarrow\mathcal{O}$ of degree one and a generating tuple $\mathcal{T}'$ of $\pi_1^o(\mathcal{O}')$ such that $\eta_{\ast}(\mathcal{T}')$ is Nielsen equivalent to  $\mathcal{T}$, then $\mathcal{T}$ is either reducible or   is Nielsen equivalent to a  standard generating tuple. 
\end{remark}

\begin{example}
Let $\mathcal{O}=T^2(2n+1, 2n)$  where  $n=2n_0+1\geqslant 5$ such that $n$ and $3$ are coprime. Let further  $\eta:\mathcal{O}'\rightarrow \mathcal{O}$ be the special almost orbifold  covering from  Example~\ref{ex:almost1}, and   $\sigma_1, \sigma_2, \sigma_3, \alpha_1, \alpha_2, \beta_1, \gamma_1$   the elements of $\pi_1^o(\mathcal{O}')$   illustrated in Fig.~\ref{fig:exthm}.   Consider the generating tuple   
$$\mathcal{T}_1=(s_1^2, s_2^3, a_1,  b_1^{-1} a_1b_1 , b_1^3,  b_1s_1b_1, b_1 s_2^4 b_1^{-1})$$ of 
$\pi_1^o(\mathcal{O})=\langle a_1,b_1, s_1, s_2 \ | \ s_1^{2n+1}, s_2^{2n}, [a_1,b_1]s_1s_2\rangle.$ The description of $\eta$ implies that $\eta_{\ast}$ maps the generating  tuple  $\mathcal{T}'=(\sigma_1^2, \sigma_2^3, \alpha_1, \alpha_2, \beta_1 , \gamma_1, \sigma_3^2)$  of $\pi_1^o(\mathcal{O}')$ onto a tuple that is Nielsen equivalent to $ \mathcal{T}_1$.   Now consider the generating tuple $$\mathcal{T}_2=(s_1^2, s_2, a_1,  b_1^{-1} a_1b_1 , b_1^3,  b_1s_1b_1, b_1 s_2^4 b_1^{-1})$$ 
of $\pi_1^o(\mathcal{O})$. Note that $\eta_{\ast}$ maps the generating tuple $\mathcal{T}''= (\sigma_1^2, \sigma_2 , \alpha_1, \alpha_2, \beta_1 , \gamma_1, \sigma_3^2)$  of $\pi_1^o(\mathcal{O}')$ onto a tuple that is Nielsen equivalent to $ \mathcal{T}_2$.   Using the relation $s_2^{-1}= [a_1,b_1]s_1$ we see that $\mathcal{T}_2$ is Nielsen equivalent to the tuple $(a_1, b_1, s_1, 1, 1, 1,1)$.   Thus there is an  almost orbifold   covering $\eta'':\mathcal{O}''\rightarrow \mathcal{O}$  of degree one (and therefore special)   and a  reducible  generating tuple $\mathcal{T}''$ of $\pi_1^o(\mathcal{O}'')$ such that $\eta_{\ast}''(\mathcal{T}'')\sim_{NE} \mathcal{T}_2$.   
\begin{figure}[h!]
\begin{center}
\includegraphics[scale=1]{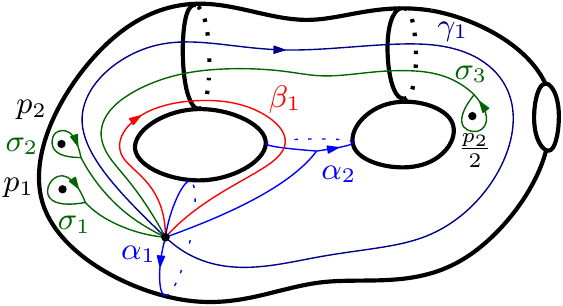}
\end{center} 
\caption{The generating tuple $\mathcal{T}'=(\sigma_1^2, \sigma_2^3, \alpha_1, \alpha_2,   \beta_1,  \gamma_1,   \sigma_3^{2})$ of $\pi_1^o(\mathcal{O}')$.}\label{fig:exthm}
\end{figure}  
\end{example}
\begin{remark}
The previous example  shows that the special almost orbifold covering and the generating tuple given in Theorem~\ref{MainThm} are not unique. However, in  a  joint work with  Richard Weidmann  we  aim to address the extent of this
non-uniqueness.  We are trying to prove, in particular, that generating tuples that come from non-trivial almost orbifold coverings  are irreducible. 
\end{remark}

 Louder~\cite[Theorem 2.3]{Louder} proved that  a generating  tuple of a closed surface group is either reducible or it is Nielsen equivalent to a standard generating tuple.  As a corollary of Theorem~\ref{MainThm} we obtain the following generalization of Louder's result.  
\begin{corollary}{\label{cor:Louder}}[Louder]
If $\mathcal{O}=F(p_1,\ldots,p_r)$ is a sufficiently large orbifold with either $r=0$ or $r\geqslant 1$ and  $p_1=\ldots=p_r= 2$, then any generating tuple of $\pi_1^o(\mathcal{O})$ is either reducible or  is Nielsen equivalent to a standard generating tuple. 
\end{corollary}

 The paper is organized as follows. In Section 2, we recall  some basic facts and definitions regarding graphs of groups,  graphs of groups morphisms  and  foldings. In section 3,  we define the notion of decorated groups over the fundamntal group of small $2$-orbifolds   and state Proposition~\ref{proposition:1}, which is the main result  to prove Theorem~\ref{MainThm}.  In Section 4, we prove Theorem~\ref{MainThm}. In Section 5 we derive Louder's theorem from Theorem~\ref{MainThm}.  Finally,  in Section 6, we prove Proposition~\ref{proposition:1}. It should be noted that most features and ideas of Louder's paper are preserved, but the proof of Theorem~\ref{MainThm} is much more involved and subtle  because of the presence of torsion elements in the fundamental group of an orbifold.

\section{Graph of groups  morphisms and folds}

In this section we fix notations for  the theory of morphisms between graph of groups and foldings as we will need precise language later on. Graphs of grups morphisms  were introduced by Hyman Bass~\cite{Bass} while the notion of folds of graph of groups morphisms was introduced by Dunwoody-Bestvina-Feighn-Stallings~\cite{BF,D,S, S1}.   We will follow the treatment given in~\cite{KMW} and in \cite{RW}.

\subsection{Graphs of groups.}
A \emph{graph} $A$ is a tuple  $(VA,EA, \alpha, \omega, ^{-1})$,  where $^{-1}:EA\rightarrow EA$ is a fixed point free involution and $\alpha, \omega:EA\rightarrow VA$ are maps such that
$$ \omega(e)=\alpha(e^{-1}) \ \ \text{ and that } \ \ \alpha(e)\neq \omega(e) \ \ \text{ for all  } \ \  e\in EA.$$ 
We call $VA$  the \emph{vertex set of $A$} and $EA$ the  \emph{edge set of $A$}. We refer to $\alpha(e)$ as the \emph{initial vertex} of $e$ and to $\omega(e)$ as the \emph{terminal vertex} of $e$. For each vertex $v\in VA$, the set  
$$\st(v,A):=\{e\in EA \ | \ \alpha(e)=v\}\subseteq EA$$ 
is called the \emph{star of $v$}. The graph $A$ is said to be finite if the vertex set  and the  edge set of $A$  are finite.

\begin{remark}
Note that the definition given here  is slightly different from that given by Serre~\cite{Serre} as we do not allow edge loops, i.e we impose that $\alpha(e)\neq \omega(e)$ for all edges $e$ of $A$.  
\end{remark} 
 
A \emph{path} $p$ in $A$ is a finite  sequence of edges $e_1,\ldots, e_k$ such that $\alpha(e_{i+1})=\omega(e_{i})$ for $1\leqslant i\leqslant k-1$. We call $\alpha(e_1)$ the \emph{initial} vertex of $p$, we denote $\alpha(p)$, and $\omega(e_k)$ the \emph{terminal} vertex of $p$, which  we denote by  $\omega(p)$.   We say that  $A$ is \emph{connected} if  given any two vertices   $v$ and $w$ of  $A$  there is a path in $A$ with initial vertex $v$ and terminal vertex $w$. 
 
 
\begin{definition}[Graphs of groups]
A \emph{graph of groups} $\A$ is a tuple
$$(A, \{A_v \ | \ v\in VA\}, \{A_e \ | \ e\in EA\}, \{\alpha_e \ | \ e\in EA\}, \{\omega_e \ | \ e\in EA\})$$
where $A$ is a finite graph, $A_v$ and $A_e$ are groups (called the vertex and edge groups,  respectively) with $A_{e^{-1}}=A_e$ for all $e\in EA$, and where  $\alpha_e:A_e\rightarrow A_{\alpha(e)}$ and $\omega_e:A_e\rightarrow A_{\omega(e)}$ are   monomorphisms  such that $\omega_e=\alpha_{e^{-1}}$  for all $e\in EA$. We call the maps $\alpha_e$ and $\omega_e$ \emph{boundary monomorphisms}.   
\end{definition}
 
An \emph{$\A$-path of length $k\geqslant 0$  from $v\in VA$ to $v'\in VA$}, or simply a path in $\A$ from $v$ to $v'$,  is a finite sequence $$p=a_0,e_1,a_1,\ldots, a_{k-1}, e_k, a_k$$ 
where $k\geqslant 0$ is an integer, $e_1,\ldots, e_k$ is a path in $A$ from the vertex $v$ to the vertex $v'$, where $a_0\in A_v=A_{\alpha(e_1)}$ and $a_i\in A_{\omega(e_i)}$ for $1\leqslant i\leqslant k$. The vertex $v=\alpha(e_1)$ is called the initial vertex  of $p$, we  denote  $\alpha(p)$. Analogously, the vertex $v'=\omega(e_k)$  is called the terminal vertex of $p$, we denote  $\omega(p)$.   The  integer $k\geqslant 0$ is called the length of $p$ and is  denoted  by $|p|$. Note that   $|p|=0$ implies that   $v=v'$ and $p=a_0\in A_v$.

Let $p=a_0, e_1, a_1, \ldots, a_{k-1}, e_k, a_k$ and $p'=a_0', e_1', a_1', \ldots, a_{k'-1}',  e_{k'}', a_{k'}'$ be $\A$-paths, $i$ and $j$ integers such that $1\leqslant i\leqslant j\leqslant k$, and $a_{i-1}', a_{i-1}''$ elements of $ A_{\alpha(e_{i})}$ such  that $a_{i-1}=a_{i-1}''a_{i-1}'$. We define the following   $\A$-paths:
\begin{enumerate}
\item[1.]  The  $\A$-path $a_{i-1}, e_i, a_i, \ldots , e_j, a_j$ is called  an \emph{$\A$-subpath} of $p$ .

\item[2.] If $\omega(p)=\alpha(p')$ then the  $\A$-path  $pp':= a_0,e_1,a_1,\ldots, e_k, a_ka_0', e_1', a_1', \ldots, e_{k'}', a_{k'}$ 
is called the \emph{concatenation} of $p$ and $p'$.

\item[3.] If $\alpha(p)=\omega(p)$ then the $\A$-path $a_{i-1}', e_i, a_i, \ldots, e_k, a_ka_0, e_1, a_1, \ldots, e_{i-1}, a_{i-1}''$ 
is called a \emph{cyclic permutation} of $p$.
\end{enumerate}

An equivalence relation   $\sim$ on the  set of all $\A$-paths is  generated by the elementary equivalences 
$$a,e,\omega_e(c), e^{-1}, a' \sim a\alpha_e(c)a'   \ \ \text{ and }  \ \  a,e,a'\sim a\alpha_e(c), e, \omega_e(c^{-1}) a' \ \ \text{ for }  c \in A_e.$$
The $\sim$-equivalence class of an $\A$-path $p$ is denoted by $[p]$.   An $\A$-path is $\A$-\emph{reduced}, or just \emph{reduced},  if  it has no $\A$-subpath of the form $a, e, \omega_e(c), e^{-1}, a'$ ($c\in A_e$), in other words, no elementary equivalences of the first type are applicable to it.

\begin{definition}[Fundamental group of a graph of groups] Given a base vertex $v_0\in VA$, \emph{the fundamental group of $\A$ with respect to $v_0$}, $\pi_1(\A, v_0)$, is the set of $\sim$-equivalence classes of $\A$-paths from $v_0$ to $v_0$, with multiplication given by  $[p][q]:=[pq]$.   
\end{definition}

We say that a graph of groups $\A'$ is a \emph{sub-graph of groups of $\A$}, and write $\A'\subseteq \A$, if the graph $A'$ underlying $\A'$ is a sub-graph of  $A$  and  the following hold:
\begin{enumerate}
\item[(Vertex groups)] For each $v\in VA'\subseteq VA$, the group  $A_v'$  of the vertex $v$ in $\A'$  is a free factor of the vertex group  $A_v$ of $\A$ at $v$ and $\alpha_e(A_e)\leq A_v'$ for all $e\in \st(v, A')$, that is, there is a subgroup $A_v''\leq A_v$ such that $A_v=A_v'\ast A_v''$ and   $\alpha_e(A_e)\leq A_v'$ for all $e\in \st(v, A')$.

\item[(Edge groups)] For each $e\in EA'\subseteq EA$ the   group $A_e'$ of the edge $e$ in $\A'$  is  $A_e'=A_e$. The boundary monomorphism  $\alpha_e':A_e'\rightarrow A_{\alpha(e)}'$ is obtained from $\alpha_e:A_e\rightarrow A_{\alpha(e)}$ by replacing the codomain  $A_{\alpha(e)}$ of $\alpha_e$  by the group $A_{\alpha(e)}'\leq A_{\alpha(e)}$.
\end{enumerate}

If $\A'\subseteq \A$ is a sub-graph of groups of $\A$   with underlying graph $A'\subseteq A$  and if   $A_v'=A_v$ for all vertices $v\in VA'$,  then we will say that  $\A'\subseteq \A$ is \emph{carried} by the sub-graph $A'$  and we will  denote $\A'$ by $\A(A')$.

A special class  of sub-graph of groups  are those for which   the underlying graph $A'\subseteq A$  consists of a  single vertex  $v\in VA$. Note that in this case   $\pi_1(\A', v)\cong A_v'\leq A_v$.

Let $\A'\subseteq \A$ be a sub-graph of groups of $\A$,  $v_0\in VA$ and $v_0'\in VA'$.  For each subgroup $U\leq \pi_1(\A', v_0')$ and  each $\A$-path  $\gamma$  with initial vertex  $\alpha(\gamma)=v_0$ and terminal vertex $\omega(\gamma)=v_0'$, we define 
$$\gamma   U  \gamma^{-1}:=\{ [\gamma   p   \gamma^{-1}] \ | \  p \text{ is an } \A'\text{-path from } v_0' \text{ to } v_0' \ \text{ and } \ [p]\in U\}.$$
It follows from the Normal Form Theorem~\cite[Chapter I, Theorem 11]{Serre} that the homomorphism $U \rightarrow \pi_1(\A, v_0)$ given by $[p]\mapsto [\gamma   p   \gamma^{-1}]$ is injective. Thus the subgroup  $\gamma  U \gamma^{-1}\leq \pi_1(\A, v_0)$ is isomorphic to $U\leq \pi_1(\A', v_0')$.

\subsection{Graph of groups morphisms.} The main result of Bass-Serre theory states that if a group $G$ acts without inversions of edges on a simplicial tree $T$, then $G$ is isomorphic to the fundamental group of a graph of groups.
Now  a morphism from a  $G$-tree $T$ to a $H$-tree $S$ consists of a group homomorphism $f:G\rightarrow H$ and an $f$-equivariant  graph-morphism  $\varphi:T\rightarrow S$. Any such morphism can be encoded  on the level of the associated graph of groups  as follows.

\begin{definition}[Graph of groups morphisms]
Let $\A$ and $\B$ be graphs of groups.  A \emph{morphism}  $\varphi:\B\rightarrow \A$ from $\B$ to $\A$ is a tuple
 $\varphi=(\varphi, \{\varphi_u \ | \ u\in VB\}, \{\varphi_f \ | \ f\in EB\}, \{o_f  \ | \ f\in EB\}, \{t_f  \ | \ f\in EB\})$  
where 
\begin{enumerate}
\item[(1)] $\varphi$ is a graph-morphism from $B$ to $A$. 

\item[(2)] for each $u\in VB$,  $\varphi_u$ is a group homomorphism from $B_u$ to $A_{\varphi(u)}$. 

\item[(3)] for each $f\in EB$,  $\varphi_f$ is a group homomorphism from  $B_f$ to $A_{\varphi(f)}$ such that $\varphi_f=\varphi_{f^{-1}}$.

\item[(4)] for each $f\in EB$,  $o_f\in A_{\alpha(\varphi(f))}$ and $t_f\in A_{\omega(\varphi(f))}$ such that $t_{f}^{-1}=o_{f^{-1}}$.

\item[(5)] for each $f\in EB$ the diagram
$$
\begin{tikzcd}
B_f \arrow[rr,"\alpha_f " ] \arrow[d,"\varphi_f"' ]  &  & B_{\alpha(f)} \arrow[d,"\varphi_{\alpha(f)}" ] \\ 
A_{\varphi(f)} \arrow[r,"\alpha_{\varphi(f)} "' ] &  A_{\alpha(\varphi(f))} \arrow[r,"i_{o_f}"' ] &   A_{\alpha(\varphi(f))} 
\end{tikzcd}$$
commutes, where $i_{o_f}$ denotes the inner automorphism $g\mapsto o_f g  o_f^{-1}$. 
\end{enumerate}  
\end{definition}
The homomorphisms $\varphi_u$ and $\varphi_f$ for $u\in VB$ and $f\in EB$ are called   \emph{vertex homomorphisms} and \emph{edge homomorphisms} respectively. The elements $o_f$ and $t_f$ for $f\in EB$  are called \emph{edge elements}.  We will write $o_{f}^{\varphi}$ or $t_f^{\varphi}$ instead of $o_f$ or $t_f$ if we want to make explicit that these elements come from $\varphi$.

Every  graph of groups morphism $\varphi:\B\rightarrow \A$ induces a group homomorphism  $\varphi_{\ast}:\pi_1(\B,u_0)\rightarrow \pi_1(\A,\varphi(u_0))$ 
in the following way.  To any $\B$-path $q=b_0,f_1,b_1,\ldots, b_{k-1},  f_k,b_k$ we associate the $\A$-path
$$\varphi(q):=a_0,\varphi(f_1), a_1, \ldots, a_{k-1}, \varphi(f_k), a_k$$
where $a_0=\varphi_{\alpha(f_1)}(b_0) o_{f_1}$, $a_k=t_{f_k}  \varphi_{\omega(f_k)}(b_k)$ and $ a_i=t_{f_i} \varphi_{\omega(f_i)}(b_i)  o_{f_{i+1}}$ for $1\leqslant i\leqslant k-1.$ It is easily seen that $\varphi$ preserves the  equivalence relation $\sim$ and that $\varphi(qq')=\varphi(q)\varphi(q')$ whenever $qq'$ is defined. Thus the map 
$$\varphi_{\ast}:\pi_1(\B,u_0)\rightarrow \pi_1(\A,\varphi(u_0)),  \ \ \  [q]\mapsto [\varphi(q)]$$  
defines a group homomorphism, called the \emph{homomorphism  induced by $\varphi$}.

In what follows we will be mostly interested in morphisms that satisfy additional hypotheses. 
\begin{definition}[Folded morphisms]{\label{def:folded}}
We say that a morphism  $\varphi:\B\rightarrow \A$ is \emph{folded} if the following hold:
\begin{enumerate}
\item[(F0)] $\varphi$ is \emph{vertex injective}, i.e. $\varphi_x:B_x\rightarrow A_{\varphi(x)}$ is injective for all $x\in VB$.  

\item[(F1)] There exist  no edges $f_1\neq f_2\in EB$ with $x:=\alpha(f_1)=\alpha(f_2) \in VB$ and $e:=\varphi(f_1)=\varphi(f_2)\in EA$ such that $o_{f_2}=\varphi_x(b) o_{f_1}\alpha_{e} (c)$ for some $b\in B_x$ and $c\in A_e$.

\item[(F2)] There exists no edge $f\in EB$ such that  $o_f^{-1} \varphi_{\alpha(f)}(b)  o_f$ lies in $\alpha_{\varphi(f)} (A_{\varphi(f)})$ for some $b\in B_{\alpha(f)}-\alpha_f (B_f)$. 
\end{enumerate}
\end{definition}

 \begin{remark}
Condition (5) of the definition of graph of groups morphisms implies that $\varphi_f(B_f)$ is a subgroup of $\alpha_{\varphi(f)}^{-1}(o_f^{-1} \varphi_{\alpha(f)}(B_{\alpha(f)})o_f)$
for all $f\in EB$. If $\varphi:\B\rightarrow \A$ satisfies condition (F2) of Definition~\ref{def:folded}, then
\begin{equation}{\label{eq:F2'}}
\varphi_f(B_f)=\alpha_{\varphi(f)}^{-1} (o_f^{-1}  \varphi_{\alpha(f)}(B_{\alpha(f)})  o_f) \ \text{ for all } f\in EB. \tag*{$(F2)'$}
\end{equation}
 For vertex injective morphisms conditions (F2) and (F2)$'$ are equivalent, compare with \cite[Definition 4.5]{RW} in the context of $\A$-graphs. 
\end{remark}

With a simple modification of the proof of Lemma 4.2 of \cite{KMW} one can  show that a folded morphism  $\varphi:\B\rightarrow \A$ maps   reduced  $\B$-paths to   reduced  $\A$-paths. Consequently, by the Normal Form Theorem~\cite[Chapter I, Theorem 11]{Serre}, the induced homomorphism  $\varphi_{\ast}:\pi_1(\B,u_0)\rightarrow \pi_1(\A,\varphi(u_0))$ is injective.

\begin{definition}[Locally surjective morphisms]
Let $\varphi:\B\rightarrow \A$ be a morphism and $x$ a vertex of $\B$.   We say that $\varphi$ is \emph{locally surjective at $x$} if for each $e\in \st(\varphi(x),A)$ there is a subset $S_e\subseteq \st(x,B)\cap \varphi^{-1}(e)$ such that 
$$A_{\varphi(x)}=\bigcup_{f\in S_e} \varphi_x(B_x)  o_f   \alpha_{e}(A_e).$$  We  say that $\varphi$ is \emph{locally surjective} if $\varphi$ is locally surjective at all vertices  of $\B$.
\end{definition}

The composition of graph of groups  morphisms is defined as follows. Let $\varphi:\mathbb{B}\rightarrow \A$ and $\phi:\mathbb{D} \rightarrow \mathbb{B}$ be  graph of groups  morphisms,  we  define  $\psi:=\varphi\circ \phi:\mathbb{D}\rightarrow \A$ 
by setting 
$$\psi :=(\psi, \{\psi_w\ | \ w\in VD\} , \{\psi_g\ | \ g\in ED\}, \{o_g^{\psi}\ | \ g\in ED\},  \{t_g^{\psi}\ | \ g\in ED\})$$
where $\psi= \varphi\circ \phi:D\rightarrow A$, $\psi_{x}= \varphi_{\phi(x)}\circ \phi_x:D_x\rightarrow A_{\varphi(\phi(x))}=A_{\psi(x)}$ for all $x\in VD\cup ED$, and  $o_g^{\psi}= \varphi_{\alpha(\phi(g))}(o_g^{\phi})  o_{\phi(g)}^{\varphi}$ for all $g\in ED$.   A simple calculation shows that 
$$\psi_{\alpha(g)}\circ \alpha_g(x)= o_g^{\psi} \alpha_{\psi(g)}\circ \psi_g(x)   (o_g^{\psi})^{-1}$$ 
for all $x\in D_g$ and all   $g\in ED$. It is also easily verified that $\psi_{\ast}=(\varphi \circ \phi)_{\ast}= \varphi_{\ast} \circ \phi_{\ast}$.
 
An example of composition of morphisms  is the   restriction of a morphism to a sub-graph of groups. More precisely, let $\varphi:\B\rightarrow \A$ be a graph of groups morphism and $\B'\subseteq \B$ a sub-graph of groups of $\B$. Then there is an obvious morphism $\iota:\B'\rightarrow \B$    given by $$\iota=(\iota, \{\iota_u \ | \ u\in VB'\}, \{\iota_f \ | \ f\in EB'\}, \{o_f^{\iota} \ | \ f\in EB'\}, \{t_f^{\iota} \ | \ f\in EB\}).$$ where $\iota:B'\hookrightarrow B$ is the inclusion map, $\iota_x: B_x'\hookrightarrow B_{x}$ is the inclusion homomorphism for all $x\in VB'\cup EB'$ and $o_f^{\iota}=1$ for all $f\in EB'$.  The composition $\varphi\circ \iota:\B'\rightarrow \A$ will be denoted   by $\varphi|_{\B'}$ and referred to   as   \emph{the restriction of $\varphi$ to the sub-graph of groups $\B'$.}

\subsection{Auxiliary moves} Auxilairy moves were defined in  \cite{KMW} and in   \cite{RW},  for   $\A$-graphs and are merely used as preprocessing tools to define the main folding moves.  Here we  simply translate the definitions given  in~\cite{RW} to the language of graphs of groups morphisms.

\noindent \emph{Conjugation Move A0.} Let $\varphi:\mathbb{B}\rightarrow\A$ be a  morphism. Suppose $u$ is a vertex of $\mathbb{B}$ and that $g\in A_{\varphi(u)}$. 

Let $\varphi':\mathbb{B}\rightarrow \A$ be the morphism obtained from $\varphi$ as follows:
\begin{enumerate}
\item[(1)] replace the vertex homomorphism  $\varphi_u:B_u\rightarrow A_{\varphi(u)}$ by the homomorphism  $\varphi_u' =i_{g}\circ \varphi_u$ where  $i_g$ is the inner automorphism of $A_{\varphi(u)}$ determined by $g$.

\item[(2)] for each  $f\in\st(u,B)$ replace  the edge element  $o_f^{\varphi}\in A_{\varphi(u)}$ by $o_{f}^{\varphi'}:=g  o_{f}^{\varphi}\in A_{\varphi(u)}$. 
\end{enumerate}

In this case we will  say that $\varphi'$ is obtained from $\varphi$ by an \emph{auxiliary move of type A0}.  If $u'\in  VB$, $u\neq u'$, is another vertex  (whose vertex homomorphism is therefore not changed by the move),  we will say that this A0-move is admissible with respect to $u'$.

\noindent\textit{Bass-Serre move A1.} Let $\varphi:\mathbb{B}\rightarrow\A$ be a morphism. Suppose  $f$  is an edge of $\mathbb{B} $  and that $c\in A_{\varphi(f)}$. 

Let $\varphi':\mathbb{B}\rightarrow \A$ be the graph of groups morphism obtained from $\varphi$  as follows: 
\begin{enumerate}
\item[(1)] replace  the edge homomorphism $\varphi_f:B_f\rightarrow A_{\varphi(f)}$ by  the homomorphism  $\varphi_f':B_f\rightarrow A_{\varphi(f) }$ given by  $\varphi_f'(x) =c\varphi_f(x)c^{-1}$ for all $x\in B_f$.

\item[(2)] replace the edge element  $o_f^{\varphi}\in A_{\alpha(\varphi(f))}$ by  $o_f^{\varphi'}:= o_f^{\varphi }   \alpha_{\varphi(f)}(c^{ -1})\in A_{\alpha(\varphi(f))}$.

\item[(3)] replace  the edge element  $t_f^{\varphi}\in A_{\omega(\varphi(f))}$ by $t_f^{\varphi'}:=\omega_{\varphi(f)}(c)  t_f^{\varphi}\in A_{\omega(\varphi(f))}$. 
\end{enumerate}
In this case we will  say that $\varphi'$ is obtained from $\varphi$ by an \emph{auxiliary move of type A1}.

\begin{lemma}{\label{lemma:auxmoves}} Let $\varphi:\B\rightarrow\A$ be a graph of groups morphism  and $u_0\in VB$ with $\varphi(u_0)=v_0$. Suppose that the morphism  $\varphi':\mathbb{B} \rightarrow \A$ is obtained from $\varphi$ by an auxiliary move of type A1 or an  auxiliary move of type  A0 that is admissible with respect to $u_0$. Then  $\varphi_{\ast}=\varphi_{\ast}':\pi_1(\mathbb{B},u_0)\rightarrow \pi_1(\A,v_0).$ 
\end{lemma}

\noindent \textit{Simple adjustment A2.} Let $\varphi:\mathbb{B}\rightarrow\A$ be a  morphism. Suppose  $f$ is an edge of $\mathbb{B}$ and that  $b\in B_{\alpha(f)}$.

We define a graph of groups $\B'$ and a morphism $\varphi :\B'\rightarrow \A$ as follows:
\begin{enumerate}
\item[(1)] $\mathbb{B}'$ is obtained from $\mathbb{B}$ by replacing the boundary monomorphism $\alpha_f:B_f\rightarrow B_{\alpha(f)}$ by  the monomorphism $\alpha_f':B_f\rightarrow B_{\alpha(f)}$ defined by  $\alpha_f'(x)=b\alpha_f(x)b^{-1}$  for all $x\in B_f$. 

\item[(2)]  $\varphi':\mathbb{B}'\rightarrow \A$ is obtained from $\varphi$ by replacing  the edge element  $o_f^{\varphi}$   by $o_f^{\varphi'}:=\varphi_{\alpha(f)}(b)  o_f^{\varphi}$. 
\end{enumerate}

In this case we will  say that $\varphi':\mathbb{B}'\rightarrow\A$ is obtained from $\varphi$ by an \emph{auxiliary move of type A2}. 

\begin{lemma}{\label{lemma:simpleadjustment}}
Let $\varphi:\mathbb{B}\rightarrow \A$ be a graph of groups morphism. If  $\varphi':\mathbb{B}'\rightarrow \A$ is obtained from $\varphi$ by an auxiliary move of type A2, then there exists a graph of groups isomorphism $\sigma:\mathbb{B}\rightarrow \mathbb{B}'$ such that $\varphi' \circ \sigma = \varphi$.
\end{lemma} 
\begin{proof}
Suppose that $\varphi'$ is obtained from $\varphi$ by an auxiliary move of type A2 applied to the edge $f$ with element $b\in B_u$ where $u:=\alpha(f)$. The   assertion holds as $\varphi= \varphi' \circ \sigma$ where $\sigma:\mathbb{B}\rightarrow \mathbb{B}'$ is the graph of groups isomorphism
$$\sigma=(\sigma, \{\sigma_u\ | \ u\in VB\}, \{\sigma_f\ | \ f\in EB\},  \{o_f^{\sigma} \ | \ f\in EB\}, \{t_f^{\sigma} \ | \ f\in EB\})$$
where $\sigma=\text{Id}_B:B\rightarrow B$, $\sigma_x=\text{Id}_{B_x}:B_x\rightarrow B_x$ for all $x\in VB\cup EB$, $o_g^{\sigma}=1$ for $g\neq f$ and $o_f^{\sigma}=b^{-1}\in B_{\alpha(f)}$.
\end{proof}

\begin{remark}[Auxiliary moves of type A2 applied to graphs of groups with trivial edge groups]{\label{remark:auxmove}}
Let $\B$ be a graph of groups with trivial edge groups and suppose that the morphism $\varphi':\B'\rightarrow \A$ is obtained from  $\varphi$ by an auxiliary move of type A2 applied to the edge $f$ with element $b\in B_u$  where $u:=\alpha(f)$.  From the description of the move and the fact that all edges have trivial group in $\B$, we see that  the graph of groups  $\B'$ and $\B$ coincide.  Thus $\varphi'$ is   a morphism from $\B$ to $\A$. The previous lemma gives   a graph of groups automorphism $\sigma:\B\rightarrow \B$ such that $\varphi'\circ \sigma=\varphi$.

The image of a $\B$-path under $\sigma$ can easily be described as follows. Let  $q$ be an arbitrary  $\B$-path. Then   $q$  can be written as   
 $q=q_0  (1,f,1)^{\varepsilon_1}  q_1 (1,f,1)^{\varepsilon_2}  q_2  \cdot \ldots\cdot q_{l-1}  (1,f,1)^{\varepsilon_l} q_l $ 
where $\varepsilon_i=\pm 1$ and the paths $q_0,q_1,\ldots,q_l$ are contained in the sub-graph of groups $\B_f\subseteq \B$ carried by the sub-graph $B-\{f, f^{-1}\}$ of $B$.  By the definition of   $\sigma:\B\rightarrow \B$, we see that  
\begin{eqnarray}
\sigma(q) & = &  \sigma(q_0  (1,f,1)^{\varepsilon_1}  q_1 (1,f,1)^{\varepsilon_2}  q_2  \cdot \ldots\cdot q_{l-1}  (1,f,1)^{\varepsilon_l} q_l)\nonumber \\
          & = & \sigma(q_0) \sigma(1,f,1)^{\varepsilon_1}\sigma(q_1) \sigma(1,f,1)^{\varepsilon_2}  \sigma(q_2)  \cdot \ldots \cdot \sigma(q_{l-1}) \sigma(1,f,1)^{\varepsilon_l}\sigma(q_l)\nonumber \\
          & = & q_0 (b^{-1}, f , 1)^{\varepsilon_1}  q_1 (b^{-1},f,1)^{\varepsilon_2}  q_2 \cdot \ldots\cdot q_{l-1} (b^{-1},f,1)^{\varepsilon_l}  q_l. \nonumber
\end{eqnarray}
\end{remark}

\subsection{Folds and vertex morphisms.} We will only discuss folding  moves that are  applicable to morphisms that are not folded because either condition (F0) is not satisfied or condition   (F1) is not satisfied. Moreover, we will also restrict our attention to  morphisms defined on graph of groups with trivial edge groups. In order to define the main folding moves we  first discuss elementary folds. 

For the remainder of this section let $\B$ be a graph of groups with base vertex $u_0$  such that $B_f=1$ for all edges $f$ of $\B$, and  let  $\varphi:\B\rightarrow \A$ be  a graph of groups morphism that maps $u_0$ onto the base vertex $v_0$ of $\A$.

\smallskip

\noindent\emph{Elementary fold of type IA}.   Suppose $f_1$ and $f_2$ are edges of $\B$ with $u:=\alpha(f_1)=\alpha(f_2)$,  $ \omega(f_1)\neq \omega(f_2)$ and $e:=\varphi(f_1)=\varphi(f_2)$  such that $a:=o_{f_1}^{\varphi}=o_{f_2}^{\varphi}$ and $b:=t_{f_1}^{\varphi}=t_{f_2}^{\varphi}$. Suppose further that $ \omega(f_1)=x$ and $\omega(f_2)=y$. Clearly $x$ and $y$  are mapped by $
\varphi$ onto the vertex $w:=\omega(e) $.

We define a graph of groups $\B'$ and a morphism $\varphi':\B'\rightarrow \A$ as follows. 
\begin{enumerate}
\item[(1)]  $\B'$ is obtained from $\mathbb{B}$ by identifying the edges $f_1$ and $f_2$ into a single edge $f$, with trivial group, and by identifying  the vertices $x$ and $y$ into a single vertex $z:=\omega(f)$ with group $B_z':= B_x\ast  B_y$. 

\item[(2)] ${\varphi'}:{\mathbb{B}'}\rightarrow\A$ is given by $$\varphi'=(\varphi', \{\varphi'_{u} \ | \ u\in VB\}, \{\varphi_f \ | \ f\in EB\}, \{o_f^{\varphi'} \ | \ f\in EB\}, \{t_f^{\varphi'} \ | \  f\in EB\})$$ where the graph morphism ${\varphi'}:{B}'\rightarrow A$ is induced by $\varphi:B\rightarrow A$, that is, $\varphi'\circ \sigma =\varphi$ where $$\sigma:B\rightarrow B'=B/[f_1=f_2]$$ 
is the quotient graph-morphism. The vertex homomorphism  $\varphi_z':  {B}_z'=B_x\ast B_y \rightarrow A_w$ is the homomorphism induced by $\varphi_x:B_x\rightarrow A_{w}$ and $\varphi_y :B_y\rightarrow A_w$. Finally, $o_f^{{\varphi}'}=a   \in A_v$ and  $t_{f}^{{\varphi}'}=b \in A_w$. 
\end{enumerate}

In this case we will say that  ${\varphi}':{\mathbb{B}'}\rightarrow \A$ is obtained from $\varphi$ by an \emph{elementary fold of type IA}.
\begin{figure}[h!]
\begin{center}
\includegraphics[scale=1]{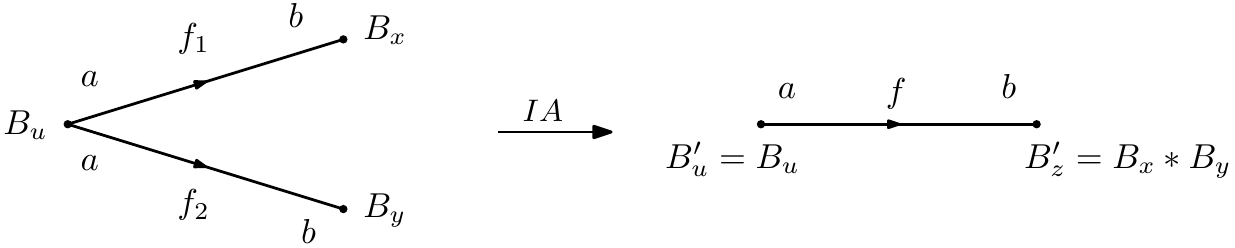}
\end{center}
\caption{An elementary fold of type IA.}
\end{figure}

\noindent\emph{Elementary fold of type IIIA.}   Suppose that $f_1$ and $  f_2$ are edge of $\B$ with  $u:=\alpha(f_1)=\alpha(f_2)$, $x:=\omega(f_1)=\omega(f_2)$ and  $e:=\varphi(f_1)=\varphi(f_2)$ such that $a:= o_{f_1}^{\varphi} = o_{f_2}^{\varphi}$.  Put $b_1:=t_{f_1}^{\varphi}$ and $b_2:=t_{f_2}^{\varphi}$.

We define a graph of groups $\B'$ and a morphism $\varphi':\B'\rightarrow \A$ as follows. 
\begin{enumerate}
\item[(1)] ${\mathbb{B}'}$ is obtained from $\mathbb{B}$ by identifying the edges $f_1$ and $f_2$ into a single edge $f$,  with trivial group, and by replacing the vertex group $B_x$  by the  group $B_x':= B_x\ast \langle b_x \ | \ -\rangle$. 

\item[(2)] ${\varphi'}:{\mathbb{B}'}\rightarrow\A$ is given by $$\varphi'=(\varphi', \{\varphi'_{u} \ | \ u\in VB\}, \{\varphi_f \ | \ f\in EB\}, \{o_f^{\varphi'} \ | \ f\in EB\}, \{t_f^{\varphi'} \ | \  f\in EB\})$$ where the graph morphism ${\varphi'}:{B}'\rightarrow A$ is induced by $\varphi:B\rightarrow A$, that is, $\varphi'\circ \sigma =\varphi$ where $$\sigma:B\rightarrow B'=B/[f_1=f_2]$$ 
is the quotient graph-morphism.   The vertex homomorphism  $\varphi_z': {B}_z'\rightarrow A_w$ is induced by the homomorphisms  $\varphi_x:B_x\rightarrow A_{w}$ and the map $b_x\mapsto b_1^{-1} b_2$.  Finally,  $o_f^{{\varphi}'}=a \in A_v$ and  $t_{f}^{{\varphi}'}= b_1 \in A_w$.
\end{enumerate}
In this case we will say  that ${\varphi}':{\mathbb{B}'}\rightarrow \A$ is obtained from $\varphi$ by an \emph{elementary fold of type IIIA}.
\begin{figure}[h!]
\begin{center}
\includegraphics[scale=1]{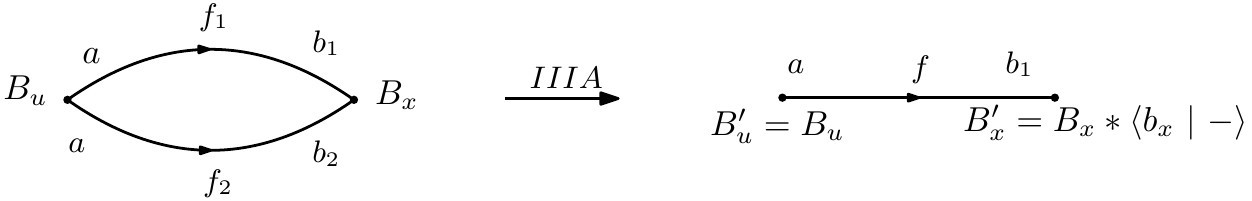}
\end{center}
\caption{An elementary fold of type IIIA.}
\end{figure}

Following the lines of the proof of Proposition 4.15 of \cite{KMW} we prove the following lemma. 
\begin{lemma}{\label{lemma:fold}}
Let $\varphi:\mathbb{B}\rightarrow \A$ be a graph of groups morphism and $u_0\in VB$ with $\varphi(u_0)=v_0$. If $ {\varphi}':{\mathbb{B}'}\rightarrow \A$ is obtained from $\varphi$ by an elementary  fold, then there exists a     morphism $\sigma:\B\rightarrow  \B'$ such that (i) ${\varphi}'\circ \sigma =\varphi$ and (ii) $\sigma_{\ast}:\pi_1(\B,u_0)\rightarrow \pi_1(\B',\sigma(u_0))$ is an isomorphism.
\end{lemma}

\smallskip 


Now we define moves   that are applicable to morphisms that are not folded because either  condition (F0) or condition (F1) fails to hold.   Assume first that $\varphi:\B\rightarrow \A$ is not folded because  condition (F0) is not satisfied. Thus there is a vertex $u$ of $\B$ such that the vertex homomorphism $ \varphi_u:B_u\rightarrow A_{\varphi(u)}$ is not injective.

We define a graph of groups $\B'$  and a morphism $\varphi'$  from $\B'$ to $\A$ as follows:
\begin{enumerate}
\item[(1)] ${\mathbb{B}'}$ is  obtained from $\mathbb{B}$ by replacing the vertex group $B_u$ by the group $ {B}_u':= B_u / \ker(\varphi_u)$.

\item[(2)] ${\varphi}':{\mathbb{B}'}\rightarrow \A$ is  obtained from $\varphi$ by replacing the vertex homomorphism  $\varphi_u:B_u\rightarrow A_{\varphi(u)}$ by the  homomorphism $\varphi_u':  B_u'=B_u/\ker(\varphi_u)  \to A_{\varphi(u)}$ defined by the equation $\varphi_u'( \ker(\varphi_u)x)=\varphi_u(x)$. 
\end{enumerate}
In this case we will say that the morphism ${\varphi}':{\mathbb{B}'}\rightarrow \A$ is obtained from $\varphi$ by a \emph{vertex morphism}.

\begin{lemma}{\label{lemma:vertexmorphism}}
Let $\varphi:\mathbb{B}\rightarrow \A$ be a graph of groups morphism. If $ {\varphi}':{\mathbb{B}'}\rightarrow \A$ is obtained from $\varphi$ by a vertex morphism, then there exists a $\pi_1$-surjective graph of groups morphism $\sigma:\mathbb{B}\rightarrow  {\mathbb{B}'}$ such that $ {\varphi}'\circ \sigma =\varphi$.  
\end{lemma} 
\begin{proof}
Te assertion holds as $\varphi= \varphi' \circ \sigma$ where $\sigma:\mathbb{B}\rightarrow \mathbb{B}'$ is the $\pi_1$-surjective graph of groups morphism 
$$\sigma=(\sigma, \{\sigma_u\ | \ u\in VB\}, \{\sigma_f\ | \ f\in EB\},  \{o_f^{\sigma} \ | \ f\in EB\}, \{t_f^{\sigma} \ | \ f\in EB\})$$
where $\sigma=\text{Id}_{B}:B\rightarrow B$, $\sigma_x=\text{Id}_{B_x}:B_x\rightarrow B_x$ for all $x\in VB\cup EB-\{u\}$, $\sigma_u$ is the projection $B_u\rightarrow B_u'=B_u/\ker(\varphi_u)$ and where $o_f^{\sigma}=1$ for all $f\in EB$.
\end{proof}

Assume now  that $\varphi:\B\rightarrow \A$ is not folded because   condition (F1) is not satisfied. Thus  there exist distinct edges $f_1$ and $f_2$ of $\B$ with $u:=\alpha(f_1)=\alpha(f_2)$ and   $e:=\varphi(f_1)=\varphi(f_1)$ such that  $o_{f_2}^{\varphi} =\varphi_u(b) o_{f_1}^{\varphi} \alpha_e(c)$ for some  $b\in B_u$ and $c\in A_e$. Denote the vertex $\varphi(u)$ by $v$ and the vertices  $\omega(f_1)$ and $ \omega(f_2)$ by $x$ and $y$ respectively. Note that $x $ and $y$ are mapped onto the same vertex $w$ of $A$.

We will modify   $\varphi$ by means of auxiliary moves until we obtain a morphism in which  an elementary fold can be applied.  Let $\varphi':\mathbb{B}\rightarrow \A$ be the morphism obtained from $\varphi$ by an auxiliary move of type A2 applied  to the edge $f_2$ with element $b^{-1}\in B_u$ (Recall that we supposed that $B_f=1$ for all $f\in EB$ so that   auxiliary moves of type A2  do  not change the graph of groups $\B$,  see Remark~\ref{remark:auxmove}).  Hence,   $o_{f_2}^{\varphi}$ is replaced by  $$o_{f_2}^{\varphi'}= \varphi_u(b^{-1}) o_{f_2}^{\varphi}  = \varphi_u(b^{-1}) \varphi_u(b) o_{f_1}^{\varphi} \alpha_e(c)=  o_{f_1}^{\varphi}\alpha_e(c).$$ 
The edge element   $o_{f_1}^{\varphi'}$ is equal to $o_{f_1}^{\varphi}$.

Let further $\varphi'':\mathbb{B} \rightarrow \A$ be the morphism obtained from $\varphi'$ by an auxiliary move of type A1 applied  to the edge $f_2$ with element $c\in A_e$. Thus,  $o_{f_2}^{\varphi'}=o_{f_1}^{\varphi}\alpha_e(c)$ is replaced by $o_{f_2}^{\varphi''}=o_{f_1}^{\varphi}\alpha_e(c) \alpha_e(c^{-1})=o_{f_1}^{\varphi}$ and  $t_{f_2}^{\varphi'}=t_{f_2}^{\varphi} \in A_w$ is replaced by $t_{f_2}^{\varphi''}= \omega_e(c)  t_{f_2}^{\varphi'}=\omega_e(c)  t_{f_2}^{\varphi }$. Note that $o_{f_1}^{\varphi''}=o_{f_1}^{\varphi'}=o_{f_1}^{\varphi}$ and that $t_{f_1}^{\varphi''}=t_{f_1}^{\varphi'}=t_{f_1}^{\varphi}$.

If $x=\omega(f_1)=\omega(f_2)=y$, then define $\overline{\varphi}:\overline{\B}\rightarrow \A$ as the morphism obtained from $\varphi''$ by an elementary fold of type IIIA that identifies the edges $f_1$  and $f_2$. In this case we will  say that $\overline{\varphi}$ is obtained from $\varphi$ by \emph{a fold of type IIIA}. 
 
 \smallskip
 
If $x=\omega(f_1)\neq \omega(f_2)=y$ then,  possibly after  exchanging $f_1$ and $f_2$,  we can assume that $y$ is not the base vertex $u_0$ of $\B$. Let $\varphi''':\mathbb{B} \rightarrow \A$ be the morphism obtained from $\varphi''$ by an auxiliary move of type A0  applied to the vertex $y$ with element $g:= (t_{f_1}^{\varphi''})^{-1} t_{f_2}^{\varphi''} =(t_{f_1}^{\varphi})^{-1} \omega_e(c)  t_{f_2}^{\varphi}$. Thus   $t_{f_2}^{\varphi''}= \omega_e(c)  t_{f_2}^{\varphi} $ is replaced by 
$$t_{f_3}^{\varphi'''}= \omega_e(c)t_{f_2}^{\varphi}  g^{-1}= t_{f_1}^{\varphi},$$ and the vertex homomorphism $\varphi_y''=\varphi_y:B_y\rightarrow A_w$ is replaced by   $\varphi_y''' = i_g\circ \varphi_y:B_y\rightarrow A_w, $ where $i_g$ is the inner automorphism of $A_w$ determined by $g$.  

Now we can apply an elementary fold of type IA to $\varphi'''$ since $o_{f_2}^{\varphi'''}=o_{f_1}^{\varphi'''}=o_{f_1}^{\varphi}$ and $t_{f_2}^{\varphi'''}=t_{f_1}^{\varphi'''}=t_{f_1}^{\varphi}$.  Let $\overline{\varphi}:\overline{\B}\rightarrow \A$ be the graph of groups morphism obtained from $\varphi''':\mathbb{B} \rightarrow \A$ by an elementary fold of type IA that identifies  the edges $f_1$ and $f_2$. In this case we will  say that   $\overline{\varphi}:\overline{\mathbb{B}}\rightarrow \A$ is obtained from $\varphi$ by a \emph{fold  of type IA}.

In both cases we will say that the morphism \emph{$\varphi:\B\rightarrow \A$ folds onto  $\overline{\varphi}:\overline{B}\rightarrow \A$} or that $\overline{\varphi}$ is obtained from $\varphi$ by a \emph{fold}.  The following lemma follows  from Lemmas~\ref{lemma:auxmoves}, \ref{lemma:simpleadjustment} and \ref{lemma:fold}.
\begin{lemma}{\label{lemma:mainfold}} 
Let $\varphi:\mathbb{B}\rightarrow \A$ be a graph of groups morphism and $u_0\in VB$ with $\varphi(u_0)=v_0$. If $ \overline{\varphi} :\overline{\mathbb{B}}\rightarrow \A$ is obtained from $\varphi$ by a fold, then there exists a morphism $\sigma:\B\rightarrow  \overline{\B}$ such that (i) $\overline{\varphi}_{\ast}\circ \sigma_{\ast} =\varphi_{\ast}$ and  (ii)  $\sigma_{\ast}:\pi_1(\B,u_0)\rightarrow \pi_1(\overline{\B},\sigma(u_0))$ is an isomorphism.
\end{lemma}


\section{Small orbifolds  and decorated groups} 
In this section we will study  homomorphisms from an arbitrary finitely generated group into the   fundamental group of a small orbifold, where we say that an orbifold  is \emph{small} if  it is isomorphic to a Moebius band or to  $F(p_1,\ldots, p_{r})$, where  $F$ is   a sphere with  $q\geqslant1$  open disks removed and  where  $0\leqslant r \leqslant 2$  such that $r=2$  if  $q=1$.
\begin{figure}[h!]
\begin{center}
\includegraphics[scale=1]{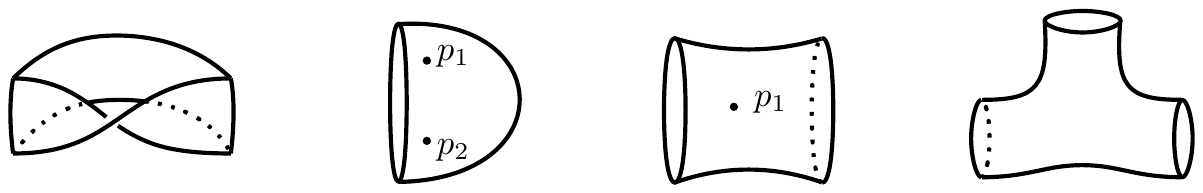}
\end{center}
\caption{Some small orbifolds.}
\end{figure}

For the remainder of this section let $\mathcal{O}$ be a small orbifold. Denote by   $\delta_1, \ldots, \delta_q$  the boundary components of $\mathcal{O}$.  As in the introduction, the element of $\pi_1^o(\mathcal{O})$  that corresponds to the curve   $\delta_i$ is   denote by  $t_i$.

Let $G$ be an arbitrary group and  $\eta$   a homomorphism from $G$ to $\pi_1^o(\mathcal{O})$. A cyclic subgroup  $C\leq G$ is called a \emph{peripheral subgroup of type $(o_C, i_{C})$},  where $o_C\in   \pi_1^o(\mathcal{O})$ and $i_C\in  \{1,\ldots, q\}$,    if $ \eta(C)$ is a non-trivial subgrup of $o_C \langle t_{i_C}\rangle  o_C^{-1}$ .   If   additionally $\eta(C)=o_C\langle t_{i_C}\rangle o_C^{-1}\cap \eta(G)$ then we say that $C$ is a \emph{maximal peripheral subgroup of type $ (o_C, i_C)$}.

\begin{definition}[Decorated groups]{\label{def:decgroup}}
 \emph{A  decorated group over $\pi_1^o(\mathcal{O})$} is a triple $(G,\eta, \{G_j\}_{j\in J})$, where  $G$ is a finitely generated group, $\eta$ is a homomorphism from $G$ to $\pi_1^o(\mathcal{O})$,  and  $\{G_j\}_{j\in J}$ is a finite set of peripheral subgroups of $G$ indexed by $J$ such that $G_j$ is of type  $(o_{G_j}, i_{G_j})$ for  $j\in J$. 
\end{definition}  

If the index set $J$ of  a decorated group   $(G, \eta, \{G_j\}_{j\in J})$  over $\pi_1^o(\mathcal{O})$ is equal to   $\{1,\ldots, n\}$   then we will also  denote $(G, \eta, \{G_j\}_{j\in J})$ by $(G, \eta, \{G_j\}_{1\leqslant j\leqslant n})$ or by $(G, \eta, \{G_1, \ldots, G_n\})$.

\begin{definition} A \emph{sub-decoration} of a decorated group $(G, \eta, \{G_j\}_{j\in J})$  is a decorated group of the type  $(G, \eta, \{G_{j'}\}_{j'\in J'})$ where $J'\subseteq J$ is a subset.    If $J' \varsubsetneq  J$ is a proper subset   then  $(G, \eta, \{G_{j'}\}_{j'\in J'})$ is called a  \emph{proper sub-decoration of $(G, \eta,  \{G_{j}\}_{j\in J})$}.  
\end{definition}

If the index  set $J$ of a decorated group $(G, \eta, \{G_j\}_{j\in J})$ is equal to $\{1,\ldots, n\}$ and if $J'\subseteq J$ is equal to $\{j_1', \ldots, j_m'\}$ then we will also  denote the sub-decoration $(G, \eta, \{G_{j'}\}_{j'\in J'})$ alternatively by $(G, \eta, \{G_{j_i'}\}_{1\leqslant i\leqslant m})$ or by $(G, \eta, \{G_{j_1'}, \ldots, G_{j_m'}\})$.

\begin{definition}{\label{def:equivdecgroups}}
Let $(G, \eta, \{G_j\}_{j\in J})$ and $(H, \lambda, \{H_k\}_{k\in K})$ be decorated groups over $\pi_1^o(\mathcal{O})$. We say that  $(G,\eta,  \{G_j\}_{j\in J})$ is  \emph{isomorphic  to}  $(H,\lambda,  \{H_k\}_{k\in K})$/$(G, \eta, \{G_j\}_{j\in J})$  \emph{projects onto} $(H,\lambda,  \{H_k\}_{k\in K})$, we write $$(G,\eta,  \{G_j\}_{j\in J})\cong (H, \lambda, \{H_k\}_{k\in K}) \ \  / \  \   (G,\eta,  \{G_j\}_{j\in J})\twoheadrightarrow (H, \lambda, \{H_k\}_{k\in K}),$$  if  there is an isomorphism/epimorphism  $\sigma:G\rightarrow H$   such that $\lambda \circ \sigma =\eta$,  and a   bijection $\tau:J \rightarrow K$ such that  for each index  $j\in J$ the following hold:  (i) $i_j:=i_{G_j}=i_{H_{\tau(j)}}\in \{1,\ldots, q\}$,  (ii) $\sigma(G_j)=h_jH_{\tau(j)}h_j^{-1}$  for some $h_j\in H_{\tau(j)}$, and   $o_{G_j}=\lambda(h_j)o_{H_{\tau(j)}} t_{i_j}^{z_j}$ for some integer  $z_j$.
\end{definition}

\begin{definition}
Let $(G, \eta, \{G_j\}_{j\in J})$   be decorated group  over $\pi_1^o(\mathcal{O})$  and $G_0$ a subgroup of $G$.   We say that  $(G, \eta, \{G_j\}_{j\in J})$   is $G_0$-\emph{collapsible} if the group  $G$ splits as $G_0\ast_{j\in J} G_j$.  If   $G_0\cap \ker(\eta)=1$, then  we say that $(G, \eta, \{G_j\}_{j\in J})$ is \emph{strongly $G_0$-collapsible}. We will say that   $(G, \eta, \{G_j\}_{j\in J} )$ is \emph{collapsible} (\emph{strongly collapsible}) if it is $G_0$-collapsible (strongly $G_0$-collapsible) for some subgroup $G_0$ of $G$.
\end{definition}

\begin{remark}{\label{remark:collapsible}}
Note that a decorated group $(G, \eta, \{G_j\}_{j\in J})$ over $\pi_1^o(\mathcal{O})$ is  isomorphic  to a (strongly) collapsible decorated group over $\pi_1^o(\mathcal{O})$ if, and only if,  there is a subgroup $G_0\leq G$ (with $G_0\cap \ker(\eta)=1$) and elements $g_{j}\in G$ for $j\in J$   such that  $G=G_0\ast_{j\in J} g_{j}G_{j}g_{j}^{-1}$. 
\end{remark}

\begin{example}[Decorated group induced by an orbifold covering]{\label{ex:01}}
If $\eta':\mathcal{O}'\rightarrow \mathcal{O}$ is an orbifold covering and $\pi_1^o(\mathcal{O}')$ is finitely generated (equivalently, $\mathcal{O}'$ has compact core), then $\eta$ defines a decorated group over $\pi_1^o(\mathcal{O})$ in the following way. Let    $\delta_{1}',\ldots, \delta_{n}'$  be the compact boundary components of $\mathcal{O}'$ and  $C_1',\ldots, C_n'$ the subgroups of $\pi_1^o(\mathcal{O}')$  that correspond  to the boundary components $\delta_1', \ldots, \delta_n'$ respectively.   Then $  (\pi_1^o(\mathcal{O}'), \eta_{\ast}', \{C_j'\}_{1\leqslant  j\leqslant n})$  
is a decorated group over $\pi_1^o(\mathcal{O})$.   We will say that $(\pi_1^o(\mathcal{O}), \eta_{\ast}', \{C_j'\}_{1\leqslant j\leqslant n})$ is  \emph{induced by the orbifold covering $\eta':\mathcal{O}'\rightarrow \mathcal{O}$}. Note that $(\pi_1^o(\mathcal{O}'), \eta_{\ast}', \{C_j'\}_{1\leqslant j\leqslant n})$ is well defined up to equivalence of decorated groups and that each $C_j'$ is a  maximal peripheral  subgroup of $\pi_1^o(\mathcal{O}')$.  
\end{example}

\begin{remark}{\label{rem:collapsible}}  
 The decorated group  $(\pi_1^o(\mathcal{O}'), \eta_{\ast}', \{C_j'\}_{1\leqslant j\leqslant n})$ is   strongly collapsible if  and only if the orbifold covering  $\eta'$ has infinite degree.  Also,  any  proper sub-decoration of $(\pi_1^o(\mathcal{O}'), \eta_{\ast}', \{C_j'\}_{1\leqslant j\leqslant n})$  is strongly collapsible.  
\end{remark}
 
We will need the following result  whose proof    follows immediately from the theory of orbifold coverings. 
\begin{lemma}{\label{lemma:inj}}
Let $(G, \eta, \{G_j\}_{j\in J})$ be a strongly collapsible  decorated group over $\pi_1^o(\mathcal{O})$.   If $\eta$ is injective, then there is an orbifold covering $\eta':\mathcal{O}'\rightarrow \mathcal{O}$   such that  $(G, \eta, \{G_j\}_{j\in J})$  is isomorphic to a sub-decoration  of the decorated group $(\pi_1^o(\mathcal{O}'), \eta_{\ast}', \{C_j'\}_{1\leqslant j\leqslant n})$ induced by $\eta'$.   
\end{lemma}

\begin{example}{\label{ex:03}}
Let $\eta':\mathcal{O}'\rightarrow \mathcal{O}$ be an almost orbifold covering with exceptional point $x$ of order $p\geqslant 1$, exceptional disk $D$ and exceptional boundary component $\delta'\subseteq \partial \mathcal{O}'$. Let further $d$ be an integer such that $1\leqslant d=|\langle s_x\rangle :\langle s_x^d\rangle |\leqslant |\langle s_x\rangle|=p$ (recall that $s_x$ is the element of $\pi_1^o(\mathcal{O})$ that is represented by  the curve $\partial D$).  Suppose that $\partial \mathcal{O}'= \delta_1'\cup \ldots \cup  \delta_n' \cup \delta'$.  We  have the  following decorated groups over $\pi_1^o(\mathcal{O})$.
\begin{enumerate}
\item[\textbf{Decorated group induced by  $\eta'$.}]  For each $j$  there is some $i_j\in \{1,\ldots, q\}$ such that $\eta'$ maps  $\delta_j'$ onto the boundary component $\delta_{i_j}$ of $\mathcal{O}$. Denote by $C_j'$ the subgroup of $\pi_1^o(\mathcal{O}')$  that corresponds to   $\delta_j'$. Then     $C_j'$ is a peripheral subgroup of $\pi_1^o(\mathcal{O}')$.  We say that the decorated group  $(\pi_1^o(\mathcal{O}'), \eta_{\ast}', \{C_j'\}_{1\leqslant j\leqslant n})$  is \emph{induced by the almost orbifold covering $\eta':\mathcal{O}'\rightarrow \mathcal{O}$}. Again $(\pi_1^o(\mathcal{O}'), \eta_{\ast}', \{C_j'\}_{1\leqslant j\leqslant n})$  is well defined up to equivalence of decorated groups.

\item[\textbf{Adjoining a finite subgroup.}]  The element represented by the exceptional boundary component $\delta'$ is mapped by $\eta_{\ast}'$ onto $gs_x^kg^{-1}$ for some $g\in \pi_1^o(\mathcal{O})$, where $k$ is the degree of the map $\eta'|_{\delta'}:\delta'\rightarrow \partial D$.  Let $$\lambda :\pi_1^o(\mathcal{O}')\ast g\langle s_x^d\rangle g^{-1} \rightarrow \pi_1^o(\mathcal{O})$$ be the homomorphism induced by   $\eta_{\ast}':\pi_1^o(\mathcal{O}')\rightarrow \pi_1^o(\mathcal{O})$  and the inclusion map $g\langle s_x^d\rangle g^{-1} \hookrightarrow\pi_1^o(\mathcal{O}).$  Then the decorated group   $(\pi_1^o(\mathcal{O}')\ast g\langle s_x^d \rangle g^{-1}, \lambda,  \{C_j'\}_{1\leqslant j\leqslant n} )$ 
is said to be obtained from the decorated group induced by $\eta'$ by \emph{adjoining a finite subgroup}.  Note that  if $d=|\langle s_x\rangle|$ then  $\langle s_x^d\rangle $ is trivial, and therefore  $(\pi_1^o(\mathcal{O}')\ast g\langle s_x^d \rangle g^{-1}, \lambda,  \{C_j'\}_{1\leqslant j\leqslant n})$ is equal to  the decorated group induced by $\eta'$.
\end{enumerate}
\end{example}

\begin{remark}
The decorated group  $(\pi_1^o(\mathcal{O}'), \eta_{\ast}', \{C_j'\}_{1\leqslant j\leqslant n})$ is strongly collapsible since $\mathcal{O}'$ has  non-empty boundary. Consequently  any   decorated group that is obtained from $(\pi_1^o(\mathcal{O}'), \eta_{\ast}', \{C_j'\}_{1\leqslant j\leqslant n})$  by adjoining a finite subgroup is   collapsible.
\end{remark}

\begin{definition} 
Let  $(G, \eta, \{G_j\}_{j\in J})$ be a strongly  $G_0$-collapsible decorated group  over $\pi_1^o(\mathcal{O})$. 
\begin{enumerate}
\item We say that $(G, \eta, \{G_j\}_{j\in J})$ \emph{folds peripheral subgroups} if there are  distinct elements    $k,  l\in J$  with  $i:= i_{G_k}=i_{G_l}$  such that  $o_{G_{l}}=\eta(g) o_{G_{k}} t_{i}^z$ for some integer $z$  and some $g\in G_0\ast_{j\in J-\{l\}} G_j$.

\item  We say that $(G, \eta, \{G_j\}_{j\in J})$  has an  \emph{obvious relation} if there is some $k\in J$  such that 
 $$0<|o_{G_k}\langle t_{i_{G_k}}\rangle o_{G_k}^{-1}:(\eta(\overline{G})\cap o_{G_k}\langle t_{i_{G_k}}\rangle o_{G_k}^{-1})|< |  o_{G_k}\langle t_{i_{G_k}}\rangle o_{G_k}^{-1}:  \eta(G_k)|$$
where $\overline{G}= G_0\ast_{j\in J-\{k\}} G_j$. 
\end{enumerate}
\end{definition}

For a  finitely generated group $G$ we define the \emph{torsion} of $G$, $\text{tn}(G)$, as the number of conjugacy classes of maximal finite  subgroups of $G$. For example, for a compact orbifold $\mathcal{O}=F(p_1, \ldots, p_r)$  we have $\text{tn}(\pi_1^o(\mathcal{O}))=r$.

\smallskip  
 
It is not hard to see that if a strongly  collapsible decorated group $(G, \eta, \{G_j\}_{j\in J})$  projects onto a  decorated group that  folds peripheral subgroups or  has an obvious relation,   or if $(G, \eta, \{G_j\}_{j\in J})$  is isomorphic to the a  decorated group   obtained from the decorated group induced by an almost orbifold covering by adjoining a finite subgroup,   then  clearly   the homomorphism $\eta:G\rightarrow \pi_1^o(\mathcal{O})$ is not injective. Next Proposition, which is the main ingredient to prove  Theorem~\ref{MainThm}, shows that a stronger  converse  of this remark   holds.

\begin{proposition}{\label{proposition:1}} Let $\mathcal{O}$ be an orientable small orbifold and  $(G,\eta, \{G_j\}_{j\in J})$  a strongly collapsible  decorated group  over $\pi_1^o(\mathcal{O})$. If the homomorphism $\eta:G\rightarrow\pi_1^o(\mathcal{O})$ is not injective, then  one of the following holds:  
\begin{enumerate}
\item[($\alpha$)] $(G,\eta, \{G_j\}_{j\in J})$ projects onto a  strongly collapsible decorated group $(H, \lambda, \{H_k\}_{k\in K})$ over $\pi_1^o(\mathcal{O})$ such that one of the following three conditions holds:
\begin{enumerate}
\item[($\alpha.1$)] $(H, \lambda, \{H_k\}_{k\in K})$ folds peripheral  subgroups.

\item[($\alpha.2$)] $(H, \lambda, \{H_k\}_{k\in K})$  has an obvious relation.

\item[($\alpha.3$)] $\rk(H)<\rk(G)$ or $\rk(H)=\rk(G)$ and $\text{tn}(H)>\text{tn}(G)$.
\end{enumerate}

\item[($\beta$)] There is a special almost orbifold covering $\eta':\mathcal{O}'\rightarrow \mathcal{O}$ such that $(G,\eta, \{G_j\}_{j\in J})$ is isomorphic   to a decorated group over $\pi_1^o(\mathcal{O})$ that is obtained from the decorated group induced by $\eta'$  by adjoining a finite subgroup.

\end{enumerate} 
\end{proposition}
The proof  the proposition is rather long and the main idea  is to identity the fundamental group of $\mathcal{O}$ with a graph of groups and then analyze the non-injectivity of $\eta$ using folds of graphs of groups morphisms.   To avoid a lengthy interruption of the proof of the main theorem we will postpone the   proof Proposition~\ref{proposition:1} to the end of the paper.


\section{Proof of Theorem~\ref{MainThm}}
Let $\mathcal{O}$ be a sufficiently large orbifold  and $\mathcal{T}_0$ a generating tuple of $\pi_1^o(\mathcal{O})$. We want to find a special almost orbifold covering $\eta':\mathcal{O}'\rightarrow \mathcal{O}$ and a generating tuple $\mathcal{T}'$ of $\pi_1^o(\mathcal{O}')$ such that $\eta_{\ast}'(\mathcal{T}')$ is Nielsen equivalent to $\mathcal{T}_0$.  The proof of the theorem is organized as follows:
\begin{enumerate}
\item In subsection 4.1 we will fix an identification between the fundamental group of $\mathcal{O}$ and the fundamental groups of a graph of groups $\A$ with cyclic edge groups. We  then  consider marked morphisms over $\A$.  This are triples $((\B, u_0), \varphi, \mathcal{T})$ consisting of a graph of groups $\B$, a morphism $\varphi:\B\rightarrow \A$ with special properties, and a generating tuple of the fundamental group of $\B$.  The idea is that the morphism $\varphi$  will ``approximate"   an  almost orbifold covering of $\mathcal{O}$. 

\item In subsection 4.2  we  consider the set  $\Omega(\mathcal{T}_0)$ consisting of all marked morphisms  $((\B, u_1), \varphi, \mathcal{T})$  with the property that $\varphi_{\ast}(\mathcal{T})$ is Nielsen equivalent to $\mathcal{T}_0$. Then we consider the set $\Omega_{\text{min}}(\mathcal{T}_0)$    of all marked morphisms  in $\Omega(\mathcal{T}_0)$  with minimal $c$-complexity.  

\item In subsection 4.3  we will prove, among other things, that if a marked morphism $((\B, u_0), \varphi, \mathcal{T})$ lies in $\Omega_{\text{min}}(\mathcal{T}_0)$ then $\varphi$ can not violate condition (F1) of the definition of folded morphisms.

\item In subsection 4.4 we show that there is at least one marked morphism    $((\B, u_0), \varphi, \mathcal{T})$   in $\Omega_{\text{min}}(\mathcal{T}_0)$ such that $\varphi$ is not vertex injecive.

\item In subsetction 4.5, using Proposition~\ref{proposition:1}, we show that a non-vertex injective marked morphism in  $\Omega_{\text{min}}(\mathcal{T}_0)$  that has the smallest  $d$-complexity gives rise to a special almost orbifold  covering $\eta':\mathcal{O}'\rightarrow \mathcal{O}$ and a tuple $\mathcal{T}'$ with the desired properties. 
\end{enumerate}
  

\subsection{Splitting the fundamental  group of a sufficiently large orbifold and marked  morphisms}
Let $\mathcal{O}=F(p_1,\ldots, r_{\mathcal{O}})$ be a sufficiently large orbifold. Then there  is  a non-empty (nonunique)  collection  $\gamma_1,\ldots, \gamma_t$  of disjoint simple closed curves on $\mathcal{O}$ such that the following hold:
\begin{enumerate}
\item[(1)] The closure of each component of    $\mathcal{O}- \gamma_1\cup\gamma_2\cup\ldots\cup \gamma_t$ 
is  a small orbifold. 

\item[(2)]  For each $1\leqslant i\leqslant t$,  the boundary components of a small regular neighborhood  $N(\gamma_i)\cong \gamma_i \times [-1,1]$  of $\gamma_i$  are contained in distinct components of $\mathcal{O}-\gamma_1\cup\gamma_2\cup\ldots\cup \gamma_t$.  
\end{enumerate}
\begin{figure}[h!]
\begin{center}
\includegraphics[scale=1]{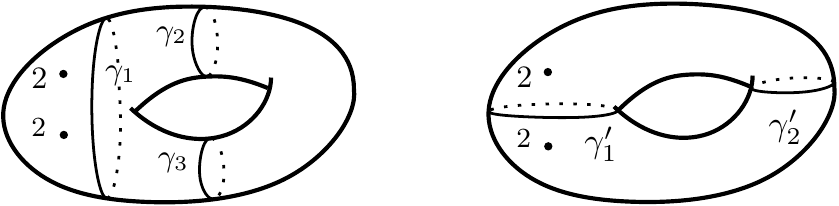}
\end{center}
\caption{Two splittings of $\mathcal{O}=T^2(2,2)$.}\label{fig:no1}
\end{figure}

For the remainder of this section    $\A=\A(\mathcal{O}, \gamma_1, \ldots, \gamma_t)$ will denote  the graph of groups with $\pi_1(\A,v_0)\cong \pi_1^{o}(\mathcal{O})$ corresponding to the above decomposition of $\mathcal{O}$, where  $v_0$ is the base vertex of $\A$.    Thus all vertex groups of $\A$ carry the fundamental group of a sub-orbifold (small orbifold) of $\mathcal{O}$, all edge groups of $\A$ are infinite cyclic,  and the images of the boundary monomorphisms correspond to the peripheral subgroups. More precisely,  
$A_v =\langle a_{v}, t_{v,1} \ | \ a_v^2=t_{v,1}\rangle $
if the vertex orbifold  $\mathcal{O}_v$ is isomorphic to a   Moebius band and 
$$A_v= \langle  s_{v,1},\ldots, s_{v, r_v}, t_{v,1}, \ldots, t_{v, q_v} \ | \ s_{v,1}^{p_{v,1}}, \ldots, s_{v, r_v}^{p_{v, r_v}},    s_{v,1}\cdot\ldots\cdot s_{v,r_{v}}=  t_{v,1}\cdot \ldots\cdot t_{v,q_{v}} \rangle$$
for the remaining cases, where:
\begin{enumerate}
\item[(a)] $r_v=r_{\mathcal{O}_v}=\text{tn}(\pi_1^o(\mathcal{O}_v))\geqslant 0$ is the number of cone points of $\mathcal{O}_v$.

\item[(b)] $p_{v,1}, \ldots, p_{v,r_v}$ are the oreders of the cone points  of $\mathcal{O}_v$.

\item[(c)] $q_v=q_{\mathcal{O}_v}\geqslant 1$ is the number of boundary components of $\mathcal{O}_v$.
\end{enumerate}
For each  $e\in EA$ the edge  group $A_e$ is  generated by the homotopy class $a_e$ of some  of the closed curves $\gamma_1,\ldots, \gamma_t$.  The boundary monomorphism $\alpha_e:A_e\rightarrow A_{\alpha(e)}$ is defined by $\alpha_e(a_e)=t_{\alpha(e), i_e}^{\varepsilon_e}$ for some $\varepsilon_e\in \{\pm 1\}$ and some  $i_e\in \{1,\ldots, q_{\alpha(e)}\}$. Moreover,    for any vertex $v\in VA$ and any   $1\leqslant i\leqslant q_v$ there is exactly one edge $e$  with  $\alpha(e)=v$  such that $i_e=i$.

\begin{definition}[Marked  Morphisms]{\label{def:marked}}
A \emph{marked morphism over $\A$}  is a triple $((\B, u_0), \varphi, \mathcal{T})$, where 
\begin{enumerate}
\item[(1)] $\B$ is a  connected graph of finitely generated  groups and $u_0$ is  a vertex of $\B$ such that  the fundamental group of $\B$  with respect to $u_0$ (hence with respect to any vertex) splits as a free product of cyclic groups.

\item[(2)]  $\varphi$ is a morphism from $\B$ to $\A$  with $\varphi(u_0)=v_0$  that has the following  properties:
\begin{enumerate}
\item[(2.i)] Edge homomorphisms are injective, that is,    $\varphi_f:B_f\rightarrow  A_{\varphi(f)}$  is injective for all $f\in EB$. 

\item[(2.ii)]   If $u\in VB$ is such that $\varphi_u:B_u\rightarrow A_{\varphi(u)}$ is injective, then there is an orbifold covering  
$$\eta_u':\mathcal{O}_u'\rightarrow \mathcal{O}_{\varphi(u)}$$ 
such that $(B_u, \varphi_u, \{\alpha_f(B_f)\}_{f\in \st_1(u, \B)})$  is isomorphic   to a  sub-decoration of the decorated group induced by $\eta_u'$, 
where $\st_1(u, \B):=\{f\in EB  \ | \ \alpha(f)=u \text{ and } B_f\neq 1\}$, see Fig.~\ref{fig:markedmorph}.

\item[(2.iii)] If $u\in VB$  is such that $\varphi_u:B_u\rightarrow A_{\varphi(u)}$ is not injective, then $\st_1(u, \B)\neq \emptyset$ and the decorated group  $(B_u, \varphi_u,\{\alpha_{f }(B_{f})\}_{f\in \st_1(u, \B)})$ 
is isomorphic  to a strongly collapsible decorated group  over  $A_{\varphi(u)}$.
\end{enumerate}

\item[(3)] $\mathcal{T}$ is a generating tuple of $\pi_1(\B, u_0)$. 
\end{enumerate} 

\begin{figure}[h]
\centering
\includegraphics{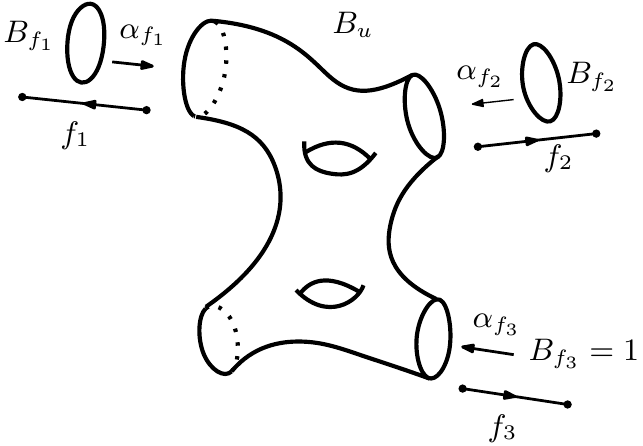}
\caption{$\st_1(u, \B)= \{f_1, f_2\}$ and $\st(u, B)=\{f_1, f_2, f_3\}$.}\label{fig:markedmorph}
\end{figure}  
\end{definition}  

\begin{remark}{\label{remark:type}}
Condition (2.i) implies  that  all edge groups  in $\B$  are either trivial or infinite cyclic.    Item (5) of the definition of graph of groups morphisms further implies that    $\varphi_u(\alpha_{f}(B_{f}))$ is a non-trivial subgroup of  $o_{f}\alpha_{\varphi(f)}(A_{\varphi(f)})o_{f}^{-1}$ which in turn is equal to  $o_{f} \langle t_{\varphi(\alpha(f)), i_{\varphi(f)}}\rangle o_{f}^{-1}$. Hence $\alpha_f(B_f)$ is a peripheral subgroup of $B_{\alpha(f)}$  of type   $(o_{f} , i_{\varphi(f)})$. 
\end{remark}

\begin{remark}{\label{remark:markedmorphisms}}
Let  $\delta_{u,1}', \ldots, \delta_{u,n_u}'$  be the compact boundary components of the orbifold $\mathcal{O}_u'$ given in item (2.ii),  and   $C_{u,j}'$ the corresponding  (maximal) peripheral subgroup of $\pi_1^o(\mathcal{O}_u')$.  The condition that   $(B_u, \varphi_u, \{\alpha_f(B_f)\}_{f\in \st_1(u, \B)})$ is isomorphic   to a  sub-decoration of the decorated group induced by $\eta_u'$  means that  there is a  possibly empty  subset $\{k_{u, 1}, \ldots, k_{u, {m_u}}\}$ of $\{1,\ldots, n_u\}$ such that  
$$(B_u, \varphi_u, \{\alpha_{f}(B_{f})\}_{f\in \st_1(u, \B)}) \cong (\pi_1^o(\mathcal{O}_u'), (\eta_u')_{\ast}, \{C_{u,k_{u,j}}'\}_{1\leqslant j\leqslant m_u}).$$
\end{remark}

\begin{remark}{\label{remark:markedmorphisms1}}
Item (2.iii) is equivalent to saying  that  there are elements $b_f\in B_u$, for $f\in \st_1(u, \B)$,  and  a subgroup $C_u$ of $ B_u$, which intersects $\ker(\varphi_u)$  trivially, such that the vertex group  $B_u$  splits as $$B_u=C_u\ast_{f\in \st_1(u, \B)} b_f \alpha_{f}(B_f) b_f^{-1}.$$ 
\end{remark}

\begin{example}{\label{ex:induced}} An  almost orbifold covering $\eta':\mathcal{O}'\rightarrow \mathcal{O}$  naturally induces  a  morphism over $\A$  for which conditions  (2.i)-(2.iii) are satisfied. Indeed, the inverse image of the    simple closed curves $\gamma_1, \ldots, \gamma_t$ under $\eta'$ are  simple closed curves $ \gamma_1', \ldots, \gamma_{t'}' $ in $\mathcal{O}'$.   Let $\B=\B(\mathcal{O}', \gamma_1', \ldots, \gamma_{t'}')$ be the graph of groups with $\pi_1(\B, u_0)\cong \pi_1^o(\mathcal{O}')$ corresponding to the decomposition of $\mathcal{O}'$ along  $\gamma_1',\ldots, \gamma_{t'}'$. Let further   $\varphi:B\rightarrow A$ be  the  graph-morphism     induced by $\eta'$ (that is, for each vertex  $u$ of $B$,   $\varphi(u)$ is the vertex of $A$  such that $\eta'(\mathcal{O}_u')=\mathcal{O}_{\varphi(u)}$, and  for each  edge  $f$ of $B$, $\varphi(f)$ is the edge of $A$ such that $\eta'(\gamma_f')=\gamma_{\varphi(f)}$). 
\begin{figure}[h!]
\begin{center}
\includegraphics[scale=1]{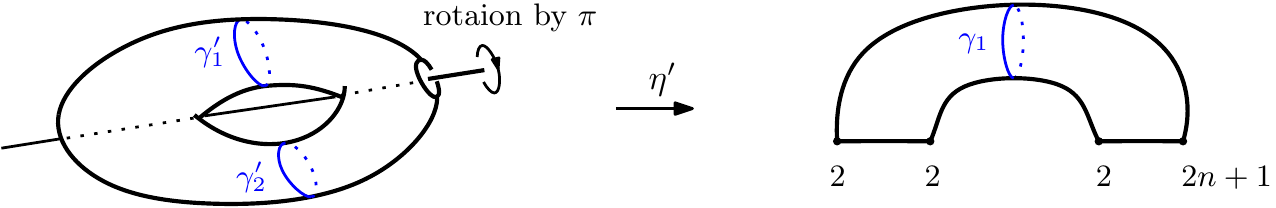}
\end{center}
\caption{The exceptional point of $\eta'$  is the cone point of order $2n+1$.}\label{fig:exinduced}
\end{figure}

Then    $\varphi:\B\rightarrow \A$ is defined    by 
  $ \varphi=(\varphi, \{\varphi_u \ | \ u\in VB\}, \{\varphi_f \ |  \  f\in EB\}, \{o_f \ | \ f\in EB\}, \{t_f \ | \ f\in EB\})$ 
where:
\begin{enumerate}
\item $\varphi_u:B_u\rightarrow A_{\varphi(u)}$ is induced by $\eta'|_{\mathcal{O}_u'}:\mathcal{O}_u'\rightarrow \mathcal{O}_{\varphi(u)}$ for all $u\in VB$.

\item  $\varphi_f:B_f\rightarrow A_{\varphi(f)}$ is induced by $\eta'|_{\gamma_f'}:\gamma_f'\rightarrow \gamma_{\varphi(f)}$ for all $f\in EB$. 

\item  $o_f\in A_{\alpha(\varphi(f))}$ is such that $\varphi_{\alpha(f)}(\alpha_f(B_f)) = o_f \alpha_{\varphi(f)}(\varphi_f(B_f)) o_f^{-1}$ for all $f\in EB$.  
\end{enumerate}
Fig.~\ref{fig:exinduced1} shows the   morphism induced by the almost orbifold covering $\eta$ depicted in Fig~\ref{fig:exinduced}. Notice that we   can always assume that  the exceptional point of $\eta$ is contained in the interior of some vertex orbifold $\mathcal{O}_v\subseteq \mathcal{O}$.  Thus there is a vertex $u'$ of $\B$ such that $\eta'|_{\mathcal{O}_{u'}'}:\mathcal{O}_{u'}'\rightarrow \mathcal{O}_{\varphi(u)}$ is an almost orbifold covering and hence  the vertex homomorphism $\varphi_{u'}$ is not injective. Therefore the morphism   $\varphi$ is not vertex injective.    
\begin{figure}[h!]
\begin{center}
\includegraphics[scale=1]{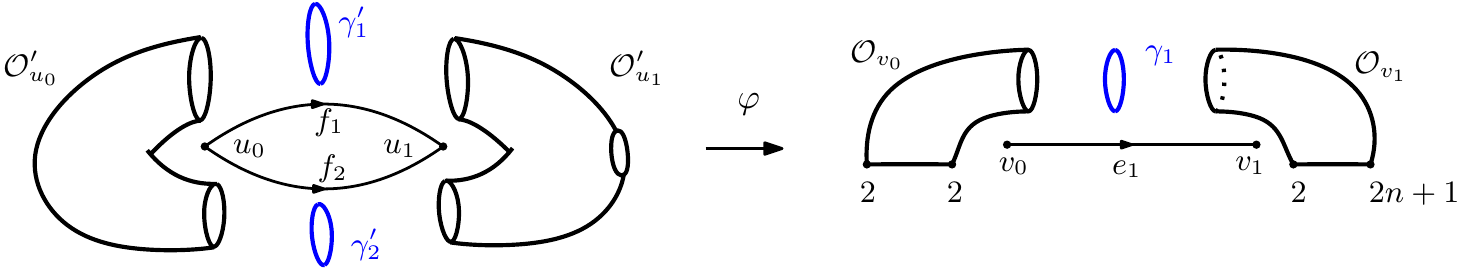}
\end{center}
\caption{The   morphism induced by $\eta'$  where $o_{f_1}=1$, $t_{f_1}=1$,  $o_{f_2}=s_{v_1, 1}$ and $t_{f_{2}}=s_{v_2,1}$.}\label{fig:exinduced1}
\end{figure}
\end{example}

\begin{lemma}{\label{lemma:possinlefolds}}
Let $\M$ be a marked morphism over $\A$. Suppose  that  $\varphi$ is vertex injective. If $\varphi$  is not folded,   then one of the following conditions holds: 
\begin{enumerate}
\item[($\overline{ F1}$)] There are distinct  edges $f$ and $g$  starting at the same vertex  $u$  with  $e:=\varphi(f)=\varphi(g)$  such that  $B_f=1$ or $B_g=1$ and  that $o_g =\varphi_u(b) o_f  \alpha_e(c)$  for some $b\in B_u$ and   $c\in A_e$.

\item[($\overline{F2}$)] There is an edge $f\in EB$ with $B_f=1$ such that $\alpha_{\varphi(f)}^{-1}( o_f^{-1}  \varphi_{\alpha(f)}(B_{\alpha(f)})  o_f)\leq A_{\varphi(f)}$
is non-trivial. 
\end{enumerate} 
\end{lemma}
\begin{proof}
As $\varphi$ is not   is not folded,  one of the conditions (F0)-(F2) of the definition of folded morphisms is violated. By hypothesis $\varphi$ is vertex injective, and hence one of the following holds:
\begin{enumerate}
\item[(F1)] There are distinct edges $f$ and $g$   with $u:=\alpha(f)=\alpha(g)$ and  $e:=\varphi(f)=\varphi(g)$ such that   $o_g=\varphi_u(b) o_f \alpha_e(c)$  for some $b\in B_u$ and $c\in A_e$. 

\item[(F2)] There is an edge $f\in EB$  such that
 $\varphi_f(B_f)$ is a proper subgroup of $  \alpha_{\varphi(f)}^{-1}(o_f^{-1}  \varphi_{\alpha(f)}(B_{\alpha(f)})  o_f).$  
\end{enumerate}

Since $\varphi_u$ is injective it follows from   Condition (2.ii)  that for distinct edges   $f$ and $g$ in $\st_1(u, \B)$,   the peripheral subgroups $\alpha_f(B_f)$ and $\alpha_g(B_g)$ of $B_u$ correspond  to  distinct boundary components of the orbifold  $\mathcal{O}_u'$.  This implies that: (i) the $(\varphi_u(B_u), \alpha_e(A_e))$-double cosets $\varphi_u(B_u) o_g  \alpha_e(A_e) $ and $\varphi_u(B_u) o_f \alpha_e(A_e) $  are distinct,  and (ii) $\varphi_f(B_f)= \alpha_e^{-1}(o_f^{-1} \varphi_u(B_u) o_f)$. Therefore if $\varphi$ is not folded then either $(\overline{ F1})$ occurs or $(\overline{F2})$ occurs.
\end{proof}

For a graph of groups $\B$, we denote by $|\B|_c$ the number of edge pairs of the underlying graph of $\B$ that have non-trivial edge group, that is, $|\B|_c:=\frac{1}{2}|\{f\in EB \ | \ B_f\neq 1\}|.$ 
\begin{lemma}{\label{lemma:schenitzer}}
Let $\M$   be a marked  morphism over $\A$.  If  $|\B|_c\geqslant 1$, then there exists a  vertex $u$ of $\B$ with  $\st_1(u, \B)\neq \emptyset$  such that the  decorated group $(B_u, \varphi_u, \{\alpha_{f}(B_{f})\}_{f\in \st_1(u, \B)})$ is isomorphic  to a strongly collapsible decorated group over $A_{\varphi(u)}$. 
\end{lemma}
\begin{proof}
Suppose that the result does not hold.  Denote by $B_1$ the sub-graph of $B$ having vertex set $V_1:=\{u\in VB \ | \ \st_1(u, \B) \text{ is non-empty}\}$
and edge set $E_1:=\{f\in EB \ | \ B_f\neq 1\}$.  Choose a component  $B'\subseteq B_1$  of  $B_1$ and a vertex $u_0'\in VB'$.  Denote by  $\B'=\B(B')$ the sub-graph of groups of $\B$ carried by the   sub-graph $B'\subseteq B$. Since $\pi_1(\B,u_0')\cong \pi_1(\B,u_0)$ splits as a free product of cyclic groups,  Kurosh subgroup Theorem~\cite{Kurosh} implies  that  the subgroup $\pi_1(\B',u_0')\leq \pi_1(\B, u_0')$   splits as a free product of cyclic groups.

On the other hand, let $\varphi'$ be  the restriction of the morphism  $\varphi$ to $\B'$. Then $\varphi':\B'\rightarrow \A$ satisfies conditions (2.i)-(2.iii) of Definition~\ref{def:marked}. Since the decorated group 
$$(B_u', \varphi_u', \{\alpha_f'(B_f')_{f\in \st_1(u, \B')})= (B_u, \varphi_u, \{\alpha_f(B_f)_{f\in \st_1(u, \B)})$$ is not isomorphic to a strongly collapsible decorated group,  it follows that the following hold: (i) $\varphi_u'$ is injective,  (ii)   the orbifold covering $\eta_u:\mathcal{O}_u'\rightarrow \mathcal{O}_{\varphi(u)}$ given in (2.ii)  has finite degree, and (iii) $(B_u', \varphi_u', \{\alpha_f'(B_f')\}_{f\in \st_1 (u, \B')})$ is isomorphic  to the decorated group induced  by $\eta_u$.  Therefore   $\pi_1(\B', u_0')$ is isomorphic to the fundamental group of a closed $2$-orbifold, and so cannot split as a free product of cyclic groups.  This contradiction completes the proof of the lemma. 
\end{proof} 
 
 \smallskip

Now we define the $c$-complexity and the $d$-complexity of a marked morphism.   Denote the rank of $\pi_1(\B,u_0)$ by $\text{rk}(\B)$ and the torsion of $\pi_1(\B, u_0)$ by $\text{tn}(\B)$. \emph{Throughout the paper  we will always assume that the set $\mathbb{N}_0^k= \mathbb{N}_0\times \ldots \times \mathbb{N}_0 $ is equipped  with the lexicographic order. }

\begin{definition}{\label{def:complexity}} Let $\M$ be a marked morphism over $\A$ . \begin{enumerate}
\item[(\emph{$c$-complexity})] We define the $c$-complexity of $\M$ as the  tuple 
$$c\M:=(\text{rk}(\B), \text{rk}(\B)-\text{tn}(\B), |EB|).$$

\item[(\emph{$d$-complexity})] We define the $d$-complexity of   $\M$ as the   pair 
$$d\M:=(|\B|_c, c_E(\varphi))$$ 
where  $c_E(\varphi):=\frac{1}{2}\sum_{f\in EB, B_f\neq 1} |A_{\varphi(f)}:\varphi_f(B_f)|.$ 
\end{enumerate} 
\end{definition}
\begin{remark}{\label{remark:complexity}}
Note that the $c$-complexity takes into account only the graph of groups in a marked morphism. More precisely,  if $\M$ and $((\B', u_0'), \varphi', \mathcal{T}')$  are marked morphism  such that the underlying graphs and the fundamental groups  of $\B$ and $\B'$ are isomorphic, then they have the same $c$-complexity. On the $d$-complexity we see that the  $c_E$ factor is related to the morphism $\varphi$ in $\M$. 
\end{remark}

 
\subsection{The sets $\Omega(\mathcal{T}_0)$ and $\Omega_{\text{min}}(\mathcal{T}_0)$}  Using the splitting of $\pi_1^o(\mathcal{O})$ as the fundamental group of the graph of groups $\A$ we can think  of $\mathcal{T}_0=(g_1, \ldots, g_m)$ as a tuple in $\pi_1(\A, v_0)$. The idea is to use marked morphisms over $\A$ to get a better understanding of  the Nielsen equivalence class of $\mathcal{T}_0$. Thus we  consider the set, $\Omega(\mathcal{T}_{0})$,  of  all marked morphisms $\M$  over $\A$  with the property that the tuple   $\varphi_{\ast}(\mathcal{T})$ is Nielsen equivalent to $\mathcal{T}_0$.

\begin{remark}
The condition that  $\varphi_{\ast}(\mathcal{T})$ and  $\mathcal{T}_0$ are Nielsen equivalent implies that the  morphism  $\varphi:\B\rightarrow \A$ is $\pi_1$-surjective,  and that ${size}(\mathcal{T}_0)={size}(\mathcal{T})\geqslant \text{rk}(\B)$. By definition,   $\pi_1(\B, u_0)$ splits as $G_1\ast \ldots \ast G_{\text{rk}(\B)}$ where $G_1,\ldots, G_{\text{rk}(\B)}$ are cyclic subgroups of $\pi_1(\B, u_0)$.    Consequently, by  applying   Grushko's theorem  we see that  $\mathcal{T}$ is reducible  if $\text{rk}(\B)<size(\mathcal{T})$.  
\end{remark}

\begin{remark}{\label{remark:notfolded}}
Notice that if $\M$ is a marked morphism, then the homomorphism $\varphi_{\ast}$ induced by $\varphi$ is not injective. This  follows from the fact that  $\pi_1(\A, v_0)$ is isomorphic to the fundamental group of a closed orbifold, and therefore does not  split as a  free product of cyclic groups.  As a consequence we see  that $\varphi$ is not folded.
\end{remark}

We first observe that the set  $\Omega(\mathcal{T}_0)$ is non-empty.  Indeed,  for each $1\leqslant i\leqslant m$, we can  represent the element $g_i$ by a non-necessarily reduced   $\A$-path $p_i=a_{i,0},e_{i,1},a_{i,1},\ldots, a_{i,k_i-1},e_{i,k_i}, a_{i,k_i}$ 
from $v_0$ to $v_0$ of positive length. Then the marked  morphism $\M$  shown in Fig.~\ref{fig:swedge}   belongs to $\Omega(\mathcal{T}_0)$, where $\mathcal{T}=(x_1,\ldots, x_m)$  and  where $x_i$ is represented by the $\B$-path $q_i$.  
\begin{figure}[h!]
\begin{center}
\includegraphics[scale=1]{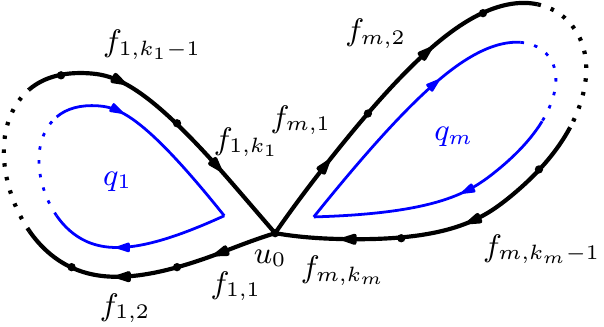}
\end{center}
\caption{Vertex and edge groups of $\B$ are trivial.  $\varphi:\B\rightarrow \A$ maps the edge $f_{i,j}$ onto $e_{i,j}$, and for each $1\leqslant i\leqslant m$,  $o_{f_{i,1}}^{\varphi}=a_{i,0}$  and $ o_{f_{i,j}}^{\varphi}=1$ for   $1 \leqslant j\leqslant  k_i$,  $t_{f_{i,j}}^{\varphi}=a_{i,j}$ for $1\leqslant j\leqslant k_i$.}\label{fig:swedge}
\end{figure} 
 
\smallskip 
 
Now we define some modifications on elements of $\Omega(\mathcal{T}_0)$.  We start with  a  partial vertex morphism. Let $\M$ be a marked morphism  in $\Omega(\mathcal{T}_0)$ and $u$ a vertex of $\B$ with $v:=\varphi(u)$.  Suppose that $\st_1(u, \B)$ is non-empty and that  the decorated group $(B_u, \varphi_u, \{\alpha_f(B_f)\}_{f\in \st_1(u, \B)})$ is strongly collapsible.  Suppose further  that  $(H, \lambda , \{H_k\}_{k\in K})$ is a strongly collapsible decorated group  over $A_v$  such that 
$$(B_u, \varphi_u, \{\alpha_f(B_f)\}_{f\in \st_1(u, \B)})\twoheadrightarrow (H, \lambda , \{H_k\}_{k\in K}).$$ 
We will define a new marked morphism  $((\B', u_0), \varphi', \mathcal{T}')$  over $\A$ that still  belongs to $\Omega(\mathcal{T}_0)$ as follows.

By definition, there is an epimorphism $\psi:B_u\rightarrow H$ such that $\lambda\circ \psi =\varphi_u$, and there is  a bijection  $\tau:\st_1(u, \B)\rightarrow K$ such that for each $f\in \st_1(u,\B)$ the following hold: (i)  $i_f:= i_{\varphi(f)}=i_{H_{\tau(f)}}\in \{1,\ldots,  q_{v}\}$, (ii) $\psi( \alpha_{f}(B_{f})) =a_f H_{\tau(f)}a_f^{-1}$ for some $a_f\in H_{\tau(f)}$, and  (iii) $o_{f}=\lambda(a_f) o_{H_{\tau(f)}}t_{i_f}^{z_f}$ for some integer $z_f$.  
 
Let $\B'$ be the graph of groups   obtained from $\B$ by replacing the vertex group $B_u$ by the group $B_u':=H$,   and  by replacing  the boundary monomorphism $\alpha_{f}:B_{f} \rightarrow B_u$   by  the monomorphism  $\alpha_{f}':B_{f}'= B_{f}=\langle b_{f}\rangle \to B_u'$ defined by  $ \alpha_{f}'(b_{f})= h_{\tau(f)}$, where $h_{\tau(f)}$ is a generator of the peripheral subgroup $H_{\tau(f)}\leq H=B_u'$. 

Note that $\B$ and $\B'$ have the same underlying graphs. Moreover, the fundamental group of $\B'$  splits as a free product of cyclic groups as   $\B'$ and $\B$ only differ  at the vertex $u$,  and at this vertex   the decorated groups $(B_u, \varphi_u, \{\alpha_f(B_f)\}_{f\in \st_1(u, \B)})$ and  $(B_u', \varphi_u', \{\alpha_f'(B_f')\}_{f\in \st_1(u, \B')})$ are   strongly collapsible.

Let  $\varphi':\B'\rightarrow \A$ be the  morphism  obtained from $\varphi$ by replacing the vertex homomorphism  $\varphi_u:B_u\rightarrow A_{v}$ by the homomorphism $\varphi_u':=\lambda:B_u'\rightarrow A_{v}$, and  by replacing the edge element   $o_{f}$ by   $o_{f}^{\varphi'}:=o_{H_{\tau(f)}} \in  A_{v}$. 
\begin{claim}
The morphism $\varphi'$ satisfies conditions (2.i)-(2.iii).
\end{claim}
\begin{proof}[Proof of Claim]
We only need to look at the vertex $u$ as the reaming vertices do not change. If $\varphi_u'$ is not injective,  then (2.iii) is satisfied since the decorated group  $(B_u', \varphi_u', \{\alpha_f'(B_f')\}_{f\in \st_1(u, \B')})$ is strongly collapsible.   If $\varphi_u'$ is injective,  then    Lemma~\ref{lemma:inj}  implies  that $(B_u', \varphi_u', \{\alpha_{f }'(B_{f }')\}_{f\in \st_1(u, \B')})$ is isomorphic to  a sub-decoration of the decorated group induce by some orbifold covering over the vertex orbifold $\mathcal{O}_v$, that is, (2.ii) holds at the vertex $u$.   
\end{proof}

\begin{claim}
There is a $\pi_1$-surjective morphism $\sigma:\B\rightarrow \B'$ such that $\varphi'\circ \sigma =\varphi$.
\end{claim}
\begin{proof}[Proof of Claim]
Define $\sigma:\B\rightarrow \B'$  by 
$$\sigma=(\sigma, \{\sigma_u\ | \ u\in VB\}, \{\sigma_f\ | \ f\in EB\},  \{o_f^{\sigma} \ | \ f\in EB\}, \{t_f^{\sigma} \ | \ f\in EB\})$$
where $\sigma=\text{Id}_{B}:B\rightarrow B$,   vertex homomorphisms  are given by  $\sigma_w=\text{Id}_{B_w}:B_w \rightarrow B_w=B_w'$  if  $w\neq u$ and $\sigma_u=\psi:B_u\rightarrow B_u'$, edge homomorphisms   are given by  $\sigma_f=\text{Id}_{B_f}:B_g\rightarrow B_{f}$  fr all $f\in EB$, and where edge elements are given by  $o_g^{\sigma}=1$ for  $g \notin \st_1(u, \B)$  and $o_{g}^{\sigma}:=a_{\tau(g)}\in B_u'$ for $g\in \st_1(u, \B)$. 
\end{proof} 

Put $\mathcal{T}':=\sigma_{\ast}(\mathcal{T})$. From the previous claim we see that $\varphi_{\ast}'(\mathcal{T}')$ is equal to $\varphi_{\ast}(\mathcal{T})$. The  tuple $\varphi_{\ast}(\mathcal{T})$  is Nielsen equivalent to $\mathcal{T}_0$ and so $\varphi_{\ast}(\mathcal{T}')$ is also Nielsen equivalent  to $\mathcal{T}_0$.    Thus  $((\B', u_0), \varphi', \mathcal{T}')$ is a marked  morphism    over $\A$   that belongs to $\Omega(\mathcal{T}_0)$. We will say that  $((\B', u_0), \varphi', \mathcal{T}_0')$  is obtained from $\M$ by \emph{replacing the decorated group $(B_u, \varphi_u, \{\alpha_f(B_f)\}_{f\in \st_1(u, \B)})$ by the decorated group   $(H, \lambda, \{H_k\}_{k\in K})$}.

Since $\sigma_{\ast}:\pi_1(\B, u_0)\rightarrow \pi_1(\B', u_0)$ is surjective, and $\B$ and $\B'$ have the same underlying graphs,  it follows that
$$c   ((\B', u_0), \varphi', \mathcal{T}') \leqslant c\M.$$
Furthermore, as edge groups/homomorphisms do  not change, we further conclude that the $d$-complexity does not change. We will need the following lemma.
\begin{lemma}{\label{lemma:pvm}}
If $\rk(H)<\rk(B_u)$  or if $\rk(H)=\rk(B_u)$ and  $\text{tn}(H)>\text{tn}(B_u)$, then  
$$ c((\B', u_0'), \varphi', \mathcal{T}') <c\M.$$  
\end{lemma} 
\begin{proof}
Let $C_u\leq B_u$ and $C_u'\leq B_u'$ such that  $(B_u, \varphi_u, \{ \alpha_f(B_f)\}_{f\in \st_1(u, \B)})$ is strongly $C_u$-collapsible and that    $(B_u', \varphi_u', \{ \alpha_f'(B_f')\}_{f\in \st_1(u, \B')})$  is strongly $C_u'$-collapsible.  The rank and the torsion  of the groups $B_u$ and $B_u'$ are given by:
 \begin{enumerate}
 \item $\text{rk}(B_u)=\text{rk}(C_u)+n$ and $\text{tn}(G)=\text{tn}(G_0)$;

 \item $\text{rk}(B_u')=\text{rk}(C_u')+n$ and $\text{tn}(B_u')=\text{tn}(C_u')$;
\end{enumerate}
where   $n=|\st_1(u, \B)|=|\st_1(u, \B')|$. Consequently, we either have   $\rk(C_u)<\rk(C_u')$ or  $\rk(C_u)=\rk(C_u')$ but  $\text{tn}(C_u)<\text{tn}(C_u')$.  Now, the rank and the torsion of $\B$ and $\B'$ are  related   by 
 $$\text{rk}(\B)-\text{rk}(\B')=\text{rk}(C_u)-\text{rk}(C_u') \ \ \text{ and } \  \ \text{tn}(\B)-\text{tn}(\B')= \text{tn}(C_u)-\text{tn}(C_u').$$
Therefore   $c((\B', u_0), \varphi', \mathcal{T}')$ is smaller than $c\M$, which completes the proof of the lemma. 
\end{proof}

Let  $(n_1, n_2, n_3)\in \mathbb{N}_0^3$  be the minimum of the  $c$-complexity  map $c:\Omega(\mathcal{T}_0)\rightarrow \mathbb{N}_0^3$. Denote the subset 
$$c^{-1}(n_1, n_2, n_3)\subseteq \Omega(\mathcal{T}_0)$$ 
of $\Omega(\mathcal{T}_0)$ by    $\Omega_{\text{min}}(\mathcal{T}_0)$, that is, $\Omega_{\text{min}}(\mathcal{T}_0)$ consists of all marked morphisms in $\Omega(\mathcal{T}_0)$  that have minimal $c$-complexity.

\smallskip

Now we describe the effect of auxiliary  moves  on  marked  morphisms. All the moves   take as input an element of $\Omega_{\text{min}}(\mathcal{T}_0)$   and  give as output  a new  element in $\Omega_{\text{min}}(\mathcal{T}_0)$. Let $\M \in \Omega_{\text{min}}(\mathcal{T}_0)$.  

 \smallskip 
 
\noindent\textit{Auxiliary moves of type A0 and A1.} Suppose that the morphism  $\varphi':\B \rightarrow \A$ is obtained from $\varphi$ by an auxiliary move of type A1  or by an auxiliary move  of type A0  that is $u_0$-admissible.  By Lemma~\ref{lemma:auxmoves},  the induced homomorphisms   $\varphi_{\ast}'$ and $\varphi_{\ast}$  coincide, and  so  $\varphi_{\ast}'(\mathcal{T})=\varphi_{\ast}(\mathcal{T})$.   We need to check that  conditions (2.i)-(2.iii) hold. This follows easily from the description of these moves   since edge groups/homomorphisms do not   change (edge groups in $\A$ are abelian)  and for each vertex $u$ of $\B$ we have    
$$(B_u, \varphi_u', \{\alpha_{f}(B_f)\}_{f\in \st_1(u, \B)})\cong  (B_u, \varphi_u, \{\alpha_{f}(B_f)\}_{f\in \st_1(u, \B)}).$$  
We  will say that $((\B, u_0), \varphi', \mathcal{T})$ is obtained from $\M$  by an auxiliary move of type A0 or by an auxiliary move of type A1, accordingly.  Clearly    $((\B, u_0), \varphi', \mathcal{T})$  belongs to $\Omega_{\text{min}}(\mathcal{T}_0)$.

\smallskip   
  
\noindent\textit{Auxiliary moves of type A2.} Suppose that the morphism $\varphi':\B'\rightarrow \A$ is obtained from $\varphi$ by an auxiliary move of type A2.   Lemma~\ref{lemma:auxmoves} gives a graph of groups isomorphism $\sigma:\B\rightarrow \B'$ such that $\varphi'\circ \sigma =\varphi$. Thus $\pi_1(\B', \sigma(u_0))$ splits as a free product of cyclic group and  the tuple   $\mathcal{T}':=\sigma_{\ast}(\mathcal{T})$  generates  $\pi_1(\B',\sigma( u_0))$.  To conclude that $((\B', u_0), \varphi', \mathcal{T}')$ is indeed a marked morphism  we need to  show that $\varphi'$ satisfies  conditions (2.i)-(2.iii) of Definition~\ref{def:marked}.   But this is trivial  since  the  description of the  fold  implies that   
$$ (B_{\sigma(u)}', \varphi_{\sigma(u)}', \{\alpha_{f}'(B_f')\}_{f\in \st_1(\sigma(u), \B')})\cong (B_u, \varphi_u, \{\alpha_{f}(B_f)\}_{f\in \st_1(u, \B')})$$
for all vertices of of $\B$.   We will say that the marked morphism   $((\B', u_0), \varphi', \mathcal{T}')$  is  obtained from $\M$ by \emph{an auxiliary move of type A2}. Clearly  $((\B', u_0), \varphi', \mathcal{T}')$  belongs to $\Omega(\mathcal{T}_0)$.  As $\B'$ and $\B$ have isomorphic fundamental groups and  the same underlying graphs, we further see that    $((\B', u_0), \varphi', \mathcal{T}')$ belongs to $\Omega_{\text{min}}(\mathcal{T}_0)$.

Finally notice that the $d$-complexity is invariant under auxiliary moves since  edge groups and edge homomorphisms do not change  in any of these moves.


\subsection{Some properties of the  elements   of    $\Omega_{\text{min}}(\mathcal{T}_0)$} Let $\M$ be an element of  $\Omega_{\text{min}}(\mathcal{T}_0)$  and  $u$ a vertex of $\B$.  Suppose that the decorated group $(B_u, \varphi_u, \{\alpha_{f }(B_{ f})\}_{f\in \st_1(u, \B)})$ is strongly collapsible.    Thus $B_u$ splits as $B_u = C_u \ast_{f\in \st_1(u,\B)}  \alpha_{f}(B_{f})$  where $C_u$ is a subgroup of $B_u$  that has trivial intersection with the kernel of $\varphi_u$.   Let $g$ be an edge  starting at $u$ with non-trivial group.   We construct  a new  marked morphism   as follows.    

Let  $\B'$  be the  graph of groups obtained  from $\B$ by replacing the vertex group $B_u$ by the group  
$$B_u':=C_u\ast_{f\in \st_1(u, \B)-\{g\}}  \alpha_{f}(B_{f}),$$  
and  by  replacing the  edge group $B_{g}$ by the  trivial group $B_{g}':=1$.

Let $\varphi':\B'\rightarrow \A$ be the morphism   obtained from $\varphi:\B\rightarrow \A$ by    replacing the vertex homomorphism $\varphi_u: B_u\rightarrow A_{\varphi(u)}$ by the homomorphism $\varphi_u':= \varphi_u\circ \iota_u$,  where $\iota_u :B_u'\hookrightarrow B_u$ is the inclusion map.

Note that the underlying graph of $\B'$ is equal to the underlying graph of $\B$. Moreover, edge elements do not change, that is,    $o_{f}^{\varphi'}=o_{f}^{\varphi}$ for all $f\in EB'=EB$.   The morphism   $\varphi':\B'\rightarrow \A$     satisfies   conditions (2.i)-(2.iii)  of Definition~\ref{def:marked}  since 
$$(B_w', \varphi_w', \{\alpha_f'(B_f')\}_{f\in \st_1(w, \B')})=(B_w, \varphi_w, \{\alpha_f(B_f)\}_{f\in \st_1(w, \B)})$$
for all $w\neq u$ and  $(B_u, \varphi_u, \{\alpha_f(B_f)\}_{f\in \st_1(u, \B)})$ is replaced by 
$$(B_u', \varphi_u', \{\alpha_f'(B_f')\}_{f\in \st_1(u, \B')})=(B_u', \varphi_u'=\varphi_u\circ \iota_u, \{\alpha_f(B_f)\}_{f\in \st_1(u, \B)-\{g\}}).$$  
\begin{claim}{\label{claim:unfold}}
There is a morphism $\sigma:\B'\rightarrow \B$ such that (i) $\sigma_{\ast}:\pi_1(\B', u_0)\rightarrow \pi_1(\B, u_0)$ is an isomorphism   and (ii) $\varphi  \circ \sigma =\varphi'$.  
\end{claim}
\begin{proof}[Proof of Claim]
Indeed, consider the morphism $\sigma:\B'\rightarrow \B$ defined by 
$$\sigma=(\sigma, \{\sigma_u\ | \ u\in VB\}, \{\sigma_f\ | \ f\in EB\},  \{o_f^{\sigma} \ | \ f\in EB\}, \{t_f^{\sigma} \ | \ f\in EB\})$$
where $\sigma=\text{Id}_{B}:B\rightarrow B$, $\sigma_x=\text{Id}_{B_x }:B_x \rightarrow B_x$ for all $x\in VB\cup EB-\{u\}$ and $\sigma_u:B_u'\rightarrow B_u$ is the inclusion map  $B_u'\hookrightarrow B_u$, and where $o_f^{\sigma}=1$ for all $f\in EB$.
\end{proof}

It follows from   item (i)  that $\pi_1(\B', u_0)$ splits as a free product of cyclic groups and that  the tuple $\mathcal{T}':=\sigma_{\ast}^{-1}(\mathcal{T})$ generates $\pi_1(\B', u_0)$. Item (ii) further implies that $\varphi_{\ast}'(\mathcal{T}')$ is equal to $\varphi_{\ast}(\mathcal{T})$.   Thus $((\B', u_0), \varphi', \mathcal{T})$ is a marked morphism over $\A$. We will say that $((\B', u_0), \varphi', \mathcal{T})$  is obtained from $\M$ by \emph{unfolding along the edge $g$}.    We have the following result.
\begin{lemma}{\label{lemma:unfold}}
The marked morphism $((\B', u_0), \varphi', \mathcal{T})$  belongs to $\Omega_{\text{min}}(\mathcal{T}_0)$  and   $|\B'|_c= |\B|_c-1$. 
\end{lemma}

\begin{lemma}{\label{lemma:unfoldI}} If there is a marked morphism  $\M$ in   $\Omega_{\text{min}}(\mathcal{T}_0)$  such that an elementary fold  is applicable to  $\varphi$,  then there there  is a marked morphism  $((\B', u_0'), \varphi', \mathcal{T}')$ in  $\Omega_{\text{min}}(\mathcal{T}_0)$  such that (i) all edge groups of $\B'$ are trivial,  and (ii) an elementary fold is applicable to  $\varphi'$. 
\end{lemma}  
\begin{proof}
Our assumption  on $\varphi$  means that there are distinct edges $f$ and $g$ of $\B$  with same initial vertex and $\varphi(f)=\varphi(g)$ such that $o_f =o_g$ (and $t_f=t_g$ if $\omega(f)\neq \omega(g)$).  By  Lemma~\ref{lemma:possinlefolds}, we can assume  that the edge $g$ has trivial group in $\B$.  Assume further that the elementary fold is of type IIIA, i.e.  $\omega(f)=\omega(g)$. The case of an elementary fold of type IA is handled similarly.  

It follows from Lemma~\ref{lemma:schenitzer} that there is a vertex $u$ of $\B$   such that  $\st_1(u, \B)$ is non-empty  and that   the decorated group $(B_u, \varphi_u, \{\alpha_{h} (B_{h})\}_{h\in \st_1(u, \B)} )$   is isomorphic   to a strongly collapsible decorated group over $A_{\varphi(u)}$.   Thus, after applying auxiliary moves of type A2  to the edges in $\st_1(u, \B)$,   we can assume that
$$(B_u, \varphi_u, \{\alpha_{h} (B_h)\}_{h\in \st_1(u, \B)})$$  
is strongly collapsible.  But we apply this moves in such a way that, if an auxiliary move of type  A2 is applied to the edge $f$    with element $b$ (this   only occurs if   $f\in \st_1(u, \B)$),  then we also apply an auxiliary move of type  A2 to the edge  $g$ with element $b$. Doing in  this way, we still have    $o_f =o_g $ after this A2 move.

Finally,  let $((\B_1, u_0), \varphi_1 , \mathcal{T}_1)$  be the marked   morphism   obtained from  $\M$  by unfolding along an arbitrary edge of $\st_1(x, \B)$.  Thus, $((\B_1, u_0), \varphi_1 , \mathcal{T}_1)$ belongs to $\Omega_{\text{min}}(\mathcal{T}_0)$. Moreover,   an elementary fold is still applicable to   the edges $f$ and $g$. Since  the number of edge pairs with non-trivial group decreases by one, we see that   repeating this argument   $|\B|_c$-times   we obtain an element of $\Omega_{\text{min}}(\mathcal{T}_0)$ with the desired properties.  
\end{proof}

\begin{lemma}{\label{lemma:reduceedges}}
No element of $\Omega_{\text{min}}(\mathcal{T}_0)$ violates condition (F1) of Definition~\ref{def:folded}. 
\end{lemma}
\begin{proof}
Suppose that there is a marked morphism $\M$ in $\Omega_{\text{min}}(\mathcal{T}_0)$  such that $\varphi$ violates condition (F1) of the definition of folded morphisms. We will derive a contradiction by showing that this implies the existence  of a marked morphism   in $\Omega(\mathcal{T}_0)$ with $c$-complexity strictly smaller than $(n_1, n_2, n_3)$.   It follows from Lemma~\ref{lemma:possinlefolds} that there are  distinct edges $f$ and $g$ of $\B$   with $u:=\alpha(f)=\alpha(g)$ and $e:=\varphi(f)=\varphi(g)$ such that $B_g=1$ and that  $o_g^{\varphi}= \varphi_u(b_u) o_f^{\varphi} \alpha_e(c)$ for some $b_u\in B_u$ and $c\in A_e$. 

By  applying  auxiliary moves to  $\varphi$,   we obtain a marked morphism    $((\B, u_0), \varphi',   \mathcal{T}')$  in  $\Omega_{\text{min}}(\mathcal{T}_0)$  such that an elementary fold is applicable to the edges $f$ and $g$, that is, such that   $o_{g}^{\varphi'}=o_{f}^{\varphi'}$ (and $t_{g}^{\varphi'}=t_{f}^{\varphi'}$ if the vertices $\omega(f)$ and $\omega(g)$ are distinct).  Lemma~\ref{lemma:unfoldI} gives us   a  marked    morphism $(\varphi'':\B''\rightarrow \A, \mathcal{T}'')$ in $\Omega_{\text{min}}(\mathcal{T}_0)$  such that   $B_h''=1$ for all $h\in EB''$    and that an elementary fold is applicable to $\varphi''$. Thus there are distinct edges $f''$ and $g''$ of $\B''$ with $u'':=\alpha(f'')=\alpha(g'')$ and $e'':=\varphi(f'')=\varphi(g'')$  such that $o_{g''}^{\varphi''}=  o_{f''}^{\varphi''}$ (and $t_{g''}^{\varphi''}=t_{f''}^{\varphi''}$ if the vertices $\omega(f'')$ and $\omega(g'')$ are distinct)..

Let $\overline{\varphi}:\overline{\B}\rightarrow \A$ be the morphism   obtained from $\varphi''$ by first  folding the edges  $f''$ and $g''$ into a single edge $\overline{f}$,  and then applying   a vertex morphism to the vertex $\omega(\overline{f})$ of $\overline{B}=B''/[f''=g'']$.     Thus $\overline{\varphi}:\overline{\B}\rightarrow \A$ is a vertex injective morphism  with $\overline{B}_f=1$ for  all edges of $\overline{\B}$. Therefore      $\overline{\varphi}$ trivially  satisfies conditions (2.i)-(2.iii) of Definition~\ref{def:marked}   and the fundamental group of $\overline{\B}$ splits as a free product of cyclic groups.  

By Lemmas~\ref{lemma:fold} and \ref{lemma:vertexmorphism}, there is    a  $\pi_1$-surjective  morphism $\sigma:\B''\rightarrow \overline{\B}$ such that $\overline{\varphi}\circ \sigma=\varphi''$. Denote by  $ \overline{\mathcal{T}}$   the tuple $\sigma_{\ast}(\mathcal{T}'')$. Thus $\overline{\mathcal{T}}$ generates the fundamental group of $\overline{\B}$ and $\overline{\varphi}_{\ast}(\overline{\mathcal{T}})=\overline{\varphi}_{\ast}(\sigma_{\ast}(\mathcal{T}''))=  \varphi_{\ast}''(\mathcal{T}'')$ which is Nielsen equivalent to $\varphi_{\ast}(\mathcal{T})$. To see that the $c$-complexity decreases note that 
$$ \text{rk}(\B)-\text{rk}(\overline{\B})= \text{rk}(B_{\omega(f)})+\text{rk}(B_{\omega(g)})- \text{rk}(\overline{B}_{\omega(\overline{f})})\geqslant 0$$
and that 
$$ \text{tn}(\overline{\B})-\text{tn}( \B)= \text{tn}(\overline{B}_{\omega(\overline{f})}) -( \text{tn}(B_{\omega(f)})+ \text{tn}(B_{\omega(g)}))  \geqslant 0.$$
Since $|E\overline{B}|=|EB''|-2=|EB|-2$ it  follows that $c(\overline{\varphi})<c(\varphi)$.
\end{proof}

\begin{corollary}{\label{cor:reduceedges}}
Let $\M$ be a marked  morphism  in $\Omega_{\text{min}}(\mathcal{T}_0)$. Then for any vertex $u$ of $\B$ the  corresponding   decorated group $(B_u, \varphi_u, \{\alpha_f(B_f)\}_{f\in \st_1(u, \B)})$ does not fold peripheral subgroups.  
\end{corollary}
\begin{proof}
Suppose that for some vertex $u$ of   $\B$ the decorated group  $(B_u, \varphi_u, \{\alpha_f(B_f)\}_{f\in \st_1(u, \B)})$ folds peripheral subgroups.  By definition,  the following hold:  (1) $B_u=C_u \ast_{f\in \st_1(u, \B)}     \alpha_f(B_f)$ for some subgroup $C_u\leq B_u$ with $C_u\cap \ker(\varphi_u)=1$,  and  (2) there are  distinct edges $f$ and $g $ in $\st_1(u, \B)$  with  $e:= \varphi(f)=\varphi(g)$ such that  
$$o_{g}^{\varphi}=\varphi_u(b) o_{f}^{\varphi} \alpha_e(c)$$ 
for some $b\in C_u\ast_{h\in \st_1(u, \B)-\{g\}}   \alpha_h(B_h)$ and  $c\in A_e$.

Let $((\B', u_0), \varphi', \mathcal{T})$   be the marked morphism    obtained from $\M$   by unfolding along the edge $g$.   Thus  the group of  $g$ in $\B'$ is trivial,   the group of   the vertex $u$ in $\B'$ is equal to 
$$B_u'=C_u \ast_{h\in \st_1(u, \B)-\{g\} }    \alpha_h(B_h),$$   
and so the edge  element $o_g^{\varphi'}$ lies in the double coset $\varphi_u'(B_u') o_f^{\varphi'} \alpha_e(A_e)$. This implies that $((\B', u_0), \varphi', \mathcal{T})$ is an element of $\Omega_{\text{min}}(\mathcal{T}_0)$ that violates condition (F1) of Definition~\ref{def:folded}, contradicting     Lemma~\ref{lemma:reduceedges}.
\end{proof}


\subsection{Searching for non-vertex injective marked  morphisms in $\Omega_{\text{min}}(\mathcal{T}_0)$} Let $\M$  be a marked morphism in $\Omega_{\text{min}}(\mathcal{T}_0)$.  Assume that   $\varphi$ is vertex injective. From Lemma~\ref{lemma:reduceedges} we know that  $\varphi$ does not violate    condition (F1).   Since $\varphi$ is not folded,   Lemma~\ref{lemma:possinlefolds} implies  that there is an edge $f$ of $\B$   such that  
$$\text{(i)} \ B_f=1, \ \ \text{ and } \ \ \text{(ii)}  \ \alpha_{e}^{-1}(o_f^{-1}  \varphi_u(B_u)  o_f) \neq 1,$$ 
where $u:=\alpha(f)$ and  $e:=\varphi(f)$.   Denote the vertex $\omega(f)$ by $u'$, the vertex $\varphi(u)$ by $v$ and the vertex $\varphi(u')$ by $v'$. We will construct a new marked morphism in $\Omega_{\text{min}}(\mathcal{T}_0)$  that has fewer edges with trivial group than $\M$.

Recall that the edge group   $A_{e}$ is infinite cyclic  with generator $a_{e}$, that corresponds to a simple closed curve on $\mathcal{O}$.  Thus  $\alpha_{e}^{-1}( o_f^{-1}  \varphi_u(B_u)  o_f)=\langle a_{e}^{k_f'}\rangle$   for some integer $k_f'>0$. As the vertex homomorphism $\varphi_u:B_u\rightarrow A_v$ is injective, there is a unique element $b_f'\in B_u$ such that $\varphi_u(b_f')= o_f \alpha_{e}(a_{e}^{k_f'}) o_f^{-1}$. Then we have 
$$ \langle \varphi_u(b_f')\rangle  = o_f   \langle \alpha_{e}^{k_f'}\rangle  o_f^{-1}= \varphi_u(B_u)\cap  o_f   \alpha_{e}(A_{e})  o_f^{-1}.$$

Let $\B'$  be the graph of  groups obtained from $\B$ by replacing the edge group $B_f=1$ by the group $B_{f}':=\langle b_f'\rangle \leq B_{u}$,  and by replacing the vertex  group $B_{u'} $ by the group $B_{u'}':= B_{u'}\ast \langle b_f'\rangle$.  The boundary monomorphisms $\alpha_{f}':B_f'\rightarrow B_u'=B_u$ and $\omega_f':B_f'\rightarrow B_{u'}'$  are defined as   the inclusion maps.

Let   $\varphi':\mathbb{B}'\rightarrow \A$  be  the morphism  obtained from $\varphi$ by replacing  the edge homomorphism $\varphi_f$ by the homomorphism $\varphi_f':B_f' \rightarrow A_{e}$ defined by   $\varphi_f'(b_f')= a_{e}^{k_f'}$,   and by replacing  the vertex homomorphism $\varphi_{u'}$  by the   homomorphism  $\varphi_{u'}':B_{u'}'  \rightarrow A_{\varphi(u')}$ induced by $\varphi_{u'}$ and  the map   $b_f'\mapsto  t_f^{-1} \omega_{e}(a_{e}^{k_f'}) t_f$. 

Note that the underlying graph of $\B'$ is equal to the underlying graph $B$ of $\B$. Moreover, $o_h^{\varphi'}=o_h$ for all  $h\in EB'=EB$.  As in Claim~\ref{claim:unfold}, we construct a morphism $\sigma:\B\rightarrow \B'$ such that (i) $\sigma_{\ast}$ is an isomorphism, and (ii) $\varphi=\varphi'\circ \sigma$.    Consequently, $\pi_1(\B', u_0)$ splits as a free product of cyclic groups,   the tuple    $\mathcal{T}':=\sigma(\mathcal{T})$ generates $\pi_1(\B', u_0)$, and  clearly   $\varphi_{\ast}'(\mathcal{T}')$ is Nielsen equivalent to $\mathcal{T}_0$.   Fig.~\ref{fig:foldalong} shows the effect of the fold on the vertices $u$ and $u'$. The remaining vertex and edge groups/homomorphisms do not change.   \begin{figure}[h!] 
\begin{center}
\includegraphics[scale=0.9]{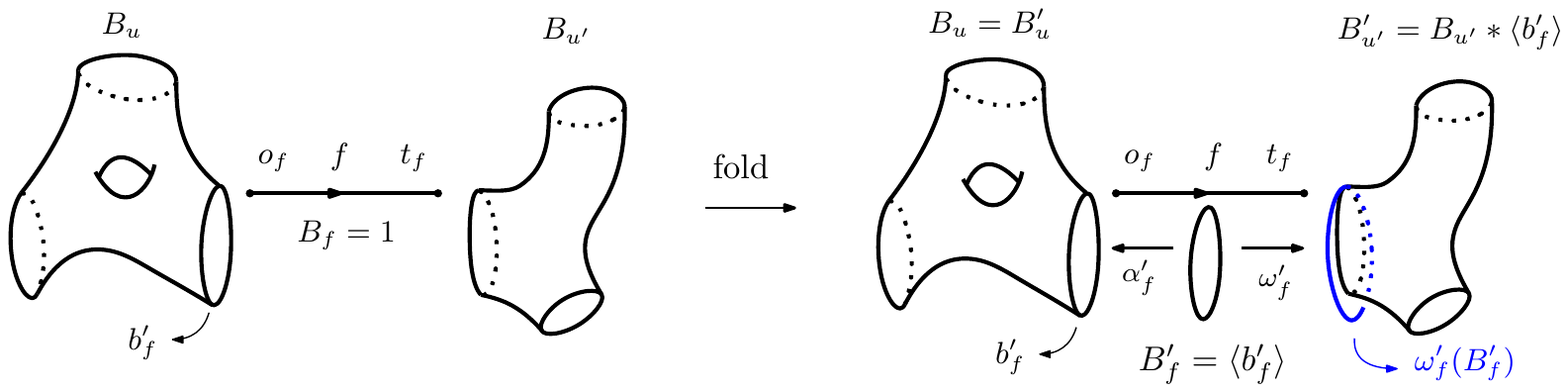}
\end{center}
\caption{A fold along the edge $f$. }\label{fig:foldalong}
\end{figure}

We want to show that   $((\B', u_0), \varphi', \mathcal{T}')$ is marked morphism. To this end we   need  to check that $\varphi'$  satisfies conditions (2.i)-(2.iii) of Definition~\ref{def:marked}.  This is done in the next claim.
\begin{claim}{\label{claim:coverlikefold}}
The morphism  $\varphi':\B'\rightarrow \A$  satisfies   conditions (2.i)-(2.iii) of Definition~\ref{def:marked}. 
\end{claim}
\begin{proof}[Proof of Claim]
Note that   we only  need to look at the vertices $u$ and $u'$   as the remaining vertex and edge groups/homomorphisms   do not change in the fold.

We start at the vertex $u$. The vertex homomorphism $\varphi_u':B_u'\rightarrow A_{v}$ is injective since  $\varphi$ is vertex injective and $\varphi_u'$ is equal to $\varphi_u$.  Thus we need to show that  (2.ii) holds.   By  definition  there is an  orbifold covering  $\eta_u':\mathcal{O}_u'\rightarrow \mathcal{O}_v$  such that $(B_u, \varphi_u, \{\alpha_h(B_h)\}_{h\in \st_1(u, \B)})$ is isomorphic to a sub-decoration of the decorated group induced by $\eta_u'$. More precisely, let $\delta_{u,1}', \ldots,\delta_{u,n_u}'$  be the compact boundary components of $\mathcal{O}_u'$ and let $C_{u,j }'$ be  the maximal peripheral subgroup $\pi_1^o(\mathcal{O}_u')$ corresponding to  $\delta_{u,j}'$.  Then there is a subset $\{k_{u,1}, \ldots, k_{u,m_u}\}$ of $\{1,\ldots, n_u\}$ such that
$$ (B_u, \varphi_u, \{\alpha_h(B_h)\}_{h\in \st_1(u, \B)}) \cong (\pi_1^o(\mathcal{O}_u'), (\eta_u')_{\ast}, \{C_{u,k_{u,j}}'\}_{1\leqslant j\leqslant m_u}).$$
To the element $b_f'$ there corresponds a unique compact boundary component, $\delta_{u,k_{u,0}}'$,  of $\mathcal{O}_u'$. We need to show that $\delta_{u,k_{u,0}}'\neq \delta_{u,k_{u,1}}', \ldots, \delta_{u,k_{u,m_u}}'$. In fact, suppose that  $\delta_{u, k_0}'=\delta_{u, k_{u,j}}'$ for some $1\leqslant j\leqslant m_u$. Let   $g$ be the edge in  $\st_1(u, \B)$ that corresponds to $\delta_{u, k_{u,j}}'$. Then the edge element  $o_{f}$ belongs to the $(\varphi_{u'}(B_{u'}), \alpha_{e}(A_e))$-double coset  of $o_g$   and hence   the morphism  $\varphi$  violates condition (F1) of Definition~\ref{def:folded}, a contradiction.    Therefore the decorated group $(B_u', \varphi_u', \{\alpha_h'(B_h')\}_{h\in \st_1(u, \B')})$ is isomorphic to the decorated group 
$$ (\pi_1^o(\mathcal{O}_u'), (\eta_u')_{\ast}, \{C_{u,k_{u,0}}', C_{u,k_{u,1}}', \ldots, C_{u,k_{u,m_u}}'\})$$
which is a sub-decoration of $(\pi_1^o(\mathcal{O}'), (\eta_u')_{\ast}, \{C_{u,j}'\}_{1 \leqslant j\leqslant n_u})$.

Now we look at the vertex $u'$. We first observe that  $(B_{u'}, \varphi_{u'}, \{\alpha_{h}(B_h)\}_{h\in \st_1(u', \B)})$ is isomorphic to a strongly decorated group as otherwise the morphism $\varphi$ would violate condition (F1) of Definition~\ref{def:folded}.  Consequently the decorated group   $(B_{u'}', \varphi_{u'}', \{\alpha_{h}'(B_h')\}_{h\in \st_1(u', \B')})$  is also  isomorphic to a strongly decorated group. This implies that $\varphi'$ satisfies condition (2.iii) if the vertex homomorphism  $\varphi_{u'}'$ is not injective.  

Suppose  $\varphi_{u'}'$ is injective. We need to show that (2.ii) holds at $u$. But this follows from the fact that   $(B_u, \varphi_u, \{\alpha_f(B_f)\}_{f\in \st_1(u, \B)})$ is strongly  collapsible and from  Lemma~\ref{lemma:inj}.
\end{proof}

Therefore $((\B', u_0), \varphi', \mathcal{T}')$ is a marked morphism  that clearly   belongs to $\Omega(\mathcal{T}_0)$. Since $\pi_1(\B, u_0)$ and $\pi_1(\B', u_0)$ are isomorphic and have the same underlying graphs, we further see that    
$$c((\B', u_0), \varphi', \mathcal{T}') =c\M. $$
Notice that   the number of edge pairs  that have non-trivial group in   $\B'$ is equal to $|\B|_c+1$.   We will say that the marked morphism  $((\B', u_0), \varphi', \mathcal{T}')$ is obtained from $\M$ by \emph{folding along the edge $f$}.

\begin{lemma}{\label{lemma:nonvertexinj}}
There exists a marked morphism  $((\overline{\B}, \overline{u}_0), \overline{\varphi}, \overline{\mathcal{T}})$  in $\Omega_{\text{min}}(\mathcal{T}_0)$ such that $\overline{\varphi}$ is not vertex injective. 
\end{lemma}
\begin{proof}
Let $\M$ be an arbitrary element of $\Omega_{\text{min}}(\mathcal{T}_0)$. Set $((\B_0, w_{0}), \varphi_0,  \mathcal{T}_0):=\M$. Suppose that for $i\geqslant 0$, the marked morphism $((\B_i, w_{i}), \varphi_i, \mathcal{T}_i)\in \Omega_{\text{min}}(\mathcal{T}_0)$ has been defined so that $|\B_i|_c=|\B|_c+i$.  If $\varphi_i$ is not vertex injective, then   
$$ ((\overline{\B}, \overline{u}_0), \overline{\varphi}, \overline{\mathcal{T}}) :=((\B_i, w_{i}), \varphi_i,  \mathcal{T}_i)$$
is the the element of $\Omega_{\text{min}}(\mathcal{T}_0)$ we are searching for. Suppose that  $\varphi_i$ is vertex injective. We know that $\varphi_i$ is not folded and Lemma~\ref{lemma:reduceedges} further says that $\varphi_i$ does not violate condition (F1) of Definition~\ref{def:folded}.    Thus  there is an edge $f_i$ in $\B_i$  such that  we can fold $\varphi_i$ along   $f_i$. We define $((\B_{i+1}, w_{i+1}), \varphi_{i+1},   \mathcal{T}_{i+1})$ as the marked morphism obtained from  $(( \B_i, w_{i}), \varphi_i,  \mathcal{T}_i)$ by folding  along $f_i$.   Thus  $((\B_{i+1}, u_{0,i+1}), \varphi_{i+1},   \mathcal{T}_{i+1})$ belongs to $\Omega_{\text{min}}(\mathcal{T}_0)$.   Since, for each $i$,  we  have $|\B_i|_c\leqslant \frac{1}{2}|EB|$,  it follows that this process terminates in at most $\frac{1}{2}|EB|$ steps.  
\end{proof}

\begin{corollary}{\label{cor:decreased}}
Let $\M$  be an element of $\Omega_{\text{min}}(\mathcal{T}_0)$    and $u$ a vertex of $\B$. If the decorated group $(B_u, \varphi_u, \{\alpha_f(B_f)\}_{f\in \st_1(u, \B)})$ has an obvious relation, then there is a marked morphism  $((\overline{\B}, \overline{u}_0), \overline{\varphi}, \overline{ \mathcal{T}})$ in $  \Omega_{\text{min}}(\mathcal{T}_0)$  such that (i) $\overline{\varphi}$ is not vertex  injective, and (ii)  $d((\overline{\B}, \overline{u}_0), \overline{\varphi}, \overline{ \mathcal{T}})<d\M$. 
\end{corollary}
\begin{proof}
By definition, the decorated group $(B_u, \varphi_u, \{\alpha_f(B_f)\}_{f\in \st_1(u, \B)})$ has an obvious relation if the   following two conditions  hold:
(1) $B_u=C_u\ast_{g\in \st_1(u, \B)}  \alpha_g(B_g)$  for some subgroup $C_u\leq B_u$ such that $C_u\cap \ker(\varphi_u)=1$,  and  (2) there is some edge  $f$ in $ \st_1(u, \B)$ with $e:=\varphi(f)$  such that 
$$ 0<|A_e : \alpha_e^{-1}(o_f^{-1}\varphi_x(B_x')o_f)| <| A_e :\varphi_f(B_f)|.$$
where $B_u':= C_u \ast_{g\in \st_1(u, \B)-\{f\}}   \alpha_g(B_g)  \leq B_u$.

Let  $((\B', u_0), \varphi', \mathcal{T}')$ be the marked   morphism   obtained  from $\M$  by unfolding along the edge $f$.  Thus, in $\B'$,  the edge $f$ has trivial group and the vertex $u$ has group $$B_u':= C_u \ast_{g\in \st_1(u, \B)-\{f\}}   \alpha_g(B_g).$$ 
The vertex homomorphism $\varphi_u':B_u'\rightarrow A_v$ is equal to $\varphi_u\circ \iota_u$, where $\iota_u$ is the inclusion map $B_u' \hookrightarrow B_u$.

If the vertex homomorphism $\varphi_u'$  is not injective, then we put  $((\overline{\B}, \overline{u}_0), \overline{\varphi}, \overline{ \mathcal{T}}):=((\B', u_0), \varphi',  \mathcal{T}').$    Lemma~\ref{lemma:unfold} implies that   $((\overline{\B}, \overline{u}_0), \overline{\varphi}, \overline{ \mathcal{T}}) $ belongs to $\Omega_{\text{min}}(\mathcal{T}_0)$.  The $d$-complexity of   $((\overline{\B}, \overline{u}_0), \overline{\varphi}, \overline{ \mathcal{T}})$ is strictly smaller than $d\M$ as   the number of edges with nontrivial group decreases by one.

If $\varphi_u'$ is injective,  then we define  $((\overline{\B}, \overline{u}_0), \overline{\varphi}, \overline{ \mathcal{T}})$  as the marked   morphism   obtained from $(\varphi':\B'\rightarrow \A, \mathcal{T}')$  by folding along the edge $f$ (item (2) implies that this is possible). From the description of the fold we see that   the vertex homomorphism $\overline{\varphi}_{\omega(f)}$  is not injective,  and hence    the morphism  $\overline{\varphi}$ is not vertex injective.     To see that the $d$-complexity decreases,  note that   $|\overline{\B}|_c=|\B'|_c+1=|\B|_c$ and that 
 $$ c_E(\varphi)-c_E(\overline{\varphi}) = | A_e :\varphi_f(B_f)| - |A_e:\overline{\varphi}_f(\overline{B}_f)| = | A_e :\varphi_f(B_f)| -|A_e : \alpha_e^{-1}(o_f^{-1}\varphi_u(B_u')o_f)| <0.$$
\end{proof}


\subsection{Non-vertex injective  elements of  $\Omega_{\text{min}}(\mathcal{T}_0)$ with minimal   $d$-complexity}  In this subsection we  will show  that an element of $\Omega_{\text{min}}(\mathcal{T}_0)$ with minimal $d$-complexity yields a special almost orbifold over $\mathcal{O}$ and a generating tuple  that is mapped onto a tuple   Nielsen equivalent to $\mathcal{T}_0$.  Lemma~\ref{lemma:nonvertexinj} implies that the set 
$$\Omega_{\text{min}}^{*}(\mathcal{T}_0):=  \{((\B', u_0'), \varphi', \mathcal{T}')  \in \Omega_{\text{min}}(\mathcal{T}_0) \ | \  \varphi' \text{  is not vertex injective}\}$$
is non-empty.   Choose an element $\M$   in $\Omega_{\text{min}}^{\ast}(\mathcal{T}_0)$   that has  minimal $d$-complexity, i.e.  
$$d\M \leqslant d((\B', u_0'), \varphi', \mathcal{T}') \ \ \text{ for all } \ \ ((\B', u_0'), \varphi', \mathcal{T}')\in  \Omega_{\text{min}}^{*}(\mathcal{T}_0).$$  
As $\varphi:\B\rightarrow \A$ is not vertex injective, there must be  at least one vertex $u$ such that the vertex homomorphism $\varphi_{u}:B_{u}\rightarrow A_{v}$ is not injective,  where $v:=\varphi(u)$.   The next lemma  combined with    condition (2.iii) of Definition~\ref{def:marked},    imply that   $\varphi$ does not satisfy   condition (F0) of Definition~\ref{def:folded} only at  the vertex $u$.  In fact,  if $\varphi_{u'}$ is not injective,  then $\st_1(u' , \B)$ is non-empty and the decorated group $(B_{u'}, \varphi_{u'}, \{\alpha_f(B_f)\}_{f\in \st_1(u', \B)})$ is isomorphic to a strongly collapsible decorated group over $A_{\varphi(u')}$. 
\begin{lemma}{\label{lemma:noncoll}}
If $w$ is a vertex of $\B$ distinct from $u$ with $\st_1(w, \B)\neq \emptyset$,   then  $(B_w, \varphi_w, \{\alpha_g(B_g)\}_{g\in \st_1(w, \B)})$  is not isomorphic to a strongly collapsible decorated group over $A_{\varphi(w)}$.  
\end{lemma}
\begin{proof}
We will show that if the result does not hold, then there is a marked morphism in  $\Omega_{\text{min}}^{*}(\mathcal{T}_0)$ with $d$-complexity smaller than $d\M$, which is a contradiction.

Suppose that the lemma does not hold. Thus   there is some vertex $w\neq u$ such that $\st_1(w, \B)\neq \emptyset$  and  the decorated group $(B_w, \varphi_w, \{\alpha_g(B_g)\}_{g\in \st_1(w, \B)})$ is isomorphic to a strongly  collapsible decorated group over $A_{\varphi(w)}$.  After applying auxiliary moves of type A2 to the edges in $\st_1(w, \B)$, we  can assume that  $(B_w, \varphi_w, \{\alpha_g(B_g)\}_{g\in \st_1(w, \B)})$ is strongly collapsible, and therefore we can unfold along any edge of $\st_1(w, \B)$. 

Let $f$ be an arbitrary edge in $\st_1(w, \B)$. The marked morphism $((\B', u_0), \varphi', \mathcal{T}')$  obtained from $\M$ by  unfolding  along  $f$  belongs to $ \Omega_{\text{min}}(\mathcal{T}_0)$.  As the vertex homomorphism $\varphi_{u}:B_{u}\rightarrow A_v$ does not change in the unfolding move, we see that  $\varphi':\B'\rightarrow \A$ is not vertex injective, and so     $((\B', u_0), \varphi', \mathcal{T}')$ belongs to $\Omega_{\text{min}}^{*}(\mathcal{T}_0)$. To finish the proof observe that  $d$-complexity of  $((\B', u_0), \varphi', \mathcal{T}')$  is strictly smaller than $d\M$ as $|\B'|_c=|\B|_c-1$.    
\end{proof}

The previous lemma  combined with  condition (2.ii),  imply that,  for all vertices  $w \in VB$ with $w\neq u$  and  $\st_1(w, \B)\neq \emptyset$,   the decorated group $(B_{w}, \varphi_{w}, \{\alpha_{f}(B_{f})\}_{f\in \st_1(w, \B)})$   is isomorphic  to the decorated group   induced by an orbifold covering $\eta_w':\mathcal{O}_w'\rightarrow \mathcal{O}_{\varphi(w)}$  of finite degree.   In particular, the morphism  $\varphi:\B\rightarrow \A$ is locally surjective at all such vertices of $\B$.

\smallskip

Now we turn our attention to the vertex $u$, where the morphism  $\varphi$ fails to be vertex injective. We will show that the minimal $d$-complexity of $\M$ forces the  decorated group $(B_u, \varphi_u,\{\alpha_f(B_f)\}_{f\in \st_1(u, \B)})$ at $u$ to be isomorphic to one  that is induced by a special almost orbifold covering of the vertex orbifold $\mathcal{O}_v$.   
\begin{lemma}{\label{lemma:almostcovering}}
There exists  a special   almost orbifold covering $\eta_{u}':\mathcal{O}_{u}'\rightarrow\mathcal{O}_{v}$ with $\partial \mathcal{O}_{u}'=\delta_{u,1}'\cup \ldots \cup \delta_{u, n_{u}}'\cup \delta_{u}'$ ($\delta_{u}'$ is the exceptional boundary component of $\mathcal{O}_{u}'$) such that 
$$ (B_{u}, \varphi_{u}, \{\alpha_{f}(B_{f})\}_{f\in \st_1(u, \B)}) \cong (\pi_1^o(\mathcal{O}_{u}'), (\eta_{u}')_{\ast}, \{C_{u,j}'\}_{1\leqslant j\leqslant n_{u}}).$$
In particular, $\varphi$ is locally surjective at $u$. 
\end{lemma}
\begin{proof}
As  $\varphi_u:B_u\rightarrow A_v$  is not injective, it follows from    Definition~\ref{def:marked}(2.iii) that $\st_1(u, \B)$ is non-empty and that  the decorated group  $(B_u, \varphi_u, \{\alpha_{f}(B_{f})\}_{f\in \st_1(u, \B)})$ 
is  isomorphic   to a  strongly collapsible decorated group over $A_v$. Thus,  after   applying  auxiliary moves of type A2 to the edges in $\st_1(u, \B)$,  we can assume that   $(B_u, \varphi_u, \{\alpha_{f}(B_{f})\}_{f\in \st_1(u, \B)})$  is strongly collapsible.  

Suppose that the vertex orbifold is orientable, that is,  $\mathcal{O}_v$ is a small orbifold not isomorphic to a Moebius band (the case in which   $\mathcal{O}_v$ is isomorphic to a Moebius band is handled   as in case ($\beta$) below).  By Proposition~\ref{proposition:1}, one of the following holds:
\begin{enumerate}
\item[($\alpha$)] $(B_u, \varphi_u, \{\alpha_{f}(B_{f})\}_{f\in \st_1(u, \B)})$ projects onto a  strongly collapsible decorated group $(H, \lambda, \{H_k\}_{k\in K})$ over $A_v$ such that one of the following three conditions occur: 
\begin{enumerate}
\item[($\alpha.1$)]$(H, \lambda, \{H_k\}_{k\in K})$ folds peripheral  subgroups.

\item[($\alpha.2$)] $(H, \lambda, \{H_k\}_{k\in K})$  has an obvious relation.

\item[($\alpha.3$)] $\rk(H)<\rk(B_u)$ or $\rk(H)=\rk(B_u)$ and $\text{tn}(H)>\text{tn}(B_u)$.
\end{enumerate}

\item[($\beta$)] There is a special almost orbifold covering $\eta_u':\mathcal{O}_u'\rightarrow \mathcal{O}_v$ with  $\partial \mathcal{O}_{u}'=\delta_{u,1}'\cup \ldots \cup \delta_{u, n_{u}}'\cup \delta_{u}'$ ($\delta_{u}'$ is the exceptional boundary component of $\mathcal{O}_{u}'$)  such that  $(B_u, \varphi_u, \{\alpha_{f}(B_{f})\}_{f\in \st_1(u, \B)})$ is isomorphic to a decorated group  obtained from the decorated group induced by $\eta_u'$  by adjoining a finite subgroup.  
\end{enumerate} 
 
\begin{claim} 
($\alpha$) cannot occur.
\end{claim}
\begin{proof}[Proof of Claim]
We will argue that if ($\alpha$) occurs then  there is a marked  morphism  $((\overline{\B}, \overline{u}_0 ), \overline{\varphi}, \overline{\mathcal{T}})$ in $\Omega(\mathcal{T}_0)$ such that either (1) $c((\overline{\B}, \overline{u}_0 ), \overline{\varphi}, \overline{\mathcal{T}})<(n_1, n_2, n_3)=c\M$,  or  (2) the marked morphism  $((\overline{\B}, \overline{u}_0 ), \overline{\varphi}, \overline{\mathcal{T}})$ belongs to $\Omega_{\text{min}}^{\ast}(\mathcal{T}_0)$  and  has  $d$-complexity strictly smaller than $d\M$.  
This contradicts the minimality of the $c$-complexity  and  of the $d$-complexity of  $\M$.  

Thus  suppose $(\alpha$) occurs.  Let $((\B', u_0'), \varphi', \mathcal{T}')$ be the marked morphism obtained from $\M$ by replacing the (strongly collapsible) decorated group   $(B_u, \varphi_u, \{\alpha_{f}(B_{f})\}_{f\in \st_1(u, \B)})$ by the (strongly collapsible) decorated group $(H, \lambda, \{H_k\}_{k\in K})$. We know that $((\B', u_0'), \varphi', \mathcal{T}')$ lies in $\Omega(\mathcal{T}_0)$. Moreover,  we have 
$$c((\B', u_0'), \varphi', \mathcal{T}')\leqslant c\M =(n_1, n_2, n_3) \ \ \text{ and } \ \ d ((\B', u_0'), \varphi', \mathcal{T}')=d\M.$$  
If the $c$-complexity of $((\B', u_0'), \varphi', \mathcal{T}') $ is strictly smaller than $(n_1, n_2, n_3)$,   then  $((\overline{\B}, \overline{u}_0 ), \overline{\varphi}, \overline{\mathcal{T}}):= ((\B', u_0'), \varphi', \mathcal{T}')$ is  a   marked morphism that belongs to $\Omega(\mathcal{T}_0)$   as in (1).  Assume that  
$$ c((\B', u_0'), \varphi', \mathcal{T}')= c\M=(n_1, n_2, n_3)$$
so that $((\B', u_0'), \varphi', \mathcal{T}')$ belongs to $\Omega_{\text{min}}(\mathcal{T}_0)$. This assumption on $((\B', u_0'), \varphi', \mathcal{T}')$ combined with  Lemma~\ref{lemma:pvm},  imply that    case  ($\alpha.3$) can not occur.

If  ($\alpha.1$) occurs,  then Corollary~\ref{cor:decreased} imlpies that there is a marked morphism $((\overline{\B}, \overline{u}_0 ), \overline{\varphi}, \overline{\mathcal{T}})$ in $\Omega(\mathcal{T}_0)$ with 
$$c ((\overline{\B}, \overline{u}_0 ), \overline{\varphi}, \overline{\mathcal{T}}) <   (n_1, n_2, n_3), $$
and therefore (1) holds.   If ($\alpha.2$) occurs, then   Corollary~\ref{cor:decreased} implies that there is a  marked morphism  $ ((\overline{\B}, \overline{u}_0 ), \overline{\varphi}, \overline{\mathcal{T}})$ in $\Omega_{\text{min}}^*(\mathcal{T}_0)$ with 
 $d((\overline{\B}, \overline{u}_0 ), \overline{\varphi}, \overline{\mathcal{T}})<d ((\B', u_0'), \varphi', \mathcal{T}') =d\M$, i.e.  $((\overline{\B}, \overline{u}_0 ), \overline{\varphi}, \overline{\mathcal{T}})$ is a marked morphism in $\Omega_{\text{min}}^*(\mathcal{T}_0)$  such that  (2) holds. 
\end{proof}

The previous claim implies that  ($\beta$) must occur. Thus,   with the notation from Example~\ref{ex:03},  we have 
$$(B_u, \varphi_u, \{\alpha_{f}(B_{f})\}_{f\in \st_1(u, \B)}) \cong (\pi_1^o(\mathcal{O}')\ast g\langle s_x^d\rangle g^{-1}, \lambda,  \{C_{u,j'}\}_{1\leqslant j\leqslant n_u}), $$ 
where  $s_x\in \pi_1^o(\mathcal{O}_v)$ corresponds to the exceptional point $x$ of $\mathcal{O}_v$, $g\in \pi_1^o(\mathcal{O}_v)$ is such that  the element of $\pi_1^o(\mathcal{O}_u')$ represented by $\delta_u'$  is mapped by $(\eta_u')_{\ast}$ onto $gs_x^kg^{-1}$,   and $d\geqslant1$ is such that $d=|\langle s_x\rangle :\langle s_x^d\rangle |$.   The assertion of the lemma  immediately follows from: 
\begin{claim}
$d=|\langle s_x\rangle|$. Consequently, 
 $$(\pi_1^o(\mathcal{O}_u')\ast g\langle s_x^d\rangle g^{-1}, \lambda,  \{C_{u,j}'\}_{1\leqslant j\leqslant n_u})= (\pi_1^o(\mathcal{O}'), (\eta_u')_{\ast},  \{C_{u,j}'\}_{1\leqslant j\leqslant n_u}).$$ 
\end{claim}
\begin{proof}[Proof of Claim]
The proof  goes as in the previous claim, i.e. we show that   if  $1\leqslant d<|\langle s_x\rangle|$, then  there is a marked  morphism  $((\overline{\B}, \overline{u}_0 ), \overline{\varphi}, \overline{\mathcal{T}})$ in $\Omega(\mathcal{T}_0)$ such that either (1) $c((\overline{\B}, \overline{u}_0 ), \overline{\varphi}, \overline{\mathcal{T}})<(n_1, n_2, n_3)=c\M$,  or  (2)  $((\overline{\B}, \overline{u}_0 ), \overline{\varphi}, \overline{\mathcal{T}})$ belongs to $\Omega_{\text{min}}^{\ast}(\mathcal{T}_0)$  and  its   $d$-complexity strictly smaller than $d\M$.

Assume that  $1\leqslant d<|\langle s_x\rangle|$.   This implies in particular that $x$ is a cone point of $\mathcal{O}_v$.    Without loss of generality assume that the decorated group  $(B_u, \varphi_u, \{\alpha_{f}(B_{f})\}_{f\in \st_1(u, \B)})$ is not only isomorphic to but   equal to $(\pi_1^o(\mathcal{O}_u')\ast g\langle s_x^d\rangle g^{-1}, \lambda,  \{C_{u,j}'\}_{1\leqslant j\leqslant n_u}).$

Let $\B'\subseteq \B$ be the  sub-graph of groups    obtained from $\B$ by replacing the vertex group   $B_u$   by the group $ B_u' := \pi_1^o(\mathcal{O}_u')\leq B_u$.   
Then $B_u=B_u'\ast g\langle s_x^d\rangle g^{-1}$,  and  the group    $\pi_1(\B, u_0)$ splits as $  \pi_1(\B', u_0) \ast \gamma g\langle s_x^d\rangle g^{-1}\gamma^{-1},$ 
for some $\B$-path $\gamma$ contained in $\B'$  with initial vertex $\alpha(\gamma)=u_0$ and terminal vertex $\omega(\gamma)=u$.  It follows from  Grushko's Theorem   that the generating  tuple  $\mathcal{T}$ of $\pi_1(\B , u_0)$ is Nielsen equivalent to a tuple of type  $$\mathcal{T}' \oplus ([\gamma\cdot  gs_x^{rd}g^{-1}  \cdot    \gamma^{-1}]),$$  
where $\mathcal{T}'$ is a generating tuple of $\pi_1(\B', u_0)\leq \pi_1(\B, u_0)$ and the integer $r$ is coprime with   $|\langle s_x^d \rangle|=|\langle s_x\rangle|/d$. Write $rd=lk+p$ for some $0\leqslant p<k$.  We have 
\begin{eqnarray*}
\varphi_{\ast} (\mathcal{T} ) & \sim_{NE}  &  \varphi_{\ast} (\mathcal{T}')\oplus \varphi_{\ast} ([\gamma\cdot gs_x^{ r d} g^{-1} \cdot \gamma^{-1}])\nonumber \\
                                 & = & \varphi_{\ast}(\mathcal{T}') \oplus ([\varphi(\gamma)\cdot g   s_x^{r d} g^{-1}\cdot \varphi(\gamma)^{-1}])\nonumber \\
                                 & \sim_{NE} & \varphi_{\ast}(\mathcal{T}')\oplus ([\varphi(\gamma)\cdot  g  s_x^{-lk} g^{-1} \cdot  \varphi(\gamma)^{-1}] [\varphi(\gamma)\cdot g  s_x^{rd} g^{-1} \cdot \varphi(\gamma)^{-1}])\nonumber \\
                                 & = & \varphi_{\ast}(\mathcal{T}') \oplus([\varphi(\gamma)\cdot g s_x^{p}  g^{-1}\cdot \varphi (\gamma)^{-1}]) \nonumber  
\end{eqnarray*}

Let $q=|\langle s_x\rangle :\langle s_x^p\rangle |$. Thus  we have    $\langle s_x^p\rangle= \langle s_x^q\rangle\leq A_v$.  Define $\B''$ as the  graph of groups obtained from $\B$  by replacing the vertex group  $B_u=\pi_1^o(\mathcal{O}_u') \ast \langle gs_x^dg^{-1}\rangle $ by the group  $B_u'':=\pi_1^o(\mathcal{O}_u') \ast \langle gs_x^qg^{-1}\rangle$.  Then    $\B'$  is a sub-graph of groups  of $\B''$, and the group  $\pi_1(\B'', u_0)$ splits as 
$$\pi_1(\B'', u_0)=\pi_1(\B', u_0) \ast \gamma \langle gs_x^q g^{-1}\rangle   \gamma^{-1}.$$

Let $ \varphi'':\B''\rightarrow \A$ be the morphism   obtained from $\varphi$ by replacing the vertex homomorphism  $\varphi_u $ by the homomorphism $\varphi_u'':B_u''\rightarrow A_v$ induced by the  homomorphism  $B_u'\hookrightarrow  B_u\xrightarrow{\varphi_u} A_v$ 
and   the inclusion map $g\langle s_x^q\rangle g^{-1} \hookrightarrow A_v$. Fig.~\ref{fig:foldfinite} illustrates the definition of $\varphi'':\B''\rightarrow \A$ at the vertex $u$. 
\begin{figure}[h!]
\begin{center}
\includegraphics[scale=1]{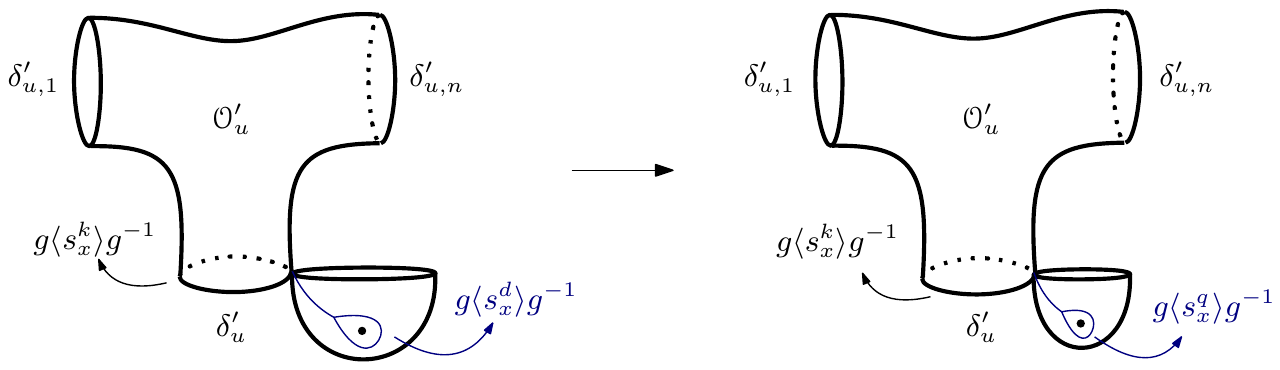}
\end{center}
\caption{The group  $B_u=\pi_1^o(\mathcal{O}_u')\ast g\langle s_x^d\rangle g^{-1}$ is replaced by $B_u''=\pi_1^o(\mathcal{O}_u')\ast g\langle s_x^q\rangle g^{-1}$.  }\label{fig:foldfinite}
\end{figure}

It is not hard to see that the morphism  $\varphi''$ satisfies conditions (2.i)-(2.iii) from Definition~\ref{def:marked}, and that  the fundamental group of  $\B'' $  splits as a free product of cyclic groups.   Moreover, the tuple $$\mathcal{T}'':=\mathcal{T}'\oplus ([\gamma \cdot  gs_x^p g^{-1}\cdot   \gamma^{-1}])$$ 
clearly generates $\pi_1(\B'', {u}_0)$ and, as    the previous paragraph shows,  
$$ \varphi_{\ast}'' (\mathcal{T}'')=\varphi_{\ast}''(\mathcal{T}') \oplus ([\varphi''(\gamma)\cdot g s_x^{p}  g^{-1}\cdot \varphi''(\gamma)^{-1}]) = \varphi_{\ast}(\mathcal{T}') \oplus ([\varphi(\gamma)\cdot g s_x^{p}  g^{-1}\cdot \varphi'(\gamma)^{-1}])$$ 
is Nielsen equivalent to $\varphi_{\ast}(\mathcal{T})$, which implies that  $ ((\B'', u_0''), \varphi'',  \mathcal{T}'')$ belongs to $\Omega(\mathcal{T}_0)$. 
As edge groups and homomorphisms do not change, we have $d((\B'', u_0''), \varphi'',  \mathcal{T}'') = d\M$.

If $q=|\langle s_x\rangle |$  then clearly $\rk(\B'')=\rk(\B)-1$,  and therefore the marked morphism  $((\overline{\B}, \overline{u}_0 ), \overline{\varphi}, \overline{\mathcal{T}}):=((\B'', u_0''), \varphi'',  \mathcal{T}'')$ lies in $\Omega(\mathcal{T}_0)$,  and  is has $c$-complexity smaller than $c\M=(n_1, n_2, n_3)$.  

If $1\leqslant q <k$, then  $((\B'', u_0''), \varphi'',  \mathcal{T}'')$ belongs to $\Omega_{\text{min}}(\mathcal{T}_0)$.  It now follows  from Lemma~\ref{lemma:adjfinite} that the decorated group 
$$(B_u'' , \varphi_u'', \{\alpha_{f }''(B_{f }'')\}_{f\in \st_1(u, \B'')})\cong  (\pi_1^o( \mathcal{O}_u')\ast \langle gs_x^q g^{-1}\rangle , \varphi_u'', \{C_{u,j}'\}_{1\leqslant j\leqslant n_u})$$
either folds peripheral subgroups or has an obvious relation. Now  Corollaries~\ref{cor:reduceedges} and \ref{cor:decreased} imply the existence of a marked morphism $((\overline{\B}, \overline{u}_0 ), \overline{\varphi}, \overline{\mathcal{T}})$ in $\Omega(\mathcal{T}_0)$ such that either (1) or (2)  occurs.      
\end{proof}
\end{proof}

\begin{lemma}{\label{lemma:nonisol}}
For all vertices of $\B$ the sets $\st_1(\cdot, \B)$ and $\st(\cdot ,B)$ coincide.  In particular,     $\st_1( \cdot , \B)\neq \emptyset.$  
\end{lemma}
\begin{proof}
In fact, if there is a vertex $w$  such that $\st_1(w, \B)$ is a proper subset  of $\st(w, B)$, then   there exists an edge $h\in EB$ such that ($i$) $B_h=1$,  and  ($ii$)  $\st_1(\alpha(h), \B)$ is non-empty.    Lemmas~\ref{lemma:noncoll}  and~\ref{lemma:almostcovering}  imply that  there exists an edge  $f$ in $ \st_1(\alpha(h), \B)$ with  $\varphi(f)=\varphi(h)$ such that $o_h =\varphi_{\alpha(h)}(b) o_{f}  \alpha_{\varphi(f)}(c)$, for some $b\in B_{\alpha(h)} $ and  $c\in A_{\varphi(f)}$, that is, $\varphi$ vioates   condition (F1) of Definition~\ref{def:folded}. This contradicts  Lemma~\ref{lemma:reduceedges}.  Therefore, the sets  $\st_1(w, \B)$ and $\st(w, B)$ must coincide, see Fig.~\ref{fig:full}.      
\begin{figure}[h!]
\begin{center}
\includegraphics[scale=1]{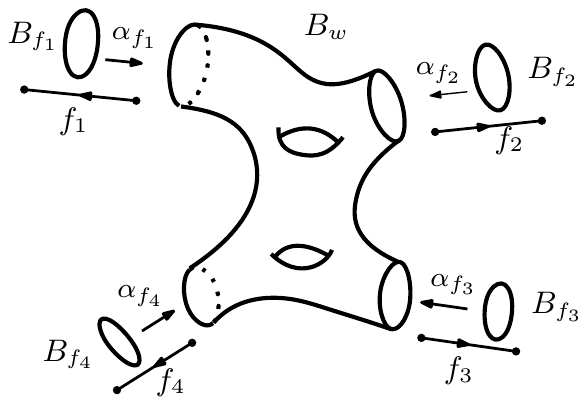}
\end{center}
\caption{$\st_1(w, \B)=\st(w,B)=\{f_1,f_2,f_3,f_4\}$.}\label{fig:full}
\end{figure}
\end{proof}

Lemmas~\ref{lemma:noncoll} and \ref{lemma:nonisol}, imply that, for any  vertex $w$ of $\B$ with $w\neq u$, there is an orbifold covering $\eta_w':\mathcal{O}_{w}'\rightarrow \mathcal{O}_{\varphi(w)}$ of finite degree   with    $\partial \mathcal{O}_w'= \delta_{w,1}', \ldots, \delta_{w, n_w}'$ such that 
$$(B_w, \varphi_w, \{\alpha_f(B_f)\}_{f\in \st_1(w, \B)})\cong (\pi_1^o(\mathcal{O}_w'), (\eta_w')_{\ast}, \{C_{w, j}'\}_{1\leqslant j\leqslant n_w}).$$  
Lemma~\ref{lemma:nonvertexinj} further implies that   $(B_u, \varphi_u(B_f)\}_{f\in \st_1(u, \B)})$ is isomorphic to $(\pi_1^o(\mathcal{O}_u'), (\eta_u')_{\ast}, \{C_{u, j}'\}_{1\leqslant j\leqslant n_u})$, 
  where $ \eta_u':\mathcal{O}_u'\rightarrow \mathcal{O}_v$ is a special almost orbifold covering with boundary $\partial \mathcal{O}_u'=\delta_{u,1}'\cup \ldots \cup \delta_{u, n_u}' \cup \delta_u'$, where $\delta_u'$ is the exceptional boundary component of $\mathcal{O}_u'$.   For each vertex $w$ of $\B$   denote  by  $\tau_w$ the bijection $\st_1(w, \B)\rightarrow \{1, \ldots, n_w\}$ given in the definition of isomorphic decorated groups.

Now we can complete the proof of Theorem~\ref{MainThm}.  Let $\overline{\mathcal{O}}$ be  the disjoint union  of the  orbifolds $\{\mathcal{O}_w' \ | \ w\in VB\}$  and let  $\overline{\eta}:\overline{\mathcal{O}}\rightarrow \mathcal{O}$ be  the orbifold map  induced by the  orbifold maps 
$$\mathcal{O}_w'\xrightarrow{\eta_w'}  \mathcal{O}_{\varphi(w)}\hookrightarrow \mathcal{O}$$ 
for $w\in VB$ (here we allow orbifolds with non-connected underlying surfaces).  Notice   that $\overline{\eta}(x)=\overline{\eta}(y)$ if, and only if, there is an edge $f$ of $\B$  such that $x\in \delta_{\alpha(f), \tau_{\alpha(f)}(f)}'$ and $y\in \delta_{\omega(f), \tau_{\omega(f)}(f^{-1})}'$.  Moreover, for any $f\in EB$, the maps 
$$\eta_{\alpha(f)}'|_{\delta_{\alpha(f), \tau_{\alpha(f)}(f)}'} \ \  \text{ and } \ \  \eta_{\omega(f)}'|_{\delta_{\omega(f), \tau_{\omega(f)}(f^{-1})}'}$$  
coincide.  Thus  the quotient $\mathcal{O}'$ of $\overline{\mathcal{O}}$ by the relation $ x\sim y$ if $\overline{\eta}(x)=\overline{\eta}(y)$   
is an  orbifold. The    the induced map  $\eta':\mathcal{O}'\rightarrow \mathcal{O}$ is  clearly a special almost orbifold covering with exceptional boundary component $\delta_u'\subset \partial \mathcal{O}_u'$. To finish the proof, observe that the morphism $\varphi:\B\rightarrow \A$ is induced by the almost orbifold covering $\eta'$, see Example~\ref{ex:induced}. Thus, there is an isomorphism $\sigma:\pi_1(\B, u_0)\rightarrow \pi^o(\mathcal{O}')$ such that   $ \eta_{\ast}' \circ \sigma =\varphi_{\ast}$.  Therefore the tuple  $\mathcal{T}':=\sigma(\mathcal{T})$ generates $\pi_1^o(\mathcal{O}')$,  and clearly $\eta_{\ast}'(\mathcal{T}')$ is Nielsen equivalent to $\mathcal{T}_0$.


\section{Proof of Corollary~\ref{cor:Louder} } 
In this section we show that any generating tuple of the fundamental group of  a sufficiently large orbifold whose cone points have order at most two, is either reducible or is Nielsen equivalent to a standard generating tuple. 
 
Let $\mathcal{T}$ be a generating tuple of $\pi_1^o(\mathcal{O})$, where $\mathcal{O}=F(p_1, \ldots, p_r)$ is a sufficiently large orbifold with either $r=0$ or $r\geqslant 1$ and $p_1=\ldots= p_r=2$. We will show that there exists an almost orbifold covering $\eta:\mathcal{O}'\rightarrow \mathcal{O}$ of degree one and a generating tuple $\mathcal{T}'$ of $\pi_1^o(\mathcal{O}')$ such that   $\eta_{\ast}(\mathcal{T}')$ is Nielsen equivalent to $\mathcal{T}$. Then, by remark~\ref{remark:deg1},    $\mathcal{T}$ is either reducible or it is Nielsen equivalent to a standard generating tuple.

By Theorem~\ref{MainThm}, there is a special almost orbifold covering $\eta:\mathcal{O}'\rightarrow \mathcal{O}$ and a generating tuple $\mathcal{T}'$ of $\pi_1^o(\mathcal{O}')$ such that $\eta_{\ast}(\mathcal{T}') \sim_{NE}\mathcal{T}$. It remains to show that $\text{deg}(\eta)=1$. Let $C\subseteq \partial \mathcal{O}'$ be the exceptional boundary component of $\mathcal{O}'$,  $x\in F$ the exceptional point of $\mathcal{O}$ and $D\subseteq F$ the exceptional disk of $\mathcal{O}$.  Since $p(x)\leqslant 2$,  it follows that the degree of the map  $\eta|_C:C\rightarrow \partial D$ is at most $2$,  and therefore   divides $p(x)$. Example~\ref{ex:almost1} implies that there exists an orbifold covering $\eta'':\mathcal{O}''\rightarrow \mathcal{O}$ of finite degree such that (i) $\mathcal{O}'\subseteq \mathcal{O}''$, (ii) $\mathcal{O}''- \text{Int}(\mathcal{O}')$ is a disk with at most one cone point of order $p(x)$, and (iii) $\eta$ is the resctriction of $\eta''$ to ${\mathcal{O}'}$. 

The fundamental group  of $\mathcal{O}''$ is obtained from $\pi_1^o(\mathcal{O}')$ by adding the relation $s^q=1$ where $s$ is the homotopy class  represented by the curve $C$ and $q=p/\text{deg}(\eta'|_{C})$.  Denote by $\mathcal{T}''$ the image of $\mathcal{T}'$ in $\pi_1^o(\mathcal{O}'')$. Thus    $\mathcal{T}''$  generates $\pi_1^o(\mathcal{O}'')$ and $\eta_{\ast}''(\mathcal{T}'')$ is Nielsen equivalent to $\mathcal{T}$.  Consequently   $\eta''$ is a $\pi_1$-surjective orbifold covering. This implies that $\mathcal{O}''=\mathcal{O}$.  Therefore,  $\mathcal{O}'=\mathcal{O}-\text{Int}(D)=(F-\text{Int}(D), p|_{F-\text{Int}(D)})$ and hence  $\eta:\mathcal{O}'\rightarrow \mathcal{O}$ has degree one.


\section{Proof of Proposition~\ref{proposition:1}} 
This section contains the proof of Proposition~\ref{proposition:1}.  The proof is organized as follows:
\begin{enumerate}
\item In subsection 6.1  we  fix an identification between the  fundamental group of the orientable small  orbifold  $\mathcal{O}=F(p_1,\ldots, p_r)$ and the   fundamental group of a graph of groups $\AO$ with trivial edge groups and finite cyclic vertex groups.     Then  we will  consider decorated groups over $\pi_1^o(\mathcal{O})$  that are induced by morphisms $\B\rightarrow \A$  that come equipped  with two collections of $\B$-paths $\{p_j\}_{1\leqslant j\leqslant n}$ and $\{\gamma_j\}_{1\leqslant j\leqslant n}$.

\item In subsection 6.2 we consider tame decorated morphisms. These induce collapsible decorated groups over $\pi_1(\AO, v_1)\cong \pi_1^o(\mathcal{O})$. Then we find sufficient conditions that guarantee that the induced decorated group  has an obvious relation or folds peripheral subgroups.

\item    In subsection 6.3 we define the local graph of a decorated morphism. This graph will reveal   local properties of the $\B$-paths $\{p_j\}_{1\leqslant j\leqslant n}$    as well as local properties of the morphism $\B\rightarrow \AO$.

\item In subsection 6.4 we will consider the set of all decorated morphisms $\B\rightarrow \AO$  that induce decorated groups equivalent to  $(G, \eta, \{G_j\}_{j\in J})$.  We then observe that  a decorated morphism   in this set that has the   minimal number of edges,   either  projects onto a decorated group over $\pi_1(\AO, v_1)$ as in  $(\alpha.1)-(\alpha.3)$  or  induces a decorated group over $\pi_1(\AO, v_1)$  as in  $(\beta)$.
\end{enumerate}


\subsection{The graph of groups associated to an orientable small orbifold.}
For the remainder of this section let $\mathcal{O}=F(p_1, \ldots, p_{r_{\mathcal{O}}})$ be an orientable small orbifold  with  boundary components  $\delta_1,\ldots, \delta_{q_{\mathcal{O}}}$, where $q_{\mathcal{O}}\geqslant 1$ and $0\leqslant r_{\mathcal{O}}\leqslant 2$. Recall that  $q_{\mathcal{O}}=1   \Rightarrow r_{\mathcal{O}}=2 $ .  Let  further  $x_1, x_2$  be points of $\mathcal{O}$ such that the order $p(x_i)$ of $x_i$ is equal to $p_i$ for $i\leqslant r_{\mathcal{O}}$ and it is equal to $1$ for $i>r_{\mathcal{O}}$.    Thus the fundamental group of $\mathcal{O}$  has the following  presentation
$$\langle s_{x_1}, s_{x_2} , t_1, \ldots, t_{q_{\mathcal{O}}} \ | \ s_{x_1}^{p(x_1)}, s_{x_2}^{p(x_2)} , s_{x_1}s_{x_2}=t_1\cdot \ldots \cdot t_{q_{\mathcal{O}}}\rangle,$$
where the generators $s_{x_1}$ and $s_{x_2}$ correspond  to the points $x_1$ and $x_2$, and  the generators $t_1, \ldots, t_{q_{\mathcal{O}}}$ correspond to the boundary components $\delta_1, \ldots, \delta_{q_{\mathcal{O}}}$.  Note that this presentation is slightly different from the standard one as  the elements $s_{x_1}$ and $s_{x_2}$ are possibly trivial.

\smallskip 

The first step to prove Proposition~\ref{proposition:1} is to associate a graph of groups $\AO$ to the orientable  small orbifold $\mathcal{O}$. $\AO$ is defined as follows:  the underlying graph $A^{\mathcal{O}}$ of $\AO$ has two vertices $v_1,v_2$ and edge-set $\{e_i^{\pm1} \ | \ 1\leqslant i\leqslant q_{\mathcal{O}}\}$ with $\alpha(e_i)=v_1$ and $\omega(e_i)=v_2$ for  all $i$.  All edge groups in  $\AO$ are trivial,  and the vertex  groups are defined as $A_{v_i}^{\mathcal{O}}=\langle s_{v_i} \ | \ s_{v_i}^{p(x_i)}\rangle\cong \mathbb{Z}_{p(x_i)}$ for  $i=1,2$.
\begin{figure}[h!]
\begin{center}
\includegraphics[scale=1]{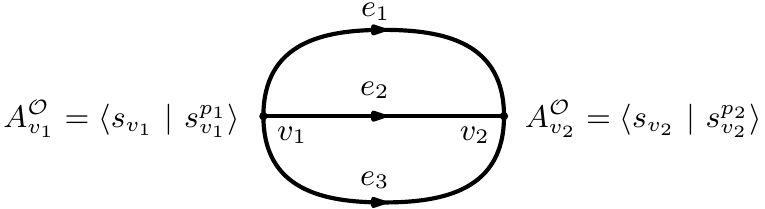}
\end{center}
\caption{The graph of   groups $\AO$ associated to   $\mathcal{O}=F(p_1,p_2)$ where $F$ is a pair of pants.}\label{fig:fundamentalgroup3}
\end{figure}

\emph{ 
For a given   integer $d$, we define $e_d:=e_s\in EA^{\mathcal{O}}$, where $s\in \{1,\ldots, q_{\mathcal{O}}\}$ is  such that $d \equiv s \ (\text{mod } q_{\mathcal{O}})$. 
 }

\smallskip

We now  define paths in $\AO$ that naturally correspond to the boundary  components of $\mathcal{O}$.   For each  $1\leqslant i\leqslant  q_{\mathcal{O}}$, let    $c_i$ be the $\AO$-path  
$$c_i:=s_{v_1}^{\varepsilon_i},e_i,s_{v_2}^{\varepsilon_i}, e_{i+1}^{-1},1$$ where $\varepsilon_1=1$ and $\varepsilon_2=\ldots =\varepsilon_{q_{\mathcal{O}}}=0$. The subgroup of $\pi_1(\AO, v_1)$ generated by the element $[c_i]$ will be denoted by $C_i$.  Note that there is an isomorphism $\theta:\pi_1^o(\mathcal{O})\rightarrow \pi_1(\AO, v_1)$  that sends $t_i$ onto the element  $[c_i]$  for $1\leqslant i\leqslant q_{\mathcal{O}}$,  sends $s_{x_1}$ onto $[s_{v_1}]$   and  $s_{x_2}$ onto $[1,e_1, s_{v_2}, e_1^{-1},1]$.

\begin{remark}
The isomorphism  $\theta:\pi_1^o(\mathcal{O})\rightarrow \pi_1(\AO, v_1)$ might be visualized as follows.   The geometric realization of the   graph $A^{\mathcal{O}}$, which we also denote by $A^{\mathcal{O}}$,  may be embedded in the underlying  surface $F$ of $\mathcal{O}$ in such a way that $A^{\mathcal{O}}$ is a strong deformation retract of $F$, and that the vertices $ v_1$ and $v_2$ are mapped to the points $x_1$ and $x_2$ respectively.  Figure~\ref{fig:strong} shows $A^{\mathcal{O}}$ embedded in $F$. In this deformation, the boundary component $\delta_i\subseteq \partial F$ ($1\leqslant i\leqslant q_{\mathcal{O}}$) is mapped onto the path $e_i e_{i+1}^{-1}$ in $A^{\mathcal{O}}$.  
\begin{figure}[h!]
\begin{center}
\includegraphics[scale=1]{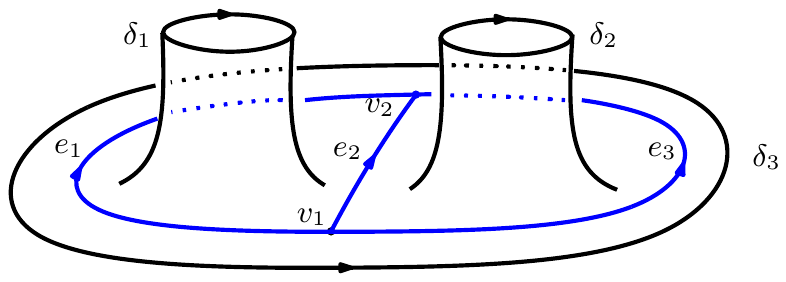}
\end{center}
\caption{The graph   $A^{\mathcal{O}}$ embedded in the underlying surface $F$ of $\mathcal{O}$.}\label{fig:strong}
\end{figure}
\end{remark}

\emph{ For a given   integer $d$, we define $c_d:=c_s\in \{c_1,\ldots, c_{q_{\mathcal{O}}}\}$,  where $s\in \{1,\ldots, q_{\mathcal{O}}\}$ is  such that $d \equiv s \ (\text{mod } q_{\mathcal{O}})$.}
 \smallskip 
 
For the remainder of the paper  we will identify $\pi_1^o(\mathcal{O})$ with $\pi_1(\AO,v_1)$ via the isomorphism ${\theta}$,  and  consequently consider decorated groups over $\pi_1(\AO, v_1)$ instead of $\pi_1^o(\mathcal{O})$.  We  will study decorated groups  over $\pi_1(\AO, v_1)$  that are induced  by  decorated morphisms defined below:  
\begin{definition}{\label{def:decmor}}
A \emph{decorated morphism  over} $\AO$ is a tuple $((\B ,u_1),\phi,  \{p_j\}_{1\leqslant j\leqslant n}, \{\gamma_j\}_{1\leqslant j\leqslant n})$ 
where
\begin{enumerate}
\item $\mathbb{B}$ is a connected graph of groups with $B_f=1$ for all $f\in EB$,  and $u_1$ is the base vertex of $\mathbb{B}$.

\item $\phi$ is a morphism from $\mathbb{B}$ to $\AO$  with $\phi(u_1)=v_1$.

\item $p_1,\ldots p_n$ are closed  paths in $\mathbb{B}$  such that, for each $1\leqslant j\leqslant n$, there is an index $i(j)\in \{1,\ldots q_{\mathcal{O}}\}$ and a cyclic permutation $p_j'$ of $p_j$, such that  
 $\phi(p_j')= a_jc_{i(j)}^{z_j}a_j^{-1} $ 
for some  positive integer $z_j$ and some element  $a_j$ of $ A_{v_1}^{\mathcal{O}}$. 

\item $ \gamma_1, \ldots, \gamma_n$ are  paths  in $\mathbb{B}$ with initial vertex $\alpha(\gamma_j)=u_1$ and terminal vertex $\omega(\gamma_j)=\alpha(p_j)$ ($=\omega(p_j)$).   
\end{enumerate} 
\end{definition}

A decorated morphism over $\AO$ naturally defines a decorated group over $\pi_1(\AO, v_0)$  as follows. By definition,   $\phi:\mathbb{B}\rightarrow \AO$ maps some cyclic permutation of $p_j$   onto the  path   
 $$ a_jc_{i(j)}^{z_j} a_j^{-1}  = a_j(s_{v_1}^{\varepsilon_{i(j)}} , e_{i(j)}, s_{v_2}^{\varepsilon_{i(j)}}, e_{i(j)+1}^{-1},1)^{z_j}a_j^{-1}$$ 
for some positive integer $z_j$ and some element $a_j$ of $A_{v_1}^{\mathcal{O}}$. More precisely, we can write $p_j= p_{j,1}p_{j,2}$  such that 
$$ \phi(p_{j,2}p_{j,1})= a_j(s_{v_1}^{\varepsilon_{i(j)}} , e_{i(j)}, s_{v_2}^{\varepsilon_{i(j)}}, e_{i(j)+1}^{-1},1)^{z_j}a_j^{-1}.$$
Then  the induced homomorphism  $\phi_{\ast}:\pi_1(\B, u_1)\rightarrow \pi_1(\AO, v_1)$ maps the element $h(p_j, \gamma_j):=[\gamma_j  p_j \gamma_j^{-1}]$  onto the element   $o_{H(p_j, \gamma_j)}[c_{i(j)}]^{z_j}o_{H(p_j, \gamma_j)}^{-1}$, where  $o_{H(p_j, \gamma_j)} = [\phi(\gamma_jp_{j,1}) a_j]$. This means that    the  subgroup  $H(p_j, \gamma_j):=\langle h(p_j, \gamma_j)\rangle\leq \pi_1(\B, u_1)$  
  is a peripheral subgroup of type $(o_{H(p_j, \gamma_j)}, i(h))$.   Therefore, 
$$(\pi_1(\mathbb{B},u_1), \phi_{\ast}, \{H(p_j,\gamma_j)\}_{1\leqslant j\leqslant n})$$ 
is a decorated group over $\pi_1(\AO,v_1)$. We will say that  $(\pi_1(\mathbb{B},u_1), \phi_{\ast}, \{H(p_j,\gamma_j)\}_{1\leqslant j\leqslant n})$ is \emph{induced by the decorated morphism  $((\B, u_0), \phi, \{p_j\}_{1\leqslant j\leqslant n}, \{\gamma_j\}_{1\leqslant j\leqslant n})$}.

\begin{example}{\label{ex:decmorph}}
Let $\mathcal{O}=\mathbb{D}^2(2,2)$.  A decorated morphism $((\B, u_1), \phi, \{p_1, p_2\}, \{\gamma_1, \gamma_2\})$ over $\AO$ is depicted in Fig.~\ref{fig:collapsible}. The paths $p_1$ and $p_2$ are given by  $p_1= 1,f_1, 1, f_2^{-1}, 1 $ and $p_2=1,f_2, 1, f_3^{-1}, 1$. The paths    $\gamma_1$ and $\gamma_2$ are both equal to the trivial path $1$  at $u_1$.   Note that $$\phi(p_1)= s_{v_1}, e_1, s_{v_2}, e_1^{-1}, 1=c_1 \ \ \ \text{ and } \ \ \ \phi(p_2)= 1, e_1, s_{v_2}, e^{-1}, s_{v_1}= s_{v_1}^{-1} c_1 s_{v_1}.$$
The fundamental group of $\B$ splits as the free product $H(p_1, \gamma_1)\ast H(p_2, \gamma_2)$.  Moreover,  $H(p_1, \gamma_1)$ is a peripheral subgroup of type $(1, 1)$ and $ H(p_2, \gamma_2)$ is a peripheral subgroup of type $(s_{v_1}^{-1}, 1)$.  
\begin{figure}[h!]
\begin{center}
\includegraphics[scale=1]{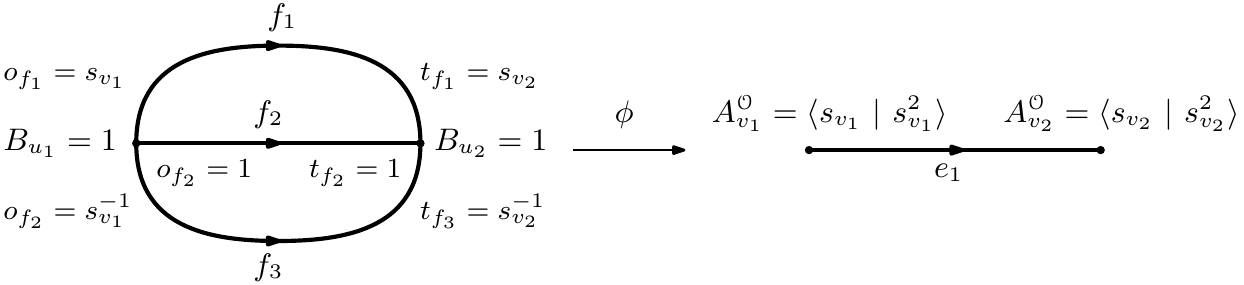}
\end{center}
\caption{A decorated morphism over $\AO$.}{\label{fig:collapsible}}
\end{figure}
\end{example}
 
\subsection*{Redecoration.}  Let  $((\mathbb{B},u_1), \phi, \{p_j\}_{1\leqslant j\leqslant n}, \{\gamma_j\}_{1\leqslant j\leqslant n})$ be a decorated morphism over $\AO$. Suppose that, for each $1\leqslant j\leqslant n$,  the path $p_j'$ is a cyclic permutation of $p_j$ and  that     $\gamma_j'$   is an arbitrary  path in $\mathbb{B}$ with initial vertex  $\alpha(\gamma_j')= u_1$ and terminal vertex  $\omega(\gamma_j')=\alpha(p_j')$.   We  define a new decorated morphism over $\AO$ as follows:
\begin{enumerate}
\item[(1)] replace the paths $p_{1}, \ldots, p_{n}$ by the paths $p_{1}', \ldots, p_{n}'$.

\item[(2)] replace the   paths $\gamma_1,\ldots, \gamma_n$ by the paths $\gamma_1', \ldots, \gamma_n'$.
\end{enumerate}

We  will call the decorated morphism $((\mathbb{B},u_1), \phi, \{p_j'\}_{1\leqslant j\leqslant n}, \{\gamma_j'\}_{1\leqslant j\leqslant n})$ a \emph{redecoration} of $\D$.  A straightforward calculation shows that $$(\pi_1(\mathbb{B},u_1),\phi_{\ast}, \{H(p_j', \gamma_j')\}_{1\leqslant j\leqslant n} ) \cong (\pi_1(\mathbb{B},u_1), \phi_{\ast}, \{H(q_j,\gamma_j)\}_{1\leqslant j\leqslant n} ).$$

\subsection*{Foldings.} The main feature of graph of groups morphisms is that they can be modified by auxiliary moves, vertex morphisms and folds without substantially  changing its properties.  We have to explain how a decorated morphism is modified in these situations. As any fold is the product of auxiliary moves, vertex morphisms and elementary folds, we need only to describe the situation for these cases. Since edge groups of $\AO$ are trivial, there is no  need to deal with auxiliary moves of type A1. For the remainder of this subsection let $((\mathbb{B},u_1), \phi, \{p_j\}_{1\leqslant j\leqslant n}, \{\gamma_j\}_{1\leqslant j\leqslant n})$ be a decorated morphism over $\AO$.

 \smallskip

 \noindent{\emph{Auxiliary move of type A0}.}   Suppose that the morphism $\phi':\mathbb{B} \rightarrow \AO$ is obtained from $\phi:\mathbb{B}\rightarrow \AO$ by an auxiliary move of type A0 applied to the vertex $u\neq u_1$ with element $a\in A_{\phi(u)}^{\mathcal{O}}$. To see that  the tuple  
 $$((\mathbb{B},u_1),\phi', \{p_j\}_{1\leqslant j\leqslant n}, \{\gamma_j\}_{1\leqslant j\leqslant n})$$  
 is a decorated morphism over $\AO$ we need to verify that condition (3) of Definition~\ref{def:decmor} holds. But this follows immediately from  the definition of an auxilairy move of type A0 since for any  closed path $p$ in $\B$ we have  
\begin{equation*}
    \phi'(p) = \begin{cases}
                \  \phi(p)            & \text{ if \ \ }  \ \alpha(p)\neq u \\
               \ a\phi(p)a ^{-1}       & \text{ if \ \ }  \ \alpha(p)=u .
           \end{cases}
\end{equation*}
Consequently $((\mathbb{B},u_1),\phi', \{p_j\}_{1\leqslant j\leqslant n}, \{\gamma_j\}_{1\leqslant j\leqslant n})$ is a decorated morphism over $\AO$. We say that it is  obtained from $((\mathbb{B},u_1),\phi , \{p_j\}_{1\leqslant j\leqslant n}, \{\gamma_j\}_{1\leqslant j\leqslant n})$ by an \emph{auxiliary move of type A0}.

It follows from Lemma~\ref{lemma:auxmoves} that $\phi_{\ast}=\phi_{\ast}'$ and hence  $((\mathbb{B},u_1), \phi', \{p_j\}, \{\gamma_j\})$ and $\D$ induce the same decorated group over $\pi_1(\AO,v_1)$.

\smallskip

\noindent{\emph{Auxiliary moves of type A2:}}   Suppose that the graph of groups morphism $\phi': \mathbb{B}\rightarrow \AO$ is obtained from $\phi$ by an auxiliary move of type A2.  We define a new  decorated morphism over $\AO$ as follows.

As the edge groups  in $\B$ are trivial, Lemma~\ref{lemma:simpleadjustment}  gives  a graph of groups isomorphism  $\sigma:\B\rightarrow \B$ such that $\phi' \circ \sigma =\phi$.  For each $j\in \{1,\ldots, n\}$, define  $p_j':=\sigma(p_j)$ and   $\gamma_j':=\sigma(\gamma_j)$. Condition (3) of Definition~\ref{def:decmor}    follows from the definition of $p_j'$ and the fact that $ \phi'\circ \sigma\phi$ . Thus  $((\B,u_1), \phi', \{p_j'\}_{1\leqslant j\leqslant n}, \{\gamma_j'\}_{1\leqslant j\leqslant n})$  is a decorated morphism over $\AO$.

We will sat that $((\B,u_1), \phi', \{p_j'\}_{1\leqslant j\leqslant n}, \{\gamma_j'\}_{1\leqslant j\leqslant n})$ is  obtained  from $\D$  by \emph{an auxiliary move of type A2}.   Notice that  the decorated group $(\pi_1(\mathbb{B}, u_1), \phi_{\ast}', \{H(p_j', \gamma_j')\}_{1\leqslant j\leqslant n})$ is isomorphic to  $(\pi_1(\mathbb{B}, u_1), \phi_{\ast}, \{H(p_j, \gamma_j)\}_{1\leqslant j\leqslant n})$.

\smallskip  
 
\noindent{\emph{Elementary folds and vertex morphisms:}}   Suppose that the graph of groups morphism $\phi': \B'\rightarrow \AO$ is obtained from $\phi$ by either an elementary fold or by a vertex morphism. Similarly as done in the case of  an   auxiliary move  of type A2,   we use  Lemmas~\ref{lemma:fold} and  \ref{lemma:vertexmorphism} to define  a new  decorated morphism 
$$((\B',u_1'), \phi', \{p_j'\}_{1\leqslant j\leqslant n}, \{\gamma_j'\}_{1\leqslant j\leqslant n})$$ 
over $\AO$   which we will say to be obtained from $\D$ by an \emph{elementary move or by a vertex morphism}, accordingly.

Note that if $((\B',u_1'), \phi', \{p_j'\}_{1\leqslant j\leqslant n}, \{\gamma_j'\}_{1\leqslant j\leqslant n})$ is obtained from $\D$ by an elementary fold, then    Lemma~\ref{lemma:fold} implies  that  
$$(\pi_1(\B, u_1), \phi_{\ast}, \{H(p_j, \gamma_j)\}_{1\leqslant j\leqslant n})\cong  (\pi_1(\B', u_1'), \phi_{\ast}', \{H(p_j', \gamma_j')\}_{1\leqslant j\leqslant n}).$$
 If $((\B',u_1'), \phi', \{p_j'\}_{1\leqslant j\leqslant n}, \{\gamma_j'\}_{1\leqslant j\leqslant n})$ is obtained from $\D$ by a vertex morphism, then  Lemma~\ref{lemma:vertexmorphism} implies   that 
 $$(\pi_1(\B, u_1), \phi_{\ast}, \{H(p_j, \gamma_j)\}_{1\leqslant j\leqslant n}) \twoheadrightarrow(\pi_1(\B', u_1'), \phi_{\ast}', \{H(p_j', \gamma_j')\}_{1\leqslant j\leqslant n}) .$$

The proof of the following result is an immediate consequence of the previous paragraphs. 
\begin{lemma}{\label{lemma:3}}
Suppose that for each pair $(r,s)\in \{1,2\}\times \{1,2\}$ a decorated morphism   
$$((\B^{r,s},u_1^{r,s}), \phi^{ r,s }, \{p_j^{ r,s }\}_{1\leqslant j\leqslant n}, \{\gamma_j^{ r,s }\}_{1\leqslant j\leqslant n})$$  
over $\AO$ is given such  that the following diagram commutes:
$$
\begin{tikzcd}
((\B^{1,1},u_1^{1,1}), \phi^{1,1}, \{p_j^{1,1}\}_{1\leqslant j\leqslant n}, \{\gamma_j^{1,1}\}_{1\leqslant j\leqslant n})  \arrow[d,"\text{ Elementary  Fold  IA/IIIA\ }"'] &    ((\B^{2,1},u_1^{2,1}), \phi^{2,1}, \{p_j^{2,1}\}_{1\leqslant j\leqslant n},  \{\gamma_j^{2,1})\}_{1\leqslant j\leqslant n}) \arrow[d,"\text{\ Elementary      Fold IA/IIIA}"]  \\ 
((\B^{1,2},u_1^{1,2}), \phi^{1,2}, \{p_j^{1,2}\}_{1\leqslant j\leqslant n}, \{\gamma_j^{1,2})\}_{1\leqslant j\leqslant n}) \arrow[r, "\text{A2}" ']                                &  ((\B^{2,2},u_1^{2,2}), \phi^{2,2}, \{p_j^{2,2}\}_{1\leqslant j\leqslant n}, \{\gamma_j^{2,2}\}_{1\leqslant j\leqslant n})
\end{tikzcd}$$ 
that is, such that the following hold:
\begin{enumerate}
\item[(a)]  $((\B^{r,2},u_1^{r,2}), \phi^{r,2}, \{p_j^{r,2}\}_{1\leqslant j\leqslant n}, \{\gamma_j^{r,2}\}_{1\leqslant j\leqslant n})$ is obtained from 
 $$((\B^{r,1},u_1^{r,1}), \phi^{r,1} , \{p_j^{r,1}\}_{1\leqslant j\leqslant n}, \{\gamma_j^{r,1}\}_{1\leqslant j\leqslant n} )$$  
by an elementary fold.

\item[(b)]  $((\B^{2,2},u_1^{2,2}), \phi^{2,2}, \{p_j^{2,2}\} , \{\gamma_j^{2,2}\})$ is obtained from 
  $$((\B^{1,2},u_1^{1,2}), \phi^{1,2}, \{p_j^{1,2}\}_{1\leqslant j\leqslant n}, \{\gamma_j^{1,2}\}_{1\leqslant j\leqslant n})$$ 
by an auxiliary move of type A2.
\end{enumerate}

Then  
$$(\pi_1(\B^{ 1,1 },u_1^{ 1,1 }), \phi_{\ast}^{ 1,1 }, \{H_j^{ 1,1 }\}_{1\leqslant j\leqslant n} ) \cong (\pi_1(\B^{ 2,1 },u_1^{ 2,1 }), \phi_{\ast}^{ 2,1 }, \{H_j^{ 2,1 }\}_{1\leqslant j\leqslant n }),$$ 
 where $H_j^{ 1,1 }:=H(p_{j}^{ 1,1 },\gamma_j^{ 1,1 })$ and $H_j^{ 2,1 }:=H(p_{j}^{ 2,1 },\gamma_j^{ 2,1 })$ for $1\leqslant j\leqslant n$. 
\end{lemma}


\subsection{Tame decorated morphisms.}  
Now we will consider a special class of decorated morhisms over $\AO$, called tame. We will see that tame decorated morphisms   induce (up to isomorphism)    collapsible  decorated groups over $\pi_1(\AO, v_1)$. All definitions in this subsection  were introduced by L. Louder in the context of graphs and graphs of graphs morphisms,  see ~\cite[Definitions 3.5 and 7.1]{Louder}. Here we simply adjust Louder's definitions  to the setting of decorated  morphisms.

\smallskip 

For a given subset $S=\{f_1, \ldots, f_r\}\subseteq  EB$  we will denote by $\B_S$, or alternatively  by $\B_{f_1, \ldots, f_r}$,  the sub-graph of groups of $\B$   carried by the sub-graph $B-\{f_1^{\pm1}, \ldots, f_r^{\pm 1}\}\subseteq B$.

\begin{definition}{\label{def:collapsible}}
A decorated morphism  $((\B, u_1), \phi, \{p_j\}_{1\leqslant j\leqslant n}, \{\gamma_j\}_{1\leqslant j\leqslant n})$ over $\AO$ is called  \emph{collapsible} provided that  for each  non-empty subset $K\subseteq \{1,\ldots, n\}$ there is   an edge $f_K\in EB$  and $k\in K$ such that (1) $p_k=p_k'  (1,f_K,1)  p_k''$  and the paths $p_k'$ and $p_k''$ are contained in $\B_{f_K}$, and (2) for each   $k'\in K$ with $k'\neq k$, the path $p_{k'}$ is contained in $\B_{f_K}$. 
\end{definition}

Let  $((\B, u_1), \phi, \{p_j\}_{1\leqslant j\leqslant n}, \{\gamma_j\}_{1\leqslant j\leqslant n})$ be a collapsible decorated morphism over $\AO$. It follows from  the definition  that there  are  distinct edges $f_1,\ldots, f_n$ of $\B$,  and a bijection $\nu:\{1,\ldots, n\}\rightarrow \{1,\ldots, n\}$ such that (1) $\{f_1,\ldots, f_n\}\cap \{f_1^{-1},\ldots, f_n^{-1}\}=\emptyset$, and (2)  for  each  $1\leqslant k\leqslant  n$  the  path $p_{\nu(k)}$ decomposes   as    $p_{\nu(k)}=p_{\nu(k)}' (1,f_k,1) p_{\nu(k)}''$  with   $p_{\nu(k)}'$ and $p_{\nu(k)}''$   contained in the sub-graph of groups $\B_{f_{1},\ldots, f_k}\subseteq \B$.  Item (2)  tells us   that   the sub-graph $B-\{f_1^{\pm1}, \ldots, f_n^{\pm1}\}\subseteq B$ is connected.

Now we define a redecoration  of $((\B, u_1), \phi, \{p_j\}_{1\leqslant j\leqslant n}, \{\gamma_j\}_{1\leqslant j\leqslant n})$ that induce a  collapsible decorated group  over $\pi_1(\AO, v_1)$.   For each $1\leqslant j\leqslant n$, let $q_{j}$ be  an arbitrary cyclic permutation    of $p_{j}$ and  let  $\delta_{j}$ be an arbitrary   path contained in  $\mathbb{B}_{f_1,\ldots, f_{n}}\subseteq\mathbb{B}$   such that     $\alpha(\delta_{j})= u_1$ and    $\omega(\delta_{j})=\alpha(q_{j})$.
\begin{claim}
The decorated morphism $((\mathbb{B}, u_1), \phi, \{q_j\}_{1\leqslant j\leqslant n}, \{\delta_j\}_{1\leqslant j\leqslant n})$ induces a collapsible decorated group over $\pi_1(\AO, v_1)$. 
\end{claim}
\begin{proof}[Proof of Claim]
In fact, condition (2) implies that  the path $q_{\nu(k)}$ can  be decomposed as  $q_{\nu(k)} =q_{\nu(k)}'(1,f_k,1)q_{\nu(k)}''$ 
with $q_{\nu(k)}'$ and $q_{\nu(k)}''$ contained in $\B_{f_1,\ldots, f_k}$.   Hence the paths  $\delta_{\nu(k)}  q_{\nu(k)}'$ and $\delta_{\nu(k)}  (q_{\nu(k)}'')^{-1}$ lie in the sub-graph of groups  $\mathbb{B}_{f_1,\ldots, f_{k}}=(\B_{f_1,\ldots, f_{k-1}})_{f_k}$.  The edge groups of $\mathbb{B}$ are trivial,  and so  $\pi_1(\mathbb{B}_{f_1,\ldots , f_{k-1}}, u_1)\leq \pi_1(\mathbb{B}, u_1)$ splits as   
$$\pi_1(\B_{f_{1},\ldots, f_{k-1}},u_1)=\pi_1(\B_{f_1,\ldots, f_k},u_1)\ast  \langle[\delta_{\nu(k)}  q_{\nu(k)}'   (1,f_k,1)  q_{\nu(k)}''  \delta_{\nu({k})}^{-1}] \rangle.$$
As  $[\delta_{\nu(k)}  q_{\nu(k)}   (1,f_k,1) q_{\nu(k)}''   \delta_{\nu(k)}^{-1}]=h(q_{\nu(k)}, \delta_{\nu(k)})$, we see that 
$$\pi_1(\B_{f_1, \ldots, f_{k-1}}, u_1)=\pi_1(\B_{f_1,\ldots, f_k}, u_1)\ast H(q_{\nu(k)}, \delta_{\nu(k)})  \ \ \text{ for all }  \ \ 1\leqslant k\leqslant n.  $$
By induction  we obtain  $\pi_1(\B, u_1)=\pi_1(\B_{f_1,\ldots, f_n}, u_1)\ast H(q_{1}, \delta_{1})\ast \ldots \ast H(q_{n}, \delta_{n})$, that is, the decorated group $( \pi_1(\B, u_1), \phi_{\ast}, \{H(q_j, \gamma_j)\}_{1\leqslant j\leqslant n})$ is collapsible. 
\end{proof}

The previous claim implies   that a collapsible decorated morphim over $\AO$  has a whole family of redecorations that  induce    collapsible decorated groups over $\pi_1(\AO, v_1)$. More precisely, we have: 
\begin{lemma}{\label{lemma:collapsible}}
If  $((\B, u_1), \phi, \{p_j\}_{1\leqslant j\leqslant n}, \{\gamma_j\}_{1\leqslant j\leqslant n})$ is a collapsible decorated morphism over $\AO$,  then there are   edges   $f_1,\ldots, f_n\in EB$  and a bijection $\nu:\{1,\ldots, n\}\rightarrow \{1,\ldots, n\}$  such that the following hold:
\begin{enumerate}
\item[(1)] $\{f_1,\ldots, f_n\}\cap \{f_1^{-1},\ldots, f_n^{-1}\}=\emptyset$.

\item[(2)] for each $1\leqslant k\leqslant n$, the path      $p_{\nu(k)}$ decomposes as  $p_{\nu(k)}=p_{\nu(k)}' (1,f_k,1) p_{\nu(k)}''$   
with $p_{\nu(k)}'$ and $p_{\nu(k)}''$ contained in  $\B_{f_{1},\ldots, f_k}\subseteq \B$.
\end{enumerate}
Furthermore,  if   $((\B, u_1), \phi, \{q_j\}_{1\leqslant j\leqslant n} , \{\delta_j\}_{1\leqslant j\leqslant n})$ is a redecoration  of $((\B, u_1), \phi, \{p_j\}_{1\leqslant j\leqslant n}, \{\gamma_j\}_{1\leqslant j\leqslant n})$ with  $\delta_{1}, \ldots, \delta_{n}$  contained in   $\B_{f_1,\ldots, f_n}$, then  for all $1\leqslant k\leqslant n$ the group   $\pi_1(\B_{f_1,\ldots, f_{k-1}}, u_1)$ splits as  
$$\pi_1(\B_{f_1,\ldots, f_{k-1}}, u_1)=\pi_1(\B_{f_1,\ldots, f_k}, u_1)\ast H(q_{\nu(k)}, \delta_{\nu(k)}).$$ 
In particular,  $\pi_1(\B, u_1)=\pi_1(\B_{f_1,\ldots, f_n},u_1)\ast H(q_{1}, \delta_{1})\ast \ldots \ast H(q_{n}, \delta_{n}).$ 
\end{lemma}

\begin{definition}{\label{def:squares}} Let $((\B, u_1), \phi, \{p_j\}_{1\leqslant j\leqslant n}, \{\gamma_j\}_{1\leqslant j\leqslant n})$  be a decorated morphism over $\AO$. 
\begin{enumerate}
\item[(S.1)] We say that $\D$ \emph{folds peripheral paths}  if there   exist an edge $f\in EB$ and distinct elements   $j,  k \in \{1,\ldots, n\}$ with  $i(j)=i(k) \in \{1, \ldots, q_{\mathcal{O}}\}$ such that 
$$p_{j}=p_{j}' (1,f,1)  p_{j}'' \ \ \text{ and } \ \ p_{k}=p_{k}' (1,f,1)  p_{k}'' .$$

\item[(S.2)] We say that $\D$ \emph{self-folds} if there  exist  an edge $f\in EB$ and $j\in \{1,\ldots, n\}$ such that 
$$p_j=p_j' (1,f,1)  p_j'' (1,f,1) p_j'''.$$ 
\end{enumerate}
We say that $\D$ \emph{folds squares} if it either  folds peripheral paths or  self-folds.
\end{definition}
\begin{figure}[h!]
\begin{center}
\includegraphics[scale=1]{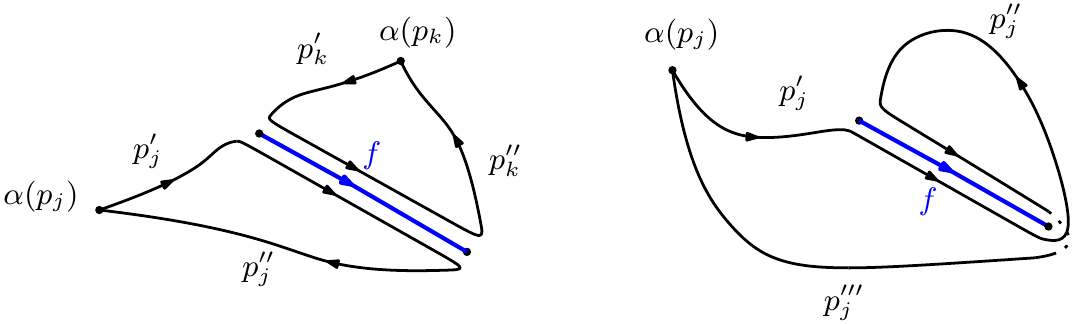}
\end{center}
\caption{Case (S.1) on the left and case (S.2) on the right.}{\label{fig:folds1}}
\end{figure}

\begin{remark}{\label{remark:6.9}}
Note that auxiliary moves  and vertex morphisms do not change the structure of the underlying graph of a graph of groups.  Thus the properties of being collapsible and of not folding squares   are clearly  preserved by these moves. 
\end{remark}

\begin{definition}{\label{def:tame}}  We say that  a decorated morphism  $((\B, u_1), \phi, \{p_j\}_{1\leqslant j\leqslant n}, \{\gamma_j\}_{1\leqslant j\leqslant n})$   over $\AO$ is   \emph{vertex injective} if the morphism  $\phi:\mathbb{B}\rightarrow\AO$ is vertex injective.  We  further say that  $((\B, u_1), \phi, \{p_j\}_{1\leqslant j\leqslant n}, \{\gamma_j\}_{1\leqslant j\leqslant n})$   is  \emph{tame} if it is  vertex injective, collapsible and   does not fold squares.
\end{definition}

We have seen that  a collapsible decorated morphism over $\AO$ induces,  up to a redecoration,   a collapsible decorated group over $\pi_1(\AO, v_1)$.  As the following lemma shows,  a decorated group over $\pi_1(\AO, v_1)$ is induced by a (tame) decorated morphim provided that  it is strongly collapsible.   
\begin{lemma}{\label{lemma:initial}}
If $(G,\eta, \{G_j\}_{j\in J})$ is a strongly collapsible decorated group over $\pi_1(\AO,v_1)$,  then there exists a tame decorated morphism  $((\mathbb{B}, u_1), \phi, \{p_j\}_{1\leqslant j\leqslant n}, \{\gamma_j\}_{1\leqslant j\leqslant n})$ over $\AO$  such that
$$(G,\eta, \{G_j\}_{1\leqslant j\leqslant n })\cong (\pi_1(\mathbb{B},u_1), \phi_{\ast}, \{H(p_j,\gamma_j)\}_{1\leqslant j\leqslant n}).$$ 
\end{lemma}
\begin{proof}
After re-indexing we can assume that $J=\{1,\ldots, n\}$.  By hypothesis  $(G, \eta, \{G_j\}_{1\leqslant j\leqslant n})$ is strongly collapsible. Thus   there is a subgroup $G_0\leq G$, which intersects $\ker(\eta)$ trivially,   such that  $G=G_0\ast G_{1}\ast \ldots \ast G_{n}$. Let $g_j$   be the  generator of the peripheral subgroup $G_j\leq G$ that is mapped  by $\eta$ onto the element $o_{G_j} [{c}_{{i_{G_j}}}]^{z_j}   o_{G_j}^{-1}$ for some positive integer $z_j$. 

Consider the decorated morphism depicted in Fig.~\ref{fig:wedge}, where $\phi|_{\B_0}$ is folded and corresponds to the subgroup $\eta(G_0)\cong G_0$ of $\pi_1(\A, v_0)$, and where for each $1\leqslant j\leqslant n$, $[\phi(\gamma_j)]=o_{G_j}$ and $\phi$ maps the path $p_j$ onto the path $ c_{i_{G_j}}^{z_j}$.

 Note that $\phi:\mathbb{B} \rightarrow \A^{\mathcal{O}}$ is   well-defined up to auxiliary moves of type A0 but this causes no harm  as we are only interested in the homomorphism induced by $\phi$.    
 It is not hard to see that $((\B, u_1), \phi, \{p_j\}_{1\leqslant j\leqslant n}, \{\gamma_j\}_{1\leqslant j\leqslant n})$ is tame and that the induced decorated group is isomorphic to $(G, \eta, \{G_j\}_{1\leqslant j\leqslant n})$.
\begin{figure}[h!]
\begin{center}
\includegraphics[scale=0.9]{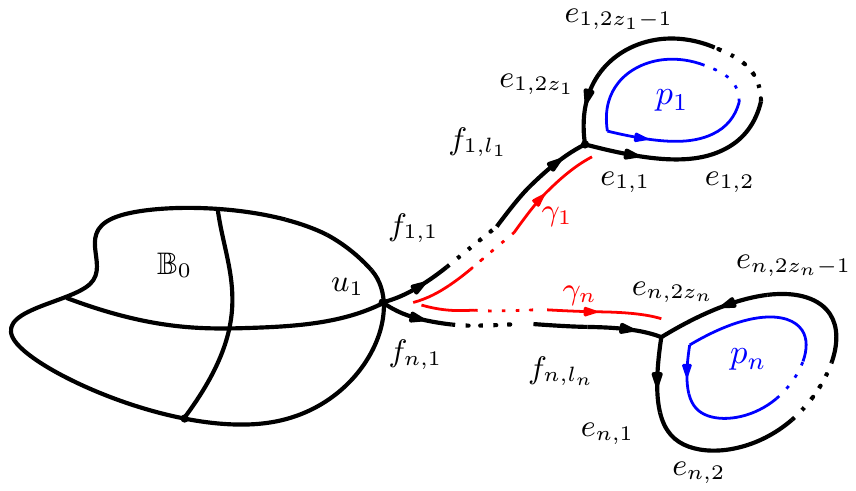}
\end{center}
\caption{$\B$ is obtained from $\B_0$ by gluing $n$ lolipos at the base vertex $u_1$.}{\label{fig:wedge}}
\end{figure}
\end{proof}

In remark~\ref{remark:6.9} we pointed out    that, if we apply an auxiliary move or a vertex morphism to a decorated morphism that does not fold squares, then the resulting decorated morphism does not fold squares.  On the other hand, folds of type IA and IIIA clearly do not have this property as they change the structure of the underlying graph. Next result tells us what happens if the resulting decorated morphism  self-folds. 
\begin{lemma}{\label{lemma:foldsquares}}
Let $((\B, u_1), \phi, \{p_j\}_{1\leqslant j\leqslant n}, \{\gamma_j\}_{1\leqslant j\leqslant n})$ be a tame decorated morphism over $\AO$. Suppose that  it  folds onto  a decorated morphism that  self-folds. Then $(\pi_1(\mathbb{B}, u_1), \phi_{\ast}, \{H(p_j, \gamma_j)\}_{1\leqslant j\leqslant n}) $ projects onto a decorated group   over $\pi_1(\AO, v_1)$  that has an obvious relation.  
\end{lemma}
 \begin{proof}
By hypothesis $((\B, u_1), \phi, \{p_j\}_{1\leqslant j\leqslant n}, \{\gamma_j\}_{1\leqslant j\leqslant n})$ folds onto a decorated morphism over $\AO$ that self-folds. Thus,  the following hold:  (a)  there are distinct edges $g$ and $h$ in $\B$  with $\phi(g)=\phi(h)$ and $u:=\alpha(f)=\alpha(h)$ such that     $ o_g^{\phi}= \phi_u(b) o_h^{\phi }$ for some $b\in B_u$,     and  (b)  there is an index    $j'\in \{1,\ldots,n\}$  such that the path $p_{j'}$ decomposes as $ p_{j'}' (1,g,1)  p_{j'}'' (1,h,1)  p_{j'}'''.$ 

It   follows from Lemma~\ref{lemma:collapsible} that there are edges  $f_1,\ldots, f_n$ of  $\mathbb{B}$ and a bijection  $\nu: \{1,\ldots , n\}\rightarrow \{1,\ldots , n\} $ such that (1) $\{f_1, \ldots, f_n\}\cap \{f_1^{-1}, \ldots, f_n^{-1}\}=\emptyset$ and (2)  for each $1\leqslant k\leqslant n$, the path      $p_{\nu(k)}$ decomposes as  
$$p_{\nu(k)}=p_{\nu(k)}' (1,f_k,1) p_{\nu(k)}''$$   
with $p_{\nu(k)}'$ and $p_{\nu(k)}''$ contained in  $\B_{f_{1},\ldots, f_k}\subseteq \B$.   To simplify the proof assume that $\nu$ is the identidy map. 
\begin{figure}[h!]
\begin{center}
\includegraphics[scale=1]{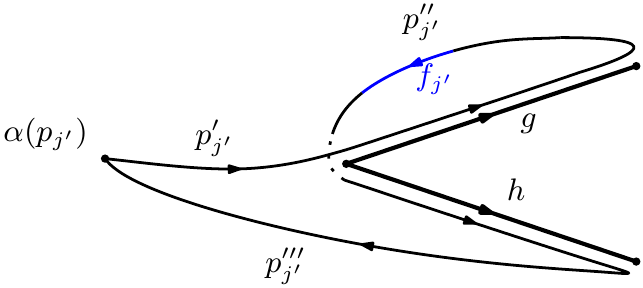}
\end{center}
\caption{The edges $g$ and $h$ are folded onto the same edge.}{\label{fig:lemma1}}
\end{figure}

After a cyclic permutation of   $p_{j'}$,  we can assume that $\phi(p_{j'})=a_{j'} c_{i(j')}^{z_{j'}}a_{j'}^{-1}$  for some  element $a_{j'}$ of  $A_{v_1}^{\mathcal{O}}$ and some positive integer $z_{j'}$ and that the edge $f_{j'}$ lies in the  path $(1,g,1)  p_{j'}''$ as shown  in Fig.~\ref{fig:lemma1}. Now let $((\B, u_1), \phi, \{p_j\}_{1\leqslant j\leqslant n}, \{\delta_j\}_{1\leqslant j\leqslant n})$ be a redecoration  of $((\B, u_1), \phi, \{p_j\}_{1\leqslant j\leqslant n}, \{\gamma_j\}_{1\leqslant j\leqslant n})$ with $\delta_1, \ldots, \delta_n$ contained in $\B_{f_1, \ldots, f_n}$. Lemma~\ref{lemma:collapsible} implies that  
$$\pi_1(\mathbb{B}_{f_1,\ldots, f_{j'}}, u_1)=\pi_1(\mathbb{B}_{f_1,\ldots,   f_{n}}, u_1)\ast H(p_{j'+1}, \delta_{j'+1})\ast \ldots \ast H(p_{n}, \delta_{n})  $$ 
and that  $$\pi_1(\mathbb{B}, u_1)=\pi_1(\B_{f_1,\ldots, f_{j'}},u_1)\ast H(p_{1}, \delta_{1})\ast \ldots \ast H(p_{j'}, \delta_{j'}).$$

Let $\overline{H}$ denote the  subgroup 
$$\pi_1(\mathbb{B}_0,u_1) \ast_{{j\in \{1,\ldots, n\}-\{j'\}}} H(p_{j}, \delta_{j})$$  
of $\pi_1(\B, u_1)$. Observe  that the  path  $p_{j'}'(1,h,1)p_{j'}'''$ is contained in $\mathbb{B}_{f_1,\ldots, f_{j'}}$  and that  its image under $\phi$ is equal to  $a_{j'} c_{i(j')}^{z'} a_{j'}^{-1}$ 
for some $0<z'<z_{j'}$. Hence   the non-trivial  element $o_{j'} [c_{i(j')}]^{z'} o_{j'}^{-1}$  belongs to   
$$o_{j'}C_{i(j')}o_{j'}^{-1}   \cap \phi_{\ast} \pi_1(\mathbb{B}_{f_1,\ldots, f_{j'}},u_1)$$
 which is clearly a subgroup of 
$o_{j'}C_{i(j')}o_{j'}^{-1}   \cap \phi_{\ast} (\overline{H})$. Therefore 
\begin{equation}{\label{eq:lemma}}
 0<|o_{j'} C_{i(j')}o_{j'}^{-1} : \phi_{\ast}(\overline{H})\cap o_{j'} C_{i(j')}o_{j'}^{-1}|<|o_{j'} C_{i(j')}o_{j'}^{-1} : \phi_{\ast}(H(p_{j'}, \delta_{j'}))|
\end{equation}  
where $o_{j'}=o_{H(p_{j'},\delta_{j'})}$.

Now define $H:= H_0 \ast_{1\leqslant j\leqslant n} H(p_j, \delta_j)$  where  $H_0:= \phi_{\ast}\pi_1( \mathbb{B}_0, u_1)\leq \pi_1(\AO, v_1)$.   Let further $\lambda:H\rightarrow \pi_1(\AO, v_1)$ be the homomorphism  induced by the inclusion map $H_0\hookrightarrow \pi_1(\AO, v_1)$ and  the homomorphisms
$$H(p_j, \delta_j) \hookrightarrow \pi_1(\B, u_1) \xrightarrow{\phi_{\ast}} \pi_1(\AO, v_1) \ \ \text{ for } \  1\leqslant j\leqslant n.$$  
It is not hard to see that  
$$(\pi_1(\B, u_0), \phi_{\ast}, \{H(p_j, \delta_j)\}_{1\leqslant j\leqslant n}) \twoheadrightarrow (H, \lambda, \{H(p_j, \delta_j)\}_{1\leqslant j\leqslant n}),$$  which implies that   
$$(\pi_1(\B, u_0), \phi_{\ast}, \{H(p_j, \gamma_j)\}) \twoheadrightarrow (H, \lambda, \{H(p_j, \delta_j)\}_{1\leqslant j\leqslant n}).$$ 
Eq.~(\ref{eq:lemma}) implies   that   $(H, \lambda, \{H_j(p_j, \delta_j)\}_{1\leqslant j\leqslant n})$ has an obvious relation.  
\end{proof}

With an entirely analogous  argument we show the following result.
\begin{lemma}{\label{cor:foldsperipheral}}
 If  the decorated morphism $((\B, u_1), \phi, \{p_j\}_{1\leqslant j\leqslant n}, \{\gamma_j\}_{1\leqslant j\leqslant n})$ over $\AO$ folds onto a decorated morphism that  folds peripheral paths,   then the decorated group $(\pi_1(\mathbb{B}, u_1), \phi_{\ast}, \{H(p_j, \gamma_j)\}_{1\leqslant j\leqslant n})$  projects onto a  decorated group  over $\pi_1(\AO, v_1)$  that folds peripheral subgroups.
\end{lemma}


\subsection{The local graph.}  Recall that a \emph{directed graph} $\Gamma$ is a pair $(V\Gamma, E\Gamma)$ consisting of a vertex set $V\Gamma$ and an edge set  $E\Gamma \subseteq V\Gamma\times V\Gamma$. An edge   $(v,w)$  of $\Gamma$ is denoted by   $v\mapsto w$. 
 
By a  \emph{circle} we mean a directed graph isomorphic to  $\mathcal{C}_n=(\{1,\ldots, n\}, \{(1,2), \ldots, (n-1, n), (n, 1)\})$  for some $n\geqslant 1$. An \emph{interval} is a directed graph isomorphic  to  $\mathcal{I}_n=(\{1,\ldots, n\}, \{(1,2), \ldots, (n-1, n) \})$   for some $n\geqslant 1$. Note that the interval   $\mathcal{I}_1=(\{1\}, \emptyset)$  is degenerate as is consists of a single vertex.

\medskip

Throughout this subsection  let $((\B, u_1), \phi, \{p_j\}_{1\leqslant j\leqslant n}, \{\gamma_j\}_{1\leqslant j\leqslant n})$  be a decorated morphism over $\AO$ and $u$ a vertex of $\B$. We assume further that  $((\B, u_1), \phi, \{p_j\}_{1\leqslant j\leqslant n}, \{\gamma_j\}_{1\leqslant j\leqslant n})$   does not fold  squares.  

\begin{definition}
The \emph{local graph  of $((\B, u_1), \phi, \{p_j\}_{1\leqslant j\leqslant n}, \{\gamma_j\}_{1\leqslant j\leqslant n})$ at the vertex $u$} is defined as the directed graph $\G^u$ having vertex set 
$$V\Gamma_{((\mathbb{B}, u_1), \phi, \{p_j\}, \{\gamma_j\})}^{u}:= \text{St}(u,B)=\{f\in EB \ | \ \alpha(f)=u\},$$ 
and having  edge set $E\G^u$ consisting  of all pairs of vertices  $(f, g)$ for which     there is some $j\in \{1, \ldots, n\}$ and a cyclic permutation $q_j$ of $p_j$  such that $$q_j= q_j'   (1,f^{-1}, b, g, 1)  q_j''$$ for some $b\in B_u.$  The     label of  the edge $f\mapsto g$ is defined by $l(f\mapsto g):= (j, b)\in \{1, \ldots, n\}\times B_u$. 
\end{definition}

Note that $l(f\mapsto g)$ is well defined as we assumed that    $((\B, u_1), \phi, \{p_j\}_{1\leqslant j\leqslant n}, \{\gamma_j\}_{1\leqslant j\leqslant n})$ does not fold squares. Furthermore,     there is at most one vertex $h_1$ of $\G^u$ with $f\mapsto h_1$, and  there is at most one  vertex $h_2$ of $\G^x$ with $h_2\mapsto f$. This  implies that the components of $\G^u$ are either  (degenerate) intervals or  circles.

\begin{example}
Consider the decorated morphism $((\B, u_1), \phi, \{p_1, p_2\}, \{\gamma_1, \gamma_2\})$ from Example~\ref{ex:decmorph}. Then the local graph $\Gamma_{((\B, u_1), \phi, \{p_1, p_2\}, \{\gamma_1, \gamma_2\})}^{u_1}$ has vertex set $\{f_1, f_2, f_3\}$  and edge set $\{f_3\mapsto f_2, f_2\mapsto f_1\}$. The label of $f_3\mapsto f_2$ is $(2, 1) \in \{1,2\}\times B_{u_1}$ and the label of $f_2\mapsto f_1$ is $(1,1)\in \{1,2\}\times B_{u_1}$, see Fig.~\ref{fig:decmorlocalgraph}.
\begin{figure}[h!]
\begin{center}
\includegraphics[scale=1]{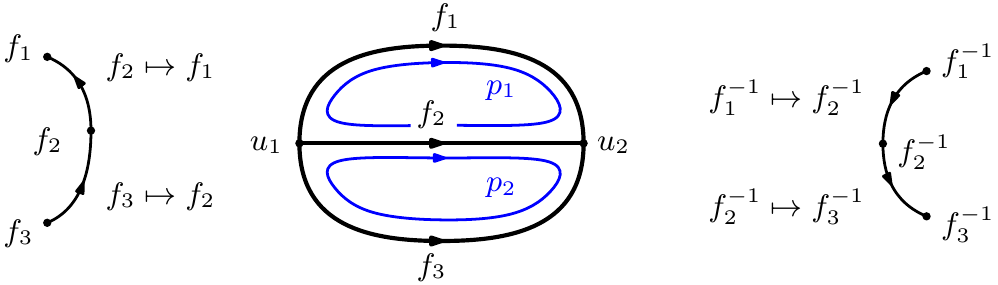}
\end{center}
\caption{$\Gamma_{((\B, u_1), \phi, \{p_1, p_2\}, \{\gamma_1, \gamma_2\})}^{u_1}$  on the left and $\Gamma_{((\B, u_1), \phi, \{p_1, p_2\}, \{\gamma_1, \gamma_2\})}^{u_2}$  on the right.}{\label{fig:decmorlocalgraph}}
\end{figure}
\end{example}

Now we collect  some facts  that local graph gives about the paths $p_1, \ldots, p_n$ and about the morphism $\phi:\B\rightarrow \AO$.    Suppose that $f\mapsto g$ is an edge of $\G^u$ with label $(j, b)$. By definition, there is a cyclic permutation $q_j$ of  $p_j$, an element $a_j\in A_{v_1}^{\mathcal{O}}$,   an index $i(j)\in \{1,\ldots, q_{\mathcal{O}}\}$,  and  a positive integer $z_j$  such that $\phi(q_j)=a_j c_{i(j)}^{z_j} a_j^{-1}$.  Recall that the $\AO$-path $c_i$ ($1\leqslant i\leqslant q_{\mathcal{O}}$) is given by
\begin{equation}{\label{eq:pathci}}
c_i = \begin{cases}
          s_{v_1}, e_1, s_{v_2}, e_2^{-1}, 1              & \text{ if  }  \ i=1  \\
          1,e_i, 1, e_{i+1}^{-1}, 1               & \text{ if  }  \ 2\leqslant  i\leqslant  q_{\mathcal{O}}-1 \\ 
          1, e_{i_{\mathcal{O}}}, 1, e_1^{-1}, 1 & \text{ if } \ i=i_{\mathcal{O}}
        \end{cases}
\end{equation}
Thus
\begin{equation}{\label{path:permutation}}
\phi(q_j)=a_j(s_{v_1}^{\varepsilon_{i(j)}}, e_{i(j)}, s_{v_2}^{\varepsilon_{i(j)}}, e_{i(j)+1}^{-1}, 1)^{z_j} a_j^{-1}\end{equation}
where $\varepsilon_{i(j)}$ is given by
\begin{equation}{\label{eq:epsilon}}
\varepsilon_{i(j)} = \begin{cases}
          1              & \text{ if  }  \ i(j)=1  \\ 
             0               & \text{ if  }  \ i(j)\geqslant 2. 
        \end{cases}
\end{equation} 
Equations~(\ref{eq:pathci})-(\ref{eq:epsilon}) imply that $o_f^{-1} \phi_u(b) o_g=s_{v}^{\varepsilon_{i(j)}}$, that  $\phi(f)=e_{i_0}^{\epsilon}$ and that   $\phi(g)=e_{i_0-\epsilon}^{\epsilon}$,    where $i_0$  and $\epsilon\in \{\pm 1\}$ are given by:
\begin{enumerate}
\item[(a)] $i_0=i(j)+1$  and $\epsilon=1$  if $v=v_1$ (i.e $\phi(u)=v_1$).

\item[(b)] $i_0=i(j)$ and $\epsilon=-1$ if  $v=v_2$ (i.e $\phi(u)=v_2$). 
\end{enumerate}

By putting together all these facts  we get the following result: 
\begin{lemma}{\label{lemma:path}}
Let  $f_1\mapsto f_2\mapsto \ldots \mapsto f_l\mapsto f_{l+1}$ be  a path in $\G^u$ of length $l\geqslant 1$. Suppose that  $l(f_t\mapsto f_{t+1})=(j_t, b_t)$  for  $1\leqslant t\leqslant l$.   Then,  for each    $1\leqslant t\leqslant l$, the following hold:  
\begin{enumerate}
\item[(i)] $o_{f_{t+1}}=\phi_u(b_1\cdot \ldots \cdot b_t)^{-1} o_{f_1} s_{v}^{ \varepsilon_{i(j_1)}+\ldots +\varepsilon_{i(j_t)} }$.

\item[(ii)] $\phi(f_1)=e_{i_0}^{\epsilon} \text{ and }  \phi(f_{t+1})=e_{i_0- \epsilon t}^{\epsilon}$.

\item[(iii)] $i(j_t)\equiv i(j_1)+\epsilon (t-1) (\text{mod }q_{\mathcal{O}})$.
\end{enumerate}
\end{lemma}
\begin{proof}
(i) and (ii) follow by induction. To prove (iii),     suppose that $\phi(u)=v_2$ (the case $\phi(x)=v_1$ is analogous).  From item (ii)  we obtain  $\phi(f_t)=e_{i(j_1)+ t-1}^{-1}$.  Applying  the argument from the previous paragraph to the edge $f_t\mapsto f_{t+1}$ we see that   $\phi(f_t)=e_{i(j_t)}^{-1}$. Thus 
$$e_{i(j_1)+(t-1)}=e_{i(j_t)}\in EA^{\mathcal{O}}$$ 
and so  $i(j_t)\equiv i(j_1)+ t-1 (\text{mod }q_{\mathcal{O}})$.    
\end{proof}

The following result shows that the existence of a closed path in the local graph $\G^u$ implies that the morphism $\phi$ is locally surjective at the vertex $u$.   
\begin{lemma}{\label{lemma:path1}}
Suppose that  $f_1\mapsto f_2\mapsto \ldots \mapsto f_l\mapsto f_{l+1}$ is a path in  $\G^u$ and that 
$$l(f_t\mapsto f_{t+1})=(j_t, b_t)\in \{1,\ldots, n\}\times B_u \ \ \text{ for }\ \  1\leqslant t \leqslant l.$$  
\begin{enumerate}
\item If $\phi(f_1)=\phi(f_{l+1})$  (in particular, if  $f_1=f_{l+1}$),  then $q_{\mathcal{O}}$ divides $l$ and 
 $$o_{f_{l+1}}=\phi_u(b_1\cdot \ldots \cdot b_l)^{-1} o_{f_1} s_{v}^{k} \ \ \text{ where } \ \ k:=l/q_{\mathcal{O}}.$$

\item If $f_{l+1}=f_1$,  then  the following hold:
\begin{enumerate}
\item[(2.i)] the index of $\phi_u(B_u)$ in  $ A_v^{\mathcal{O}} $ is less or equal to    $k$.

\item[(2.ii)] for each $e\in \st(v, A^{\mathcal{O}})$ there are exactly $k$ edges $f_{n_1}^e, \ldots f_{n_k}^e$ in $\{f_1, \ldots, f_l\}\cap \phi^{-1}(e)$. Moreover,   
$$A_{v}^{\mathcal{O}}= \phi_u(B_u) o_{f_{n_1}^e} \cup \phi_u(B_u) o_{f_{n_2}^e} \cup \ldots \cup \phi_u(B_u) o_{f_{n_k}^e}.$$
\end{enumerate}
\end{enumerate}
\end{lemma}
\begin{proof}
We start with item (1). Suppose that $\phi(f_{l+1})=\phi(f_{1})$.  From Lemma~\ref{lemma:path}(ii) we know that    
$$e_{i_0}^{\epsilon}=\phi(f_1)=\phi(f_{k+1}) =e_{i_0-\epsilon l}^{\epsilon}.$$ Hence  the edges  $e_{i_0}$ and $e_{i_0-\epsilon l}$   coincide, which implies that  $i_0\equiv i_0-\epsilon l(\text{mod }q_{\mathcal{O}})$. Thus $q_{\mathcal{O}}$ divides $l$.  For the second claim in (1),  note that  Lemma~\ref{lemma:path}(i)  gives 
$$o_{f_{l+1}}=\phi_u(b_1\cdot \ldots\cdot b_l)^{-1}o_{f_1}s_{v}^{\varepsilon_{i(j_1)}+\cdots + \varepsilon_{i(j_l)}}.$$  
By combining  Equation~\ref{eq:epsilon} and Lemma~\ref{lemma:path}(iii) we get   $\varepsilon_{i(j_1)}+\cdots+ \varepsilon_{i(j_{l})}=k$. This completes the  proof of item (1).

Now assume  that $f_{l+1}=f_1$. It follows from item (1) that $q_{\mathcal{O}}$ divides $l$ and that 
$$o_{f_{l+1}}= \phi_u(b_1\cdot \ldots\cdot b_l)^{-1}o_{f_1}s_{v}^{k}.$$   
As $f_1=f_{l+1}$  we conclude that   $o_{f_{l+1}}=o_{f_1}$, and so we have     
$$s_{v}^k=\phi_u(b_1\cdot \ldots\cdot b_l)^{-1}.$$ 
Thus  $s_{v}^k$ belongs to  $\phi_u(B_u)$, and therefore  the index of $\phi_u(B_u)$ in $A_v^{\mathcal{O}}$ is less or equal to  $k$.

To prove the second claim in (2) suppose that $\phi(x)=v_2$ (the case $\phi(x)=v_1$ is handled similarly). After a cyclic permutation of $f_1\mapsto f_2\mapsto \cdots \mapsto f_{l}\mapsto f_{l+1}=f_1$ we can assume that the index  $i(j_t)\in \{1,\ldots, q_{\mathcal{O}}\}$ is such that $i(j_t)\equiv t \ (\text{mod  }q_{\mathcal{O}})$ for all $1\leqslant t\leqslant l$. With this assumption we have 
$$\phi(f_1)=e_1^{-1}, \phi(f_2)=e_2^{-1}\ldots,  \phi(f_{q_{\mathcal{O}}})=e_{q_{\mathcal{O}}}, \phi(f_{q_{\mathcal{O}}+1})=e_1^{-1},\ldots, $$    
and   
\begin{equation}{\label{equ:1}} 
    \varepsilon_{i(j_t)} = \begin{cases}
               1              & \text{ if  }  \ t\equiv 1 \ (\text{mod  }q_{\mathcal{O}}) \\
               0               & \text{ if  }  \ t\not\equiv 1 \ (\text{mod  } q_{\mathcal{O}}) 
            \end{cases}
\end{equation}{\label{eq:epsilon1}}
Consequently,   for each $1\leqslant i\leqslant q_{\mathcal{O}}$, the intersection   $\{f_1,\ldots, f_{l}\} \cap \phi^{-1}(e_{i}^{-1})$ is equal to $\{ f_i, f_{i+q_{\mathcal{O}}}, \ldots, f_{i+(l-1)q_{\mathcal{O}}}\}.$  Lemma~\ref{lemma:path}(i) and Equation~(\ref{eq:epsilon1})    imply  that  
 $$\phi_u(B_u)o_{f_{i+tq_{\mathcal{O}}}} = \phi_u(B_u) o_{f_{i}}  s_{v_2}^{t} \ \ \ \text{ for } \ \  1\leqslant t\leqslant l-1.$$ 
From  item (1) we know that  $|A_{v_2}^{\mathcal{O}}:\phi_u(B_u)|\leqslant k$. Hence, we can write $o_{f_i}=\phi_u(b)s_{v_2}^{m}$ for some $b\in B_u$ and some $0\leqslant m\leqslant k-1$. Thus 
$$ \bigcup_{t=0}^{k-1}\phi_u(B_u)   o_{f_{i+tq_{\mathcal{O}}}}=\bigcup_{t=0}^{k-1} \phi_u(B_u)  o_{f_{i}}  s_{v_2}^{t}=\bigcup_{t=0}^{k-1} \phi_x(B_x)s_{v_2}^{t+m}=   A_{v_2}^{\mathcal{O}},$$
which completes the proof of item (2).
\end{proof}

The next property of the paths  $p_1, \ldots, p_n$  that we obtain  from the local graph is that in many situations we can assume that $p_1, \ldots, p_n$ do not cross some   vertex groups. The precise meaning of the expression ``do not cross a vertex group'' will be given  below in item (2) of Lemma~\ref{lemma:loctrivial}. 

Let $u$ be a vertex of $\mathbb{B}$ and $S$ a subset of $\st(u, B)=\{f\in EB \ |  \ \alpha(f)=u\}$.  We say that a $\mathbb{B}$-path $p$  is \emph{$S$-trivial at $u$} if $\alpha(p)\neq u\neq \omega(p)$  and  $p$ can be written  as    $$p= p_0 (1,f_1^{-1},1, g_1,1)  p_1\cdot \ldots \cdot p_{k-1} (1,f_k^{-1},1, g_k,1)  p_k$$   
where $f_1, g_1, \ldots, f_k, g_k\in S$ and the paths $p_0,\ldots, p_k$ are contained in the sub-graph of groups $\mathbb{B}_{S}\subseteq \mathbb{B}$.
\begin{figure}[h!]
\begin{center}
\includegraphics[scale=01]{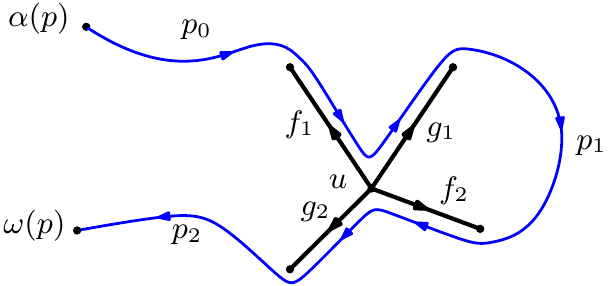}
\end{center}
\caption{$p=p_0 (1,f_1^{-1}, 1, g_1, 1)p_1 (1,f_2^{-1}, 1, g_2, 1) p_2$.}
\end{figure}

Note that a $\B$-path $p$    is   $\st(u, B)$-trivial at $u$ if, and only if,  $\alpha(p)\neq u\neq \omega(p)$ and $p$ is contained in the sub-graph of groups   $\mathbb{B}_u\subseteq \mathbb{B}$ obtained from $\mathbb{B}$  by replacing the vertex group ${B}_u$ by the trivial group.

\begin{lemma}{\label{lemma:loctrivial}}
Let $((\B, u_1), \phi, \{p_j\}_{1\leqslant j\leqslant n}, \{\gamma_j\}_{1\leqslant j\leqslant n})$ be a decorated morphism over $\AO$.  Suppose that  $f_1\mapsto f_2\mapsto \cdots \mapsto f_l$ is a contractible component of $\G^u$ with 
$$l(f_t\mapsto f_{t+1})=(j_t, b_t) \in \{1,\ldots, n\}\times B_u \ \ \text{ for } \ \    1\leqslant t\leqslant l-1.$$ 
Then  there exists a decorated morphism $((\mathbb{B},u_1), \phi', \{p_{j}'\}_{1\leqslant j\leqslant n}, \{\gamma_{j}'\}_{1\leqslant j\leqslant n})$ such that the following hold:
\begin{enumerate}
\item[(1)] $((\mathbb{B},u_1), \phi', \{p_{j}'\}_{1\leqslant j\leqslant n}, \{\gamma_{j}'\}_{1\leqslant j\leqslant n})$ is obtained from $((\B, u_1), \phi, \{p_j\}_{1\leqslant j\leqslant n}, \{\gamma_j\}_{1\leqslant j\leqslant n})$ by a  redecoration and  finitely many auxiliary moves of type A2  applied  to the edges  $f_2,\ldots, f_l$ .

\item[(2)] The paths   $p_{1}', \ldots, p_{n}'$ are $\{f_1,\ldots, f_l\}$-trivial at the vertex  $u$. 
\end{enumerate} 
In particular, the label of the edge  $f_t\mapsto f_{t+1}$    of  $\Gamma_{((\mathbb{B}, u_1), \phi', \{p_j'\}, \{\gamma_j'\})}^u$  is given by
 $$l(f_t\mapsto f_{t+1})=(j_t, 1)\in \{1,\ldots, n\}\times B_u.$$
\end{lemma}
\begin{proof}
After a  cyclic permutation  of the paths $p_1, \ldots, p_n$  we can assume that $\alpha(p_j)\neq u\neq \omega(p_j)$ for all $1\leqslant j\leqslant n$. If $l=1$, then there is nothing to show  as in this case    the paths $p_{1}, \ldots, p_{n}$ are contained in the sub-graph of groups $\mathbb{B}_{f_1}\subseteq \mathbb{B}$, and so they are $\{f_1\}$-trivial at $u$.

In order to clarify the argument we  assume that $l=3$ and that $j_1\neq j_2$. The idea of the proof is to trivialize the paths $p_{j_1},p_{j_2}$ one  by one by applying auxiliary moves to the edges $f_2$ and $f_3$  with  appropriate elements of the vertex group $B_u$.    By definition,  
$$p_{j_1}= p_{j_1}'  (1,f_{1}^{-1},b_1 ,f_{2},1)  p_{j_1}''  \ \ \text{ and } \ \ p_{j_2}=p_{j_2}' (1, f_2^{-1}, b_2, f_3, 1)p_{j_2}''.$$ 
Lemma~\ref{lemma:path} implies that 
 $o_{f_{2}}=\phi_u(b_1)^{-1} o_{f_1}s_v^{\varepsilon_{i(j_1)}}$  and  that $ o_{f_3}= \phi_u(b_1 b_2)^{-1} o_{f_1} s_v^{\varepsilon_{i(j_1)}+\varepsilon_{i(j_2)}}.$  Since the decorated morphism  $((\B, u_1), \phi, \{p_j\}_{1\leqslant j\leqslant n}, \{\gamma_j\}_{1\leqslant j\leqslant n})$ does not fold squares it follows that the paths $p_{j_1}', p_{j_1}'', p_{j_2}', p_{j_2}''$ and the paths  $p_{1}, \ldots, p_{j_1-1}, p_{j_1+1}, \ldots, p_{j_2-1}, p_{j_2+1}, \ldots, p_n$   are contained in the sub-graph of groups $\B_{f_1, f_2, f_3}\subseteq \B$.

Let $((\mathbb{B},u_1), \phi',  \{p_{j}'\}_{1\leqslant j\leqslant n},  \{\gamma_{j}'\}_{1\leqslant j\leqslant n})$ be the decorated morphism  obtained from  $$((\B, u_1), \phi, \{p_j\}_{1\leqslant j\leqslant n}, \{\gamma_j\}_{1\leqslant j\leqslant n})$$
 by an auxiliary move of type A2 applied to the edge $f_2$ with element $b_1\in B_u$. We know from Remark~\ref{remark:auxmove}  that  
$$p_{j_1}' =p_{ j_1}' (1, f_1^{-1}, b_1 b_1^{-1}, f_2, 1) p_{  j_1}''=p_{ j_1}'(1, f_1^{-1}, 1, f_2, 1) p_{  j_1}'',$$ 
that $p_{j_2}'=p_{j_2}'( 1, f_2^{-1}, b_1 b_2, f_3, 1) p_{j_2}''$,
and that  $p_j'=p_j$  for all $j\neq j_1, j_2$.

Next, define  $((\mathbb{B},u_1), \phi'',  \{p_{j}''\}_{1\leqslant j\leqslant n},  \{\gamma_{j}''\}_{1\leqslant j\leqslant n})$  as  the decorated morphism   obtained from $$((\mathbb{B},u_1), \phi',  \{p_{j}'\}_{1\leqslant j\leqslant n},  \{\gamma_{j}'\}_{1\leqslant j\leqslant n})$$ by an auxiliary move of type A2 applied to the edge $f_3$ with element $b_1b_2\in B_u$.  Again, using Remark~\ref{remark:auxmove}, we  see that    
$$p_{j_2}''=p_{j_2}'(1, f_2^{-1}, b_1b_2 (b_1b_2)^{-1}, f_3, 1) p_{j_2}''= p_{j_2}'(1, f_2^{-1}, 1, f_3, 1) p_{j_2}'',$$
and  that $p_j''=p_j'$ for all $j\neq j_2$. Therefore $p_{1}'', \ldots, p_n''$ are $\{f_1, f_2, f_3\}$-locally trivial at $u$.
\end{proof}

The following lemma is simply  a reformulation of the notion of an almost orbifold covering   of $\mathcal{O}$ to the language of decorated morphism over the associated graph of groups $\AO$. 
\begin{lemma}{\label{lemma:almost}}
Let $((\B, u_1), \phi, \{p_j\}_{1\leqslant j\leqslant n}, \{\gamma_j\}_{1\leqslant j\leqslant n})$ be a decorated morphism over $\AO$ that does not fold squares.  Let further $u$ be a vertex of $\mathbb{B}$ with $v:=\phi(u)\in  VA^{\mathcal{O}}$. Suppose that the following hold:
\begin{enumerate}
\item[(a)] For each  $w\in VB$ the local graph $\G^w$ is a circle 
 $f_{w,1}\mapsto f_{w,2}\mapsto \cdots \mapsto f_{w,l_w}\mapsto f_{w,1}$.

\item[(b)] If $w\neq u$, then the morphism  $\phi:\mathbb{B}\rightarrow \AO$ is folded at $w$.

\item[(c)] At the vertex $u$ the following hold:
\begin{enumerate}
\item[(c.1)] $B_u=\langle b_u \ | \ -\rangle \ast \langle s_{v}^d\rangle$ where $1\leqslant d=|A_v^{\mathcal{O}}:\langle s_v^d\rangle| $.

\item[(c.2)] $l(f_{u,t}\mapsto f_{u, t+1})=(j_t, 1)$ for $1\leqslant t\leqslant l_u-1$ and $l(f_{u, l_u}\mapsto f_{u, 1})=( j_{l_u}, b_u)$.

\item[(c.3) ] The vertex homomorphism  $\phi_u: B_u\rightarrow A_{v}^{\mathcal{O}}$ is induced by the  maps $\langle s_v^d\rangle \hookrightarrow A_v$ and   $b_u\mapsto s_{v}^{k_u}$, where $k_u=l_u/q_{\mathcal{O}}$.
\end{enumerate}
\end{enumerate}
 
Then there is an almost orbifold covering $\eta':\mathcal{O}'\rightarrow \mathcal{O}$ with exceptional point $v\in A^{\mathcal{O}}\subseteq F$  with the property that the element represented by the exceptional boundary component of $\mathcal{O}'$ is mapped by $\eta_{\ast}'$ onto $gs_v^{k_u}g^{-1}$ for some $g\in \pi_1^o(\mathcal{O})$ such that 
$$(\pi_1^o(\mathcal{O}')\ast g\langle s_v^d\rangle g^{-1}, \lambda , \{C_1',\ldots, C_n'\}) \cong (\pi_1(\mathbb{B},u_1), \phi_{\ast}, \{H(p_1, \gamma_1), \ldots, H(p_n, \gamma_n)\}),$$
where $\lambda$ is the homomorphism induced by  $\eta_{\ast}':\pi_1^o(\mathcal{O}')\rightarrow \pi_1^o(\mathcal{O})$  and the inclusion map $g \langle s_v^d\rangle g^{-1}\hookrightarrow \pi_1^o(\mathcal{O})$.  
\end{lemma}
\begin{remark}
Item (b) means that the vertex homomorphism  $\phi_w:B_w\rightarrow A_{\phi(w)}^{\mathcal{O}}$ is injective and  that the cosets  $\phi_w(B_w)  o_g^{\phi}$ and  $ \phi_w(B_w)  o_f^{\phi}$  are distinct whenever $f$ and $g$ are edges of $\B$  starting at $w$   with $\phi(f)=\phi(g)$. In item (c.1) note that $B_u=\langle b_u \ |  \ -\rangle $ if $d=|A_v^{\mathcal{O}}|$. 
\end{remark}

\begin{lemma}{\label{lemma:adjfinite}}
With the notation and hypothesis from Lemma~\ref{lemma:almost}, if $1\leqslant d < k_u=l_u/q_{\mathcal{O}}$, then there is a decorated  morphism $((\B', u_0'), \phi', \{p_j'\}_{1\leqslant j\leqslant n}, \{\gamma_j'\}_{1\leqslant j\leqslant n})$ over $\AO$  with  
 $$(\pi_1(\B', u_0'), \phi_{\ast}', \{H(p_j', \gamma_j')\}_{1\leqslant j\leqslant n})\cong (\pi_1(\B, u_0), \phi_{\ast}, \{H(p_j, \gamma_j)\}_{1\leqslant j\leqslant n})$$ 
such that $((\B', u_0'), \phi', \{p_j'\}_{1\leqslant j\leqslant n}, \{\gamma_j'\}_{1\leqslant j\leqslant n})$  folds onto a decorated morphism over $\AO$ that folds squares.
\end{lemma}
\begin{proof}
Let  $((\B', u_1'), \phi', \{p_j'\}_{1\leqslant j\leqslant n}, \{\gamma_j'\}_{1\leqslant j\leqslant n})$ be the decorated morphism    obtained from 
 $$((\B , u_1 ), \phi , \{p_j \}_{1\leqslant j\leqslant n}, \{\gamma_j \}_{1\leqslant j\leqslant n})$$
by unfolding the edge $f_{u,1}$ of $\B$. Thus  $((\B', u_1'), \phi', \{p_j'\}_{1\leqslant j\leqslant n}, \{\gamma_j'\}_{1\leqslant j\leqslant n})$ can be constructed as follows: 
\begin{enumerate}
\item[(1)] we add an edge $f_{u, l_u+1}$  to the graph  $B$ with  $\alpha(f_{u, l_u+1})=u$ and  $\omega(f_{u, l_u+1}):=\omega(f_{u,1})$.

\item[(2)] we extend the graph-morphism $\phi:B\rightarrow A^{\mathcal{O}}$ to a graph-morphism  
$$\phi':B'\rightarrow A^{\mathcal{O}}$$ 
by defining   $\phi'(f_{u, l_u+1}):=\phi(f_{u,1})$.

\item[(3)] we replace the vertex  group $B_u=\langle b_u \ | \ - \rangle \ast \langle s_v^d\rangle$ by the group  $B_u':=\langle s_v^d\rangle$.

\item[(4)] we define the edge elements of $f_{u,l_u+1}$  by   
 $o_{f_{u, l_u+1}}^{\phi'}:= s_{v}^{k_u}$  and  $ t_{f_{u, l_u+1}}^{\phi'}=t_{f_{u, 1}}^{\phi}.$

\item[(5)] the path $p_{j_{l_u}}=p_{j_{l_u}} '(1, f_{u,l_u}^{-1}, b_u, f_{u,1}, 1) p_{j_{l_u}}''$ is replaced by $p_{j_{l_u}}':= p_{j_{l_u}}'(1, f_{u,l_u}^{-1}, 1, f_{u,l_u+1}, 1) p_{j_{l_u}}''$.
\end{enumerate}

Note that $((\B , u_1 ), \phi , \{p_j \}_{1\leqslant j\leqslant n}, \{\gamma_j \}_{1\leqslant j\leqslant n})$  is obtained from  $((\B', u_1'), \phi', \{p_j'\}_{1\leqslant j\leqslant n}, \{\gamma_j'\}_{1\leqslant j\leqslant n})$ by  a fold of type IIIA applied to the   edges $f_{u,1}^{-1}$ and $f_{u, l_u+1}^{-1}$.  This implies that 
$$(\pi_1(\B', u_1'), \phi_{\ast}', \{H(p_j', \gamma_j')\}_{1\leqslant j\leqslant n})\cong (\pi_1(\B , u_1 ), \phi_{\ast} , \{H(p_j, \gamma_j )\}_{1\leqslant j\leqslant n}).$$

We will show that  $((\B', u_1'), \phi', \{p_j'\}_{1\leqslant j\leqslant n}, \{\gamma_j'\}_{1\leqslant j\leqslant n})$ folds onto a decorated morphism that folds squares. In  the local graph $\Gamma_{((\B', u_1'), \phi', \{p_j'\}, \{\gamma_j'\})}^u$, we have 
 $$l(f_{u, t}\mapsto f_{u, t+1})=(j_t, 1) \ \ \text{ for } \ \   1\leqslant t\leqslant l_u.$$ 
 Since $\phi'(f_{u,1})=\phi'(f_{u, l_u+1})=\phi(f_{u,1})$,  it follows  from Lemma~\ref{lemma:path1} that  
$$o_{f_{u, l_u+1}}^{\phi'}=o_{f_{u,1}}^{\phi'} s_v^{k_u}.$$    
By hypothesis $d=|A_v^{\mathcal{O}}: \phi_u'(B_u')|$ is strictly smaller than $k_u= l_u/q_{\mathcal{O}}$.   Lemma~\ref{lemma:path} implies that there is some $i<l_u+1$ such that the edges $f_{u,1}$ and $f_{u,i}$ can be folded. Now observe  that the  edge $f_{u,1}$ is crossed by the path $p_{u,j_1}$, the edge   $f_{u,i}$ is crossed by the path $p_{j_{i-1}}'$,  and the edge $f_{u,i}^{-1}$ is crossed by the path $p_{j_i}$. Hence  the decorated morphism   obtained from $((\B', u_1'), \phi', \{p_j'\}_{1\leqslant j\leqslant n}, \{\gamma_j'\}_{1\leqslant j\leqslant n})$ by folding the edges $f_{u, 1}$ and $f_{u,i}$ folds squares. This concludes the proof of the lemma.
\end{proof}


\subsection{Proof of Proposition~\ref{proposition:1}.} 
Let  $(G, \eta, \{G_j\}_{j\in J})$ be a strongly collapsible decorated group over $\pi_1(\AO, v_1)$, and suppose  that  the homomorphism  $\eta$ is not injective.    Lemma~\ref{lemma:initial} implies that the set $\Lambda_{(G, \eta, \{G_j\})}$ of all tame decorated morphisms $((\B, u_1), \phi, \{p_j\}_{1\leqslant j\leqslant n}, \{\gamma_j\}_{1\leqslant j\leqslant n})$ over $\AO$ that induce decorated groups isomorphic to $(G, \eta, \{G_j\}_{j\in J})$ is non-empty. 
 
 \smallskip

Let  $H:\Lambda_{(G, \eta, \{G_j\})} \rightarrow \mathbb{N}$ be the map defined by $H\D:=|EB| $ and  let  $n_0\in \mathbb{N}$ be the minimum of $H$.  Let 
$$\Lambda_{(G, \eta, \{G_j\})}^{\text{min}}:=H^{-1}(n_0)\subseteq  \Lambda_{(G, \eta, \{G_j\})}$$ 
be the set of  all tame decorated  morphisms  $((\B, u_1), \phi, \{p_j\}_{1\leqslant j\leqslant n}, \{\gamma_j\}_{1\leqslant j\leqslant n})$ in $\Lambda_{(G, \eta, \{G_j\})}$   with $|EB|=n_0$. By hypothesis,  $\eta:G\rightarrow \pi_1(\AO, v_1)$ is not injective. Consequently   no element of  $\Lambda_{(G, \eta, \{G_j\})}$ is    folded.

\smallskip

Let  $((\B, u_1), \phi, \{p_j\}_{1\leqslant j\leqslant n}, \{\gamma_j\}_{1\leqslant j\leqslant n})$  be an arbitrary element element of $\Lambda_{(G, \eta, \{G_j\})}^{\text{min}}$.  Since the morphism  $\phi:\B\rightarrow \AO$ is  vertex injective and the edge groups of $\AO$ are trivial, condition (F1) of Definition~\ref{def:folded}  must be violated.  Thus  either a  fold  of type IA or a  fold of type  IIIA is applicable to $\phi$.   Suppose  that the decorated morphism   $((\mathbb{B}',u_1'), \phi', \{p_j'\}_{1\leqslant j\leqslant n} , \{\gamma_j'\}_{1\leqslant j\leqslant n})$ is   obtained from   $((\B, u_1), \phi, \{p_j\}_{1\leqslant j\leqslant n}, \{\gamma_j\}_{1\leqslant j\leqslant n})$ by  one of these  folds.  Lemma~\ref{lemma:mainfold} implies that  the decorated groups   
 $$(\pi_1(\mathbb{B}', u_1'), \phi_{\ast}', \{H(p_j',\gamma_j')\}_{1\leqslant j\leqslant n}) \ \ \text{ and } \ \  (\pi_1(\mathbb{B}, u_1), \phi_{\ast}, \{H(p_j,\gamma_j)\}_{1\leqslant j\leqslant n} ) $$ 
 are isomorphic.  Since  the number of edges of the underlying graph decreases by two,  we conclude that  $((\mathbb{B}',u_1'), \phi', \{p_j'\}_{1\leqslant j\leqslant n} , \{\gamma_j'\}_{1\leqslant j\leqslant n})$ is not tame.   Thus one of the following must occur: 
\begin{enumerate}
\item[(C1)] $((\mathbb{B}',u_1'), \phi', \{p_j'\}_{1\leqslant j\leqslant n} , \{\gamma_j'\}_{1\leqslant j\leqslant n})$ folds squares.

\item[(C2)] $((\mathbb{B}',u_1'), \phi', \{p_j'\}_{1\leqslant j\leqslant n} , \{\gamma_j'\}_{1\leqslant j\leqslant n})$  is not vertex injective.

\item[(C3)] $((\mathbb{B}',u_1'), \phi', \{p_j'\}_{1\leqslant j\leqslant n} , \{\gamma_j'\}_{1\leqslant j\leqslant n})$  is not collapsible. 
\end{enumerate}

We distinguish two cases according to  whether there is   a decorated morphism in $\Lambda_{(G, \eta, \{G_j\})}^{\text{min}}$ that folds onto a decorated morphism that folds squares or not.

\noindent{\textbf{Case (1)}:} Suppose that there is a decorated morphism  $((\mathbb{B}',u_1'), \phi', \{p_j'\}_{1\leqslant j\leqslant n} , \{\gamma_j'\}_{1\leqslant j\leqslant n})$ in $\Lambda_{(G, \eta, \{G_j\})}^{\text{min}}$  that folds onto a decorated morphism that folds squares.    It follows from Lemmas \ref{lemma:foldsquares} and \ref{cor:foldsperipheral}  that  the decorated group $(\pi_1(\mathbb{B}, u_1), \phi_{\ast}, \{H(p_j, \gamma_j)\}_{1\leqslant j\leqslant n})$ projects onto a decorated group   that  either folds peripheral subgroups or has an obvious relation.   Consequently  the same holds to  $(G, \eta, \{G_j\}_{1\leqslant j\leqslant n})$ as it is isomorphic to $(\pi_1(\mathbb{B}, u_1), \phi_{\ast}, \{H(p_j, \gamma_j)\}_{1\leqslant j\leqslant n})$.

\medskip

\noindent{\textbf{Case (2)}:} Suppose now that Case (1) does not occurs. Thus  any decorated morphism  that is obtained from an  element of $\Lambda_{(G, \eta, \{G_j\})}^{\text{min}}$ by a fold does not fold squares.    We separate this case in two sub-cases according to whether   there exists a decorated morphism in $\Lambda_{(G, \eta, \{G_j\})}$  that folds onto a  (not tame) collapsible decorated morphism or not.

\smallskip

\noindent{\textbf{Case (2.a)}:} Suppose that there exists a decorated morphism   $((\B, u_1), \phi, \{p_j\}_{1\leqslant j\leqslant n}, \{\gamma_j\}_{1\leqslant j\leqslant n})$ in  $\Lambda_{(G, \eta, \{G_j\})}^{\text{min}}$  that folds onto a decorated moprhism  $((\mathbb{B}',u_1'), \phi', \{p_j'\}_{1\leqslant j\leqslant n}, \{\gamma_j'\}_{1\leqslant j\leqslant n})$ that is collapsible.  Thus  $\phi'$ is not vertex injective.  We will show that $(G, \eta, \{G_j\}_{j\in J})$  projects onto a  strongly collapsible decorated group $(H, \lambda ,\{ H_j\}_{1\leqslant j\leqslant n})^{\text{min}}$  over $\pi_1^o(\mathcal{O})$  such that either 
$$\rk(H)< \rk(G) \ \ \text{ or } \ \ \rk(H)=\rk(G) \  \text{ and } \ \ \text{tn}(H)> \text{tn}(G).$$

Indeed, as $\phi'$ is not vertex injective, there is a vertex $w$ of $\mathbb{B}'$  such that the vertex homomorphism  $\phi_w':B_w'\rightarrow A_{\phi'(w)}^{\mathcal{O}}$ is not injective. As  $((\B, u_1), \phi, \{p_j\}_{1\leqslant j\leqslant n}, \{\gamma_j\}_{1\leqslant j\leqslant n})$  is vertex injective and the vertex groups of $\AO$ are finite cyclic,  one of the following must holds:
\begin{enumerate}
\item In the case of a fold of type IA,  the vertex group $B_w'$ is a free product of two finite cyclic groups. Hence $\text{rk}(B_w')=\text{tn}(B_w')=2$. 

\item In the case of a fold of type IIIA,  the vertex $B_w'$ is the free product of an infinite cyclic group with a (possibly trivial)  finite cyclic group. Hence $1\leqslant \text{rk}(B_w)\leqslant 2$ and $ \text{tn}(B_w')=\text{rk}(B_w')-1$. 
\end{enumerate}

Lemma~\ref{lemma:collapsible} provides a sub-graph of groups $\mathbb{B}_0\subseteq \mathbb{B}$ that contains the vertex $u_1$ of $\mathbb{B}$  and a redecoration $((\mathbb{B}, u_1), \phi, \{p_j\}_{1\leqslant j\leqslant n},  \{\delta_j\}_{1\leqslant j\leqslant n})$  of $((\B, u_1), \phi, \{p_j\}_{1\leqslant j\leqslant n}, \{\gamma_j\}_{1\leqslant j\leqslant n})$  such that  $\pi_1(\mathbb{B}, u_1)$ splits as 
$$\pi_1(\mathbb{B}, u_1)=\pi_1(\mathbb{B}_0,u_1)\ast H(p_{1}, \delta_{1})\ast \ldots \ast H(p_{n}, \delta_{n}).$$

Define $H:= H_0 \ast_{j=1}^n H(p_j, \delta_j)$  where  $H_0:= \phi_{\ast}\pi_1( \mathbb{B}_0, u_1)\leq \pi_1(\AO, v_1)$.   Let further $\lambda:H\rightarrow \pi_1(\AO, v_1)$ be the homomorphism  induced by the inclusion map $H_0\hookrightarrow \pi_1(\AO, v_1)$ and  the maps 
$$H(p_j, \delta_j) \hookrightarrow \pi_1(\mathbb{B}, u_1) \xrightarrow{\phi_{\ast}}  \pi_1(\AO, v_1) \ \ \text{ for } \  1\leqslant j\leqslant n.$$  

It is not hard to see that   the decorated  group  induced by $((\B, u_1), \phi, \{p_j\}_{1\leqslant j\leqslant n}, \{\gamma_j\}_{1\leqslant j\leqslant n})$  projects onto the decorated group $(H, \lambda, \{H_j(p_j, \delta_j)\}_{1\leqslant j\leqslant n})$, and therefore 
$$(G, \eta, \{G_j\}_{j\in J})\twoheadrightarrow   (H, \lambda, \{H_j(p_j, \delta_j)\}_{1\leqslant j\leqslant n}).$$ 
If the vertex  group $B_w'$ is the free product of two non-trivial cyclic groups, then   
$$\text{rk}(H)<\text{rk}(\pi_1(\mathbb{B}',u_1'))=\text{rk}(G).$$  
If $B_w'$ is infinite cyclic, then  
$$\text{rk}(H)=\text{rk}(\pi_1(\mathbb{B}',u_1'))=\text{rk}(G) \ \ \text{ and } \ \  \text{tn}(G)=\text{tn}(\pi_1(\mathbb{B}',u_1'))<\text{tn}(H).$$

\medskip

\noindent{\textbf{Case (2.b)}:} Finally assume that (2.a) does not occur. Thus    any decorated morphism  that is obtained from an element of $\Lambda_{(G, \eta, \{G_j\})}^{\text{min}}$  by a fold    is not collapsible (by our assumption that Case (1) does not occur we also know  that it does not fold squares).    In this case we will construct a decorated group over $\pi_1(\AO, v_1)$ that  is isomorphic to the decorated group obtained from a special almost orbfold covering by adjoining a finite  subgroup.

\smallskip
 
Let $((\B, u_1), \phi, \{p_j\}_{1\leqslant j\leqslant n}, \{\gamma_j\}_{1\leqslant j\leqslant n})$  be an arbitrary element of $ \Lambda_{(G, \eta, \{G_j\})}^{\text{min}}$ and suppose that the decorated morphism  $((\mathbb{B}',u_1'), \phi', \{p_j'\}_{1\leqslant j\leqslant n}, \{\gamma_j'\}_{1\leqslant j\leqslant n})$ is obtained from $((\B, u_1), \phi, \{p_j\}_{1\leqslant j\leqslant n}, \{\gamma_j\}_{1\leqslant j\leqslant n})$ by a fold. By our assumptions we know that   
$((\mathbb{B}',u_1'), \phi', \{p_j'\}_{1\leqslant j\leqslant n}, \{\gamma_j'\}_{1\leqslant j\leqslant n})$ 
does not fold squares and it is not collapsible.

 \smallskip

It is not hard to see that   a fold of type IA preserves collapsibility. Thus $((\mathbb{B}',u_1'), \phi', \{p_j'\}_{1\leqslant j\leqslant n}, \{\gamma_j'\}_{1\leqslant j\leqslant n})$ is obtained from $((\B, u_1), \phi, \{p_j\}_{1\leqslant j\leqslant n}, \{\gamma_j\}_{1\leqslant j\leqslant n})$ by a fold of type IIIA.  After an auxilairy move of type A2  we can assume that the fold is elementary. This means that there are  distinct  edges $f$ and $g$ of $\mathbb{B}$   such that (1)  $w:=\alpha(f)=\alpha(g) $ and $u:=\omega(f)=\omega(g)$,   $e:=\phi(f)=\phi(g)$, and   $o_g^{\phi} =o_f^{\phi},$  where $v$ denotes the vertex $\phi(u)$. 
\begin{figure}[h!]
\begin{center}
\includegraphics[scale=1]{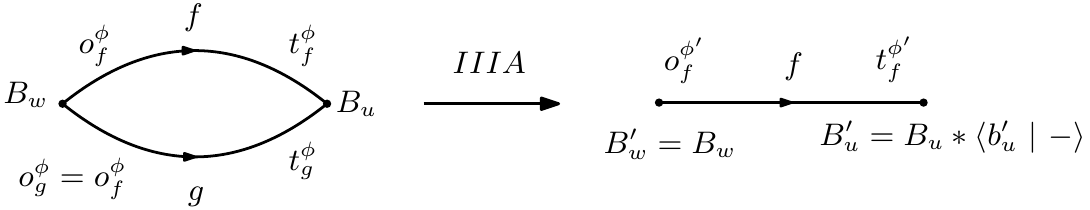}
\end{center}
\caption{$\phi_u'(b_u')=(t_f^{\phi})^{-1} t_g^{\phi}\in A_{\phi(u)}$,  $o_{f}^{\phi'}=o_f^{\phi}$ and $t_{f}^{\phi'}=t_f^{\phi}$.}{\label{fig:foldingIIIA}}
\end{figure}

The underlying graph $B'$ of $\mathbb{B}'$  is equal to  $B/[f=g]$.  In order to simplify the  notation, we will identify    $B'$  with the sub-graph $B-\{g,g^{-1}\}$ of $B$. As the morphism  $\phi:\B\rightarrow \AO$ is vertex injective,   we can  further assume that  $B_u= \langle s_{v}^d \rangle$ with $d= |A_v^{\mathcal{O}}:\langle s_v^d \rangle|\geqslant 1$ and that the vertex homomorphism $\phi_u$ is the inclusion map $\langle s_v^d \rangle \hookrightarrow A_v^{\mathcal{O}}$.   With this simplifications, we see that   $\phi':\mathbb{B}'\rightarrow \AO$ is given by
 $$\phi'=(\phi', \{\phi_y' \ | \ y\in VB'\}, \{\phi_h'\ | \ h\in EB'\}, \{o_h^{\phi'} \ | \ h\in EB'\}, \{t_h^{\phi'} \ | \ h\in EB'\}) $$
 where  (i) the graph-morphism  $\phi':B'\rightarrow A^{\mathcal{O}}$ is the restriction of $\phi$ to   $B'=B-\{g, g^{-1}\}$, (ii)  for all vertices $y\neq u$ the vertex homomorphism  $\phi_y':B_y'\rightarrow A_{\phi'(y)}$ coincides with $\phi_y$, (iii) the vertex homomorphism  $\phi_u':B_u'=B_u\ast \langle b_u'\ | \ -\rangle \rightarrow A_{v}^{\mathcal{O}}$ is induced by the maps 
$$B_u\xrightarrow{\phi_u} A_{v}^{\mathcal{O}} \ \ \text{ and } \ \ b_{u}' \mapsto (t_f^{\phi})^{-1}  t_g^{\phi} ,$$
and (iv) the edge elements are give by $o_{h}^{\phi'}=o_{h}^{\phi}$  for all  $h\in EB'=EB-\{g,g^{-1}\}$.  Fig. \ref{fig:foldingIIIA} shows the portion of the underlying graph of $\mathbb{B}$ involved in the fold.

Since $((\mathbb{B}',u_1'), \phi', \{p_{j}'\}_{1\leqslant j\leqslant n}, \{\gamma_{j}'\}_{1\leqslant j\leqslant n})$ is not collapsible and   does not fold squares,  there is a non-empty subset $K=\{k_1,\ldots, k_{m}\}\subseteq \{1,\ldots, n\}$ such that,  for each  $h\in EB'=EB-\{g,g^{-1}\}$,  exactly  one of the following holds:
\begin{enumerate}
\item[(a)] The paths $p_{k_1}', \ldots, p_{k_m}'$ are contained in the sub-graph of groups $\mathbb{B}_h'\subseteq \mathbb{B}'$.

\item[(b)] There are non necessarily distinct elements  $k$ and $ k'$ in $ K$ such that $$p_{k}'=p_{k,1}'(1,h,1)p_{k,2}' \ \ \text{ and } \  \  p_{k'}'=p_{k',2}'(1,h^{-1}, 1)p_{k',2}'$$ 
and that the following hold:  (b.i) the paths  $p_{k,1 }', p_{k,2}'$ and $p_{k',1}', p_{k',2}' $ are contained in $\B_h$, and  (b.ii) for each $k''\in K-\{k,k'\}$ the path  $p_{k''}'$ is  contained in   $\B_h$

\item[(c)] There is an element  $k$ in $ K$ such that 
$$p_k'=p_{k,1}'(1,h,1) p_{k,2}' (1,h^{-1}, 1)p_{k,3}'$$ 
and that the following hold: (c.i) the paths $p_{k,1}', p_{k,2}'$ and $ p_{k,3}'$ are contained in $\B_h$ and  (c.ii) for each $k'\in K-\{k\}$ the path   $p_{k'}'$ is  contained in   $\B_h$.   
\end{enumerate}

Let $B_K'$ be the sub-graph of $B'$ having edge set 
$$EB_K':=\{ h\in EB' \ | \  h \text{ is crossed by some } p_k' \ \text{in } \ \{p_{k_1}', \ldots, p_{k_m}'\}\}$$ 
and vertex set  $VB_K':=\{\alpha(h) \ | \  h\in EB_K'\}.$  The following hold: 
\begin{enumerate}

\item[(i)]  $B_K'$ is connected.

\item[(ii)] The paths $p_{k_1}', \ldots, p_{k_m}'$ are contained in the sub-graph of groups $\mathbb{B}_K'\subseteq \mathbb{B}'$ carried by the sub-graph $B_K'\subseteq B'$.

\item[(iii)] For each   $h\in B_{K}'$  there are   $k, k'\in K$  such that  $$p_{k}'=p_{k,1}'  (1,h,1)  p_{k,2}' \ \ \text{ and } \ \  p_{k'}'=p_{k',1}'  (1,h^{-1},1)  p_{k',2}'$$ 
and that the following hold:  (iii.a) the paths $p_{k,1}', p_{k,2}', p_{k',1}'$ and $p_{k',2}'$ are contained in  $(\B_K')_h$ and (iii.b) for each $k''\in K-\{ k, k'\}$ the path  $p_{k''}'$ is   contained in $(\B_K')_h$.
\end{enumerate}
Recall that  $(\B_K')_h$ denotes  the sub-graph of groups of  $ \mathbb{B}_K'$ carried by the sub-graph $B_K'-\{h, h^{-1}\}$ of $B_K'$.     From item  (iii)  we conclude that  
\begin{enumerate}
\item[(iv)]   for each $y$ of $B_{K}'$, the local  graph  $\Gamma_{((\mathbb{B}', u_1'), \phi', \{p_j'\}, \{\gamma_j'\})}^y$    has a circular component  
$$h_{y,1}\mapsto h_{y,2} \cdots \mapsto h_{y, l_y}\mapsto h_{y,1}$$ 
with  all edges     $h_{y,1}, \ldots, h_{y,l_y}$ contained  in $B_K'$.  
\end{enumerate}

\begin{claim}{\label{claim:0}}
For each vertex $y\in VB'=VB$   the local  graph $\Gamma_{((\mathbb{B}', u_1'), \phi', \{p_j'\}, \{\gamma_j'\})}^y$ contains at most one circle.
\end{claim} 
 \begin{proof}[Proof of Claim]
Indeed,     $ \Gamma_{((\mathbb{B}', u_1'), \phi', \{p_j'\}, \{\gamma_j'\})}^y$    can be described in terms of $ \G^y $  as follows: 
\begin{equation*}
    \Gamma_{((\mathbb{B}', u_1'), \phi', \{p_j'\}, \{\gamma_j'\})}^y = \begin{cases}
               \G^y                     & \text{ if  }  \ y\neq w,u  \\
               \G^w/[f=g]               & \text{ if  }  \ y=w \\
               \G^u/[f^{-1}=g^{-1}]     & \text{ if  }   \ y=u.     
           \end{cases}
\end{equation*}
From this,   we see that,  if   two components of $\Gamma_{((\mathbb{B}', u_1'), \phi', \{p_j'\}, \{\gamma_j'\})}^y$ are circles, then   $\Gamma_{((\mathbb{B} , u_1 ), \phi , \{p_j \}, \{\gamma_j \})}^y$  contains at least  two non-trivial components, and  one of them  must be a  circle.  Now Lemma~\ref{lemma:path1} implies that  $((\B, u_1), \phi, \{p_j\}_{1\leqslant j\leqslant n}, \{\gamma_j\}_{1\leqslant j\leqslant n})$   folds onto a decorated morphism that folds squares,  which is  a contradiction.   
\end{proof}

\begin{claim}{\label{claim:1}}
 For each vertex $y\in VB_K'$,   the local  graph $\Gamma_{((\mathbb{B}', u_1'), \phi', \{p_j'\}, \{\gamma_j'\})}^y$ is a circle.
 \end{claim} 
\begin{proof}[Proof of Claim]
Denote  the  graph $\Gamma_{((\B', u_1'), \phi', \{p_j'\}, \{\gamma_j'\})}^y$ by  $\Gamma^y$.  We will show that if $\G^y$ is not connected, then  there is a decorated morphism in $\Lambda_{(G, \eta, \{G_j\})}^{\text{min}}$  that either  folds onto a decorated morphism that folds squares (which contradicts our assumption that case (1) does not occur) or onto a morphism  that is collapsible (which contradicts our assumption that case (2.a) does not occur).   

Item  (iv) implies that  at least one component  of $\Gamma^y$  is a circle  $ f_{y,1}\mapsto f_{y,2}\mapsto \cdots \mapsto f_{y,l_y}\mapsto f_{y,1}$ 
with all edges  $f_{y,1}, \ldots, f_{y,l_y}$ in $B_{K}$. If $\Gamma^y$ is not connected, then the previous claim  implies that  there is an interval 
$$h_1\mapsto h_2\mapsto \cdots \mapsto h_{t-1} \mapsto h_t$$
in $\Gamma^y$  with  $h_1, \ldots, h_t\in EB'$.   Lemma~\ref{lemma:path1} implies that   there is some $i\in\{1,\ldots l_y\}$ such that 
$$e':=\phi'(f_{y,i})=\phi'(h_1) \ \ \text{ and  that } \ \ o_{h_1}^{\phi'}=ao_{f_{y,i}}^{\phi'}$$ 
for some element   $a$ in $\phi_y'(B_y')$.

Define the morphism $\phi'':\mathbb{B}\rightarrow \AO$ from $\phi:\mathbb{B}\rightarrow \AO$ by replacing, for each $1\leqslant i\leqslant t$,  the edge element  $o_{h_1}^{\phi'}=o_{h_1}^{\phi} $ by   $o_{h_i}^{\phi''}:=a^{-1} o_{h_i}^{\phi}$. Thus $o_{h_1}^{\phi''}=o_{f_{y,i}}^{\phi}=o_{f_{y,i}}^{\phi''}$. Now  let 
 $((\mathbb{B}, u_1), \phi'', \{p_{j}\}_{1\leqslant j\leqslant n}, \{\gamma_{j}\}_{1\leqslant j\leqslant n})$ 
be the decorated morphism   obtained from  $$((\mathbb{B}, u_1), \phi, \{p_{j}\}_{1\leqslant j\leqslant n}, \{\gamma_{j}\}_{1\leqslant j\leqslant n})$$ by replacing  $\phi:\mathbb{B}\rightarrow \AO$  by $\phi'':\mathbb{B}''\rightarrow \AO$. 

It is not hard to see that the decorated morphism  obtained from $((\mathbb{B}, u_1), \phi'', \{p_{j}\}_{1\leqslant j\leqslant n}, \{\gamma_{j}\}_{1\leqslant j\leqslant n})$ by folding the edges $f$ and $g$    is also  obtained from $((\mathbb{B}', u_1'), \phi', \{p_j'\}_{1\leqslant j\leqslant n}, \{\gamma_j'\}_{1\leqslant j\leqslant n})$ by  applying  finitely many  auxiliary moves of type A2  to the edges $h_1,\ldots, h_t$ of $B'$.  It  now follows from  Lemma~\ref{lemma:3}  that  
$$((\mathbb{B}, u_1), \phi'', \{p_{j}\}_{1\leqslant j\leqslant n}, \{\gamma_{j}\}_{1\leqslant j\leqslant n})\in \Lambda_{(G, \eta, \{g_j\})}^{\text{min}}.$$

Finally,  let  $((\mathbb{B}'', u_1''), \phi''',  \{p_j''\}_{1\leqslant j\leqslant n}, \{\gamma_j''\}_{1\leqslant j\leqslant n})$ 
be the decorated morphism obtained from   $$((\mathbb{B}, u_1), \phi'', \{p_{j}\}_{1\leqslant j\leqslant n}, \{\gamma_{j}\}_{1\leqslant j\leqslant n})$$
 by folding the edges $f_{y,i}$ and $h_1$.  If the edge $h_1$ is not crossed by the paths $p_{1}, \ldots, p_{n}$, i.e.  if the interval $h_1\mapsto \cdots \mapsto h_t$ has length $t=1$,   then  clearly   $((\mathbb{B}'', u_1''), \phi''', \{p_j''\}_{1\leqslant j\leqslant n}, \{\gamma_j''\}_{1\leqslant j\leqslant n})$ is collapsible and does  not fold squares. If $h_1$ is crossed by some  path  $p_j$, then   $((\mathbb{B}'', u_1''), \phi''', \{p_j''\}_{1\leqslant j\leqslant n}, \{\gamma_j''\}_{1\leqslant j\leqslant n})$ folds  squares.  
\end{proof}

\begin{claim}{\label{claim:2}}
$B_K'=B'=B-\{g,g^{-1}\}$. Consequently     $K=\{k_1, \ldots, k_m\}$ is equal to $\{1\ldots, n\}$.
\end{claim} 
\begin{proof}[Proof of Claim]
We know from  the previous claim that,   for each vertex $y$  of $B_K'$,  the graph  $\Gamma_{((\mathbb{B}', u_1'), \phi', \{p_j'\}, \{\gamma_j'\})}^y$  is a circle  $f_{y,1}\mapsto f_{y,2} \cdots \mapsto f_{y, l_{y}}\mapsto f_{y,1}$ (with $f_{y,1}, \ldots, f_{y,l_{y}}\in EB_K'$) and therefore   connected. 

If $B_K'$ were a  a proper sub-graph of $B'$,  then there would be an edge $h\in B'-B_K'$ with $\alpha(h)\in B_K'$. This would then imply, as the component  that contains  the vertex $h$  is  clearly distinct from the circle  
$$f_{\alpha(h),1}\mapsto f_{\alpha(h),2} \mapsto\cdots \mapsto f_{\alpha(h), l_{\alpha(h)}}\mapsto f_{\alpha(h),1},  $$
that the graph  $\Gamma_{((\mathbb{B}', u_1'), \phi', \{p_j'\}, \{\gamma_j'\})}^{\alpha(h)}$ is not connected, a contradiction.  

Therefore $B_K'=B'$.  Since the decorated morphism  $((\mathbb{B}', u_1'), \phi', \{p_j'\}_{1\leqslant j\leqslant n},  \{\gamma_j'\}_{1\leqslant j\leqslant n})$ does not fold squares,  it follows that $K=\{1,\ldots, n\}$.
\end{proof}

We will show that  item ($\beta$) of Proposition~\ref{proposition:1} holds. To this end we  show that the decorated morphism $((\B', u_0'), \phi', \{p_j'\}_{1\leqslant j\leqslant n}, \{\gamma_j'\}_{1\leqslant j\leqslant})$ satisfies the hypothesis of Lemma~\ref{lemma:almost}.  We already know that it does not fold squares. 

It follows from Claims~\ref{claim:1} and \ref{claim:2} and from the description of $\Gamma_{((\mathbb{B}', u_1'), \phi', \{p_j'\}, \{\gamma_j'\})}^u$ given in the proof of Claim~\ref{claim:0},   that the graph  $\G^u$   is an interval $$f_{u,1}\mapsto f_{u,2}\mapsto \cdots \mapsto f_{u, l_u}\mapsto f_{u, l_u+1}$$ with $\{f_{u,1}, f_{u, l_u+1}\}=\{f^{-1}, g^{-1}\}$. Without loss of generality   suppose that $f_{u,1}=f^{-1}$ and $f_{u, l_u+1}=g^{-1}$.  It  follows from   Lemma~\ref{lemma:loctrivial}  that,  after a redecoration  and finitely many auxiliary moves of type A2 applied to the edges $f_{u,1},\ldots, f_{u, l_u+1}\in \st(u, B)$,     we can assume    that the paths $p_{1}, \ldots, p_{n}$  are $\st(u,B)$-trivial at $u$. In   other words,    the paths $p_{1}, \ldots, p_{n}$ are contained in the sub-graph of groups $\mathbb{B}_u$ of $\mathbb{B}$.  Thus,  the  label of the edges  of $\G^u$ is  given   by 
$$l(f_{u,t}\mapsto f_{u, t+1})=(j_t, 1)\in \{1,\ldots, n\}\times B_u \ \ \text{ for } \ \ 1\leqslant  t \leqslant l_u.$$
As the IIIA fold is applicable to the edges $f$ and $g$,  we have  
$$\phi(f_{u,1})=\phi(f^{-1})=e^{-1}=\phi(g^{-1})=\phi(f_{u,l_u+1}).$$ 
Item (i) of Lemma~\ref{lemma:path1}  implies  that $q_{\mathcal{O}}$ divides $l_u$ and that 
$$(t_g^{\phi})^{-1}=o_{f_{k+1}}^{\phi}=o_{f_1}^{\phi }s_{v}^{k_u}=(t_f^{\phi})^{-1}s_{v}^{k_u},$$  
where $k_u:=l_u/q_{\mathcal{O}}$.

To conclude the proof of the Proposition  we  show that $((\B', u_1), \phi', \{p_j\}_{1\leqslant j\leqslant n}, \{\gamma_j\}_{1\leqslant j\leqslant n})$ satisfies  conditions (a)-(c)  of Lemma~\ref{lemma:almost}.  From Claim~\ref{claim:1} we know that, for each vertex  $w$ of $B'$, the graph    $\Gamma_{((\mathbb{B}', u_1'), \phi' ,  \{p_j'\}, \{\gamma_j'\})}^w$ is a circle 
$$ f_{w, 1}\mapsto f_{w,2}\mapsto \cdots \mapsto f_{w,l_w}\mapsto f_{w,1}.$$

If $\phi':\B'\rightarrow \A$ is not folded at some vertex $w\neq u$,  then  a  fold  can be applied to distinct edges $h', h''\in  \text{St}(w,B')$.  Since the vertex groups $B_w'$ and $B_w$ coincide,   it follows   that  a fold  that identifies the edges $h'$ and $h''$ can also be applied  to  $\phi:\B\rightarrow \AO$. This implies that  
$$((\mathbb{B}, u_1), \phi, \{p_j\}_{1\leqslant j\leqslant n},  \{\gamma_j\}_{1\leqslant j\leqslant n})$$ 
 folds onto a decorated morphism that  folds squares, a contradiction.  Therefore the morphism $\phi':\B'\rightarrow \AO$ is folded at all vertices of $B'$ distinct from $u$.

It remains to verify that at the vertex $u$  the following hold:
\begin{enumerate}
\item[(c.1)] $B_u'=\langle b_u' | \ -\rangle \ast \langle s_{v}^d\rangle$ where $d=|A_v^{\mathcal{O}}:\langle s_v^d\rangle|$.

\item[(c.2)] $l(f_{u,t}\mapsto f_{u, t+1})=(j_t, 1)$ for $1\leqslant t\leqslant l_u-1$ and $l(f_{u, l_u}\mapsto f_{u, 1})=( j_{l_u}, b_u')$. 

\item[(c.3)] The vertex homomorphism  $\phi_u': B_u'\rightarrow A_{v}^{\mathcal{O}}$ is induced by the  maps $\langle s_v^d\rangle \hookrightarrow A_v$ and   $b_u'\mapsto s_{v}^{k_u}$.
\end{enumerate}

Item  (c.1)  follows immediately from the definition of the fold.  
For item  (c.2), note that, as   the edges $f_{u,1}\mapsto f_{u,2}, \ldots, f_{u,l_u}\mapsto f_{u, l_u+1}$ of $\G^u$ have label 
$$l(f_{u, t}\mapsto f_{u, t+1})=(j_t, 1)\ \ \text{ for } \ \  1\leqslant t\leqslant l_u, $$ 
it follows that  the edges  $f_{u,1}\mapsto f_{u,2}, \ldots, f_{u,l_u}\mapsto f_{u, 1}$ of $\Gamma_{(\B', u_1'), \phi', \{p_j'\}, \{\gamma_j'\}) }^u$  have label 
$$l(f_{u, t}\mapsto f_{u, t+1})=(j_t, 1)  \ \ \text{ for } \ \  1\leqslant t\leqslant l_u-1 \ \ \text{ and } \ \ l(f_{u, l_u}\mapsto f_{u, 1})=(b_u', j_{l_u}).$$

Item (c.3)   follows immediately from Lemma~\ref{lemma:path1}. 
 
\smallskip
 
 Lemma~\ref{lemma:almost} says that  there  is an almost orbifold covering $\eta':\mathcal{O}'\rightarrow \mathcal{O}$ with exceptional point $v\in A^{\mathcal{O}}\subseteq F$  with the property that  the element represented by the exceptional boundary component of $\mathcal{O}'$ is mapped by $\eta_{\ast}'$ onto $gs_v^{k_u}g^{-1}$ for some $g\in \pi_1^o(\mathcal{O})$ such that 
$$(\pi_1^o(\mathcal{O}')\ast g\langle s_v^d\rangle g^{-1}, \lambda , \{C_1',\ldots, C_n'\}) \cong  (\pi_1(\mathbb{B},u_1), \phi_{\ast}, \{H(p_1, \gamma_1), \ldots, H(p_n, \gamma_n)\}),$$
where $\lambda$ is the homomorphism induced by  $\eta_{\ast}':\pi_1^o(\mathcal{O}')\rightarrow \pi_1^o(\mathcal{O})$ and the inclusion map $g \langle s_v^d\rangle g^{-1}\hookrightarrow \pi_1^o(\mathcal{O})$.

It remains to show that $\eta'$ is special. It  follows from  Lemma~\ref{lemma:adjfinite} and from our assumption that no element of $\Lambda_{(G, \eta, \{G_j\})}^{\text{min}}$ folds onto a decorated morphism over $\AO$   that folds squares, that   $d\geqslant k_u$.  Therefore  $k_u\leqslant |A_v^{\mathcal{O}}|$, and so    $\eta':\mathcal{O}'\rightarrow \mathcal{O}$ is a special almost orbifold covering.


\begin{thebibliography}{99}
 
\bibitem{Bass} H.~Bass, \textit{Covering theory for graphs of groups}, Journal of Pure and Applied Algebra \textbf{89} (1993), 3--47.

 \bibitem{BF} M. Bestvina and M. Feighn, \textit{Bounding the complexity of simplicial group actions on trees}, Invent. math. \textbf{103} (1991), 449-469.
  
\bibitem{D} M. J. Dunwoody, {\it Folding sequences}, The Epstein birthday schrift, Geom. Topol. Monogr. 1 (Geom. Topol. Publ., Coventry, 1998) 139-158. 

\bibitem{Grushko}  I. A. Grushko, {\it On the bases of a free product of groups}, Mat. Sb. 8 (1940) 169-182.

\bibitem{KMW} I.~Kapovich, A. Myasnikov and R.~Weidmann,
\emph{Foldings, graphs of groups and the membership problem.} Internat. J. Algebra Comput. \textbf{15} (2005), no. 1, 95--128.


 

\bibitem{Kurosh} A. G. Kurosh, {\it Die Untergruppen der freien Produkte von beliebigen Gruppen},  Math. Ann., vol. 109 (1934), pp. 647–660.
 
\bibitem{Louder} L. Louder, {\it Nielsen equivalence in closed surface groups}, preprint arXiv: 1009.0452v2 [math.GR].  


\bibitem{Lustig} M.~Lustig and Y,~Moriah, {\it Nielsen equivalence in Fuchsian groups and Seifert fibered spaces},  Topology Vol. 30, No. 2, pp191-204, 1991. 

\bibitem{Lustig1} M. Lustig{ \it  Nielsen equivalence and simple homotopy type},  Proc. London Math. Soc. 1991, \textbf{62}, 537-562.
 
 
\bibitem{N1} J. Nielsen, {\it Die Isomorphismen der allgemeinen, unendliche Gruppe mit zwei Erzeugenden}, Math. Ann. \textbf{78} (1917) 385-397.

\bibitem{N2} J. Nielsen, {\it Om Regning med ikke-kommutative Faktorer og dens Anvendelse i Gruppeteorien}, Math. Tidssk. B (1921) 77-94. 

\bibitem{N3} J. Nielsen, {\it Über dir isomorphismen unendlicher Gruppen ohne Relation}, Math. Ann. \textbf{79} (1918) 269-272. 
  

\bibitem{PRZ} N. Peczynski, G. Rosenberger and H. Zieschang, {\it Über Erzeugende ebener diskontinuierlicher Gruppen}, Invent. Math. \textbf{29} (975),  161-180.

 

\bibitem{Rosenberger} G. Rosenberger, {\it Automorphismen ebener diskontinuierlicher Gruppen}, Riemann surfaces andrelated topics, proceedings of the 1978 Stony Brook conference (eds I. Kra and B. Maskit), Annals of Mathematics Studies 97 (Princeton University Press, 1980), pp. 439-455.


\bibitem{Rosenberger1} G. Rosenberger, {\it  Minimal generating systems for plane discontinuous groups and an equation in free groups}, Groups-Korea 1988 (eds A. C. Kim and B. H. Neumann), Lecture Notes in Mathematics 1398 (Springer, Berlin, 1989), pp. 170-186.
 

\bibitem{Scott} G. P. Scott, {\it The geometries of $3$-manifolds}, Bull. London Math. Soc. 15 (1983) 401-487.  
  Proc. London Math. Soc. (3) \textbf{95} (2007), 609-652.
 
 \bibitem{Serre} J.-P. Serre, {\it Trees}, translated from French by John Stillwell (Springer-Verlag, Berlin-New York, 1980). 
 
 
\bibitem{S} J. R. Stallings, {\it Topology of finite graphs}, Invent. Math. \textbf{71} (1983), 551-565.

 \bibitem{S1} J. R. Stallings, {\it Foldings of $G$-trees}, Arboreal Group Theory (Berkley, CA, 1988).

\bibitem{RW} R. Weidmann, \textit{The rank problem for sufficiently large fuchsian groups}, Proc. London Math. Soc. (3) \textbf{95} (2007), 609-652.
 
 
 
\bibitem{Zieschang} H. Zieschang, {\it Über der Nielsensche Kürzungmethode in freien Produkten mit Amalgam}, Inventiones Math. \textbf{10} (1970), 4-37.

\end{thebibliography}
\end{document}